\documentclass[11pt,twoside,reqno]{amsart}

\usepackage{amsmath}
\usepackage{microtype}
\usepackage{comment}
\usepackage[OT1]{fontenc}
\usepackage{type1cm}
\usepackage{amssymb}
\usepackage{enumitem}
\usepackage{xcolor}
\usepackage{mathrsfs}
\usepackage{graphicx}
\usepackage{tikz}
\usepackage{cite}
\usepackage[french,english]{babel}
\usepackage{geometry}
\usepackage{hyperref}
\hypersetup{
	colorlinks=true,
	linkcolor=black,
	anchorcolor=black,
	citecolor=black,
	filecolor=black,
	menucolor=red,
	runcolor=black,
	urlcolor=black,
}

\numberwithin{equation}{section}

\theoremstyle{plain}
\newtheorem{theorem}{Theorem}[section]
\newtheorem*{intuitive-theorem}{Summary of the results}

\newtheorem{corollary}[theorem]{Corollary}
\newtheorem{proposition}[theorem]{Proposition}

\newtheorem{lemma}[theorem]{Lemma}

\theoremstyle{remark}

\newtheorem{example}[theorem]{Example}
\newtheorem*{ack}{Acknowledgements}

\theoremstyle{definition}

\newtheorem{question}[theorem]{Question}
\newtheorem*{vague-question}{Question}

\newcommand{\HH}{\mathcal{H}}
\newcommand{\LL}{\mathcal{L}}

\newcommand{\MM}{\mathcal{M}}
\newcommand{\NN}{\mathcal{N}}

\newcommand{\II}{\mathcal{I}}
\newcommand{\CC}{\mathcal{C}}

\newcommand{\UU}{\mathcal{U}}

\newcommand{\RR}{\mathcal{R}}
\newcommand{\R}{\mathbb{R}}
\newcommand{\RP}{\mathbb{RP}^1}

\newcommand{\Q}{\mathbb{Q}}

\newcommand{\N}{\mathbb{N}}
\newcommand{\hhh}{\mathtt{h}}
\newcommand{\fff}{\mathtt{f}}
\newcommand{\iii}{\mathtt{i}}
\newcommand{\ddd}{\mathtt{d}}
\newcommand{\jjj}{\mathtt{j}}
\newcommand{\kkk}{\mathtt{k}}

\renewcommand{\lll}{\mathtt{l}}
\renewcommand{\ggg}{\mathtt{g}}
\newcommand{\eps}{\varepsilon}
\newcommand{\fii}{\varphi}
\newcommand{\roo}{\varrho}

\newcommand{\A}{\mathsf{A}}

\newcommand{\dd}{\,\mathrm{d}}

\renewcommand{\ge}{\geqslant}
\renewcommand{\le}{\leqslant}
\renewcommand{\geq}{\geqslant}
\renewcommand{\leq}{\leqslant}

\DeclareMathOperator{\udimloc}{\overline{dim}_{loc}}
\DeclareMathOperator{\ldimloc}{\underline{dim}_{loc}}

\DeclareMathOperator{\dimh}{dim_H}

\DeclareMathOperator{\dima}{dim_A}
\DeclareMathOperator{\diml}{dim_L}
\DeclareMathOperator{\dimsim}{dim_{sim}}
\DeclareMathOperator{\dimaff}{dim_{aff}}

\DeclareMathOperator*{\esssup}{ess\,sup}

\DeclareMathOperator{\Tan}{Tan}

\DeclareMathOperator{\linspan}{span}
\DeclareMathOperator{\dist}{dist}
\DeclareMathOperator{\diam}{diam}
\DeclareMathOperator{\diag}{diag}
\DeclareMathOperator{\proj}{proj}
\DeclareMathOperator{\conv}{conv}

\DeclareMathOperator{\spt}{spt}
\DeclareMathOperator{\rank}{rank}

\renewcommand{\atop}[2]{\genfrac{}{}{0pt}{}{#1}{#2}}

\DeclareMathOperator{\GL}{GL}

\DeclareMathOperator{\im}{im}

\makeatletter

\begin{document}

\title{Finer geometry of planar self-affine sets}

\author{Bal\'azs B\'ar\'any}
\address[Bal\'azs B\'ar\'any]
        {Department of Stochastics \\
         Institute of Mathematics \\
         Budapest University of Technology and Economics \\
         M\H{u}egyetem rkp. 3 \\
         H-1111 Budapest,
         Hungary}
\email{barany.balazs@ttk.bme.hu}

\author{Antti K\"aenm\"aki}
\address[Antti K\"aenm\"aki]
        {Department of Physics and Mathematics \\
         University of Eastern Finland \\
         P.O.\ Box 111 \\
         FI-80101 Joensuu \\
         Finland}
\email{antti@kaenmaki.net}

\author{Han Yu}
\address[Han Yu]
        {College of Mathematics and Statistics, Center of Mathematics \\
        Chongqing University \\
        Chongqing, 401331 \\
        China}
\email{han.yu.2@cqu.edu.cn}

\date{\today}

\begin{abstract}
	For planar self-affine sets satisfying the strong separation condition, recent work of B\'ar\'any, Hochman, and Rapaport \cite{BHR} gives mild assumptions under which the Hausdorff dimension equals the affinity dimension. In this paper, we study dominated systems in that regime and ask which finer geometric properties can be characterized. In the range $\dimh(X) < 1$, we characterize Ahlfors regularity by equivalent conditions involving positivity of $\HH^s(X)$, control of projection fibers, and the identity $\diml(X)=\dimh(X)=\dima(X)$. In the range $\dimh(X) \ge 1$, we identify the maximal slice dimension as $\dima(X)-1$ in Furstenberg directions and provide examples showing that Marstrand-type all-slice bounds cannot hold in general. We also derive projection consequences for Assouad dimension and exhibit dominated irreducible examples with $\dimaff(X)<\dima(X)$.
\end{abstract}

\maketitle

\tableofcontents

\section{Introduction} \label{sec:intro}

\subsection{Context}

Let $X \subset \R^2$ be a self-affine set associated to a finite number of invertible affine contractions $\fii_i$ on $\R^2$. The defining property of $X$ is that it consists of affine copies $\fii_i(X)$ of itself. The strong separation condition requires the sets $\fii_i(X)$ to be pairwise disjoint. We write $\fii_i(x) = A_ix+v_i$ for all $x \in \R^2$, where $A_i \in \GL_2(\R)$ is the linear part and $v_i \in \R^2$ is the translation vector. Understanding geometric properties of $X$ is difficult even when the matrix entries are positive and the strong separation condition holds. The topic has attracted substantial attention in recent years; see e.g.\ B\'ar\'any \cite{Barany2015}, B\'ar\'any, K\"aenm\"aki, and Koivusalo \cite{BaranyKaenmakiKoivusalo2017}, Hueter and Lalley \cite{HueterLalley1995}, K\"aenm\"aki and Shmerkin \cite{KaenmakiShmerkin2009}, Morris and Shmerkin \cite{MorrisShmerkin2019}, and Rapaport \cite{Rapaport18} for dimension results, and B\'ar\'any and K\"aenm\"aki \cite{BaranyKaenmaki2017} and Falconer and Kempton \cite{FalconerKempton2017} for projection results.

A recent breakthrough is the article of B\'ar\'any, Hochman, and Rapaport \cite{BHR} where the authors proved that under the strong separation condition and mild assumptions on the linear parts, the Hausdorff dimension of $X$ equals the affinity dimension, a natural upper bound stemming from the definition of the self-affine set. The corresponding result for self-similar sets, a sub-class of self-affine sets where the linear parts are assumed to be constant times orthogonal matrices, was proved by Hutchinson \cite{Hutchinson1981}. He also showed that the strong separation condition implies the positivity of the Hausdorff measure. A folklore open question, explicitly stated in Falconer \cite[\S 2]{Falconer1992}, asks whether there exists a characterization or even sufficient conditions for the positivity and finiteness of the Hausdorff measure of self-affine sets.

\begin{vague-question}
  Is it possible to characterize the positivity and finiteness of the Hausdorff measure of a self-affine set?
\end{vague-question}

B\'ar\'any, Hochman, and Rapaport \cite{BHR} also improved what the classical Marstrand projection theorem gives for planar self-affine sets by showing that the Hausdorff dimension of the projection is either preserved for all directions or is equal to one. A complementary concept to projections is that of slices. The classical Marstrand slicing theorem gives that almost every slice has Hausdorff dimension at most the surplus dimension of the projection. Existing results therefore identify the almost-everywhere behavior, but they do not determine whether slicing bounds can hold for every slice in this non-carpet regime.

\begin{vague-question}
  Is the Hausdorff dimension of every slice of a self-affine set at most the surplus dimension of the projection?
\end{vague-question}

\subsection{Main contributions}

In this article, we study these questions, and related questions on finer geometry, for a large class of self-affine sets. A planar self-affine set $X$ is dominated and irreducible if the linear parts of the defining affine maps of $X$ have positive entries and they do not share a common invariant line. Some results below hold under weaker assumptions, but to keep the presentation readable we state the introductory summary in this dominated irreducible setting. The $s$-dimensional Hausdorff measure is denoted by $\HH^s$. Ahlfors $s$-regularity of $X$ means that the $\HH^s|_X$-measure of any ball of radius $r$ centered at $X$ is uniformly comparable to $r^s$. The lower and Assouad dimensions are $\diml$ and $\dima$, respectively, and the Hausdorff dimension is $\dimh$.

\begin{intuitive-theorem}
  Consider the family of dominated irreducible planar self-affine sets $X$ satisfying the strong separation condition. In the regime $\dimh(X)<1$,
  \begin{enumerate}
    \item there exist examples that are not Ahlfors regular, as well as examples that are,
    \item Ahlfors regularity of $X$ is equivalent to $\HH^s(X)>0$, where $s=\dimh(X)$, and to $\diml(X)=\dimh(X)=\dima(X)$, and also to very strong control of the multiplicity of preimages in orthogonal projections,
    \item if $X$ is not Ahlfors regular, then $\dimh(X)<1\le\dima(X)$ and, except for Hausdorff dimension zero set of directions, every orthogonal projection of $X$ has Assouad dimension one.
  \end{enumerate}
  In the regime $\dimh(X) \ge 1$,
  \begin{enumerate}[resume]
    \item Ahlfors regularity does not hold if $\dimh(X)>1$, because necessarily $\diml(X)=1$,
    \item $\max\dimh(X\cap(V+x))=\dima(X)-1<1$, where the maximum is over all $x \in X$ and all top Oseledets directions $V$ for the inverse linear action,
    \item there exist examples in which Marstrand's slicing theorem cannot be extended to all slices.
  \end{enumerate}
\end{intuitive-theorem}

The picture that emerges differs from what is known for self-similar sets and Bedford-McMullen carpets. For self-similar sets with strong separation all dimensions agree and the Hausdorff measure is always positive and finite. For Bedford-McMullen carpets, which are self-affine but defined by families of reducible linear maps, Hausdorff measure is generally infinite, and the dimensions do not agree.

We use a wide variety of methods to obtain the results. One important input is the recent work of B\'ar\'any, Hochman, and Rapaport \cite{BHR}, giving that every rank one linear image of $X$ has Hausdorff dimension $\min\{1,\dimh(X)\}$. A new observation is that Ahlfors regularity is equivalent to having very tight control over the fibers of projections. Another ingredient used repeatedly is that the behavior of weak tangent sets leads to information on the original set because weak tangent sets can be re-inserted to $X$ by linear maps which are often rank one. We also use the transfer operator to study the Hausdorff content of projections.

Let us now review some of the consequences of our main results. The first corollary, which essentially follows from the fact that the lower dimension is bounded above by one, completely characterizes the Ahlfors regularity.

\begin{corollary} \label{thm:ahlfors-equiv}
  If $X$ is a dominated irreducible planar self-affine set satisfying the strong separation condition, then $X$ is Ahlfors $s$-regular if and only if $0 \le s \le 1$ and $\HH^s(X)>0$.
\end{corollary}

We denote the collection of all lines through the origin by $\RP$. The classical Marstrand's projection theorem \cite{Marstrand1954} for Hausdorff dimension states that, given a Borel set $X \subset \R^2$, we have
\begin{equation*}
  \dimh(\proj_{V}(X)) = \min\{1,\dimh(X)\}
\end{equation*}
for Lebesgue almost all $V \in \RP$, where $\proj_V \colon \R^2 \to V$ is the orthogonal projection onto $V$. For a general class of self-affine sets, B\'ar\'any, Hochman, and Rapaport \cite{BHR} have extended the above result for all $V \in \RP$.

The Assouad dimension $\dima$ is the maximal dimension possible to obtain by looking at coverings and it serves as an upper bound for the Hausdorff dimension. Orponen \cite{Orponen2021} has shown a strong variant of Marstrand’s projection theorem for Assouad dimension. It states that, given a set $X \subset \R^2$, we have
\begin{equation} \label{eq:orponen}
  \dima(\proj_{V}(X)) \ge \min\{1,\dima(X)\}
\end{equation}
for all $V \in \RP \setminus E$, where the set $E \subset \RP$ satisfies $\dimh(E)=0$. It is worth pointing out that in general, besides proving the exceptional set $E$ countable, this result cannot be improved: Fraser and K\"aenm\"aki \cite{FraserKaenmaki2020} showed that for every upper semi-continuous function $f \colon \RP \to [0,1]$ there exists a compact set $X \subset \R^2$ with $\dima(X)=0$ such that $\dima(\proj_V(X)) = f(V)$ for all $V \in \RP$. The following corollary can therefore be considered as a manifestation of the rigid structure of self-affine sets. The set $X_F \subset \RP$ is the collection of all Furstenberg directions, i.e.\ the closure of the set of top Oseledets directions for the inverse linear action of the affine maps, and under the assumptions of the corollary, it satisfies $\dimh(X_F)>0$.

\begin{corollary} \label{thm:assouad-proj}
  If $X$ is a dominated irreducible planar self-affine set satisfying the strong separation condition such that $X$ is not Ahlfors regular, then
  \begin{equation*}
    \dima(\proj_{V^\bot}(X)) = \min\{1,\dima(X)\}
  \end{equation*}
  for all $V \in \RP \setminus E$, where the set $E \subset \RP$ satisfies $\dimh(E)=0$. If $X$ is Ahlfors regular, then the above holds for all $V \in X_F$.
\end{corollary}

Marstrand's projection theorem is a strong dimension conservation principle for Hausdorff dimension: if $\dimh(X) \le 1$, then the dimension of $X$ is conserved by almost every projection. If $\dim(X) > 1$, then the same cannot be true, as the projections have dimension at most one in every direction. This defect is resolved by the classical Marstrand's slicing theorem \cite{Marstrand1954}. It shows that almost every fiber of a projection do not store more dimension than what is the surplus. The theorem states that, given a Borel set $X \subset \R^2$ and $V \in \RP$, we have
\begin{equation*}
  \dimh(X \cap (V+x)) \le \max\{0,\dimh(X) - 1\}
\end{equation*}
for Lebesgue almost all $x \in V^\bot$. Furstenberg \cite{Furstenberg1970} conjectured that for the product of $\times 2$ and $\times 3$ invariant sets all fibers should be small. In our terminology, such sets appear as certain dynamically defined subsets of product-type Bedford-McMullen carpets. Furstenberg's conjecture was resolved by Shmerkin \cite{Shmerkin2019} and Wu \cite{Wu2019}. It is therefore interesting to ask whether the slices are small also on other self-affine sets. For a Bedford-McMullen carpet $X$ having logarithmically incommensurable contraction ratios, Algom \cite{Algom} proved that the Hausdorff dimension of any slice not parallel to the principal axes is bounded above by $\max\{0,\dima(X)-1\}$. This bound was recently improved by Algom and Wu \cite{AlgomWu-preprint}. Besides the following corollary, we are not aware of results of this type for general classes of non-carpet self-affine sets.

\begin{corollary} \label{thm:slices}
  If $X$ is a dominated irreducible planar self-affine set satisfying the strong separation condition such that $\dimh(X) \ge 1$, then
  \begin{equation*}
    \max_{\atop{x\in X}{V\in X_F}} \dimh(X \cap (V+x)) = \dima(X) - 1 < 1.
  \end{equation*}
\end{corollary}

In Example \ref{ex:dimaff-dima}, we exhibit a dominated irreducible planar self-affine set $X$ satisfying the strong separation condition such that $1 \le \dimh(X) < \dima(X)$. As the upper bound is attained in Corollary \ref{thm:slices}, we therefore see that it is not possible to extend the Marstrand slicing theorem to all slices of all self-affine sets.

The Hausdorff dimension of any self-affine set is bounded above also by the affinity dimension $\dimaff$, a number which is obtained by looking at the behavior of natural covers in the construction of the set. Prior to Example \ref{ex:dimaff-dima} and the following corollary, it was not known how the Assouad dimension compares to the affinity dimension outside self-affine carpets. It is worth recalling that on self-similar sets the strong separation condition implies the equality of the Assouad and affinity dimensions.

\begin{corollary} \label{thm:affinity-assouad}
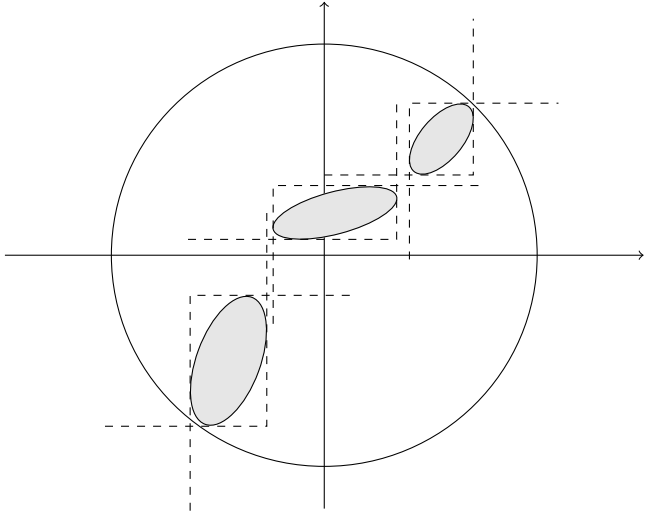
  If $X$ is a dominated irreducible planar self-affine set satisfying the strong separation condition such that $\dimh(X) = s < 1$, then either $X$ is Ahlfors regular or $\dimaff(X) < 1 \le \dima(X)$. In particular, parametrized by the translation vectors and elements of the matrices, there exists a collection with non-empty interior of such planar Ahlfors $s$-regular self-affine sets and there exists an uncountable collection of such planar self-affine sets $X$ for which $\dimaff(X) < 1 \le \dima(X)$.
\end{corollary}

\subsection{Open problems}

We finish the introduction by formulating precise open problems motivated by the results above. The first question concerns positivity and finiteness of Hausdorff measure. If $X$ is a dominated irreducible planar self-affine set satisfying the strong separation condition and $s = \dimh(X) \le 1$, then Theorem \ref{thm:ahl} gives several characterizations for $\HH^s(X)>0$. If $s = \dimh(X)>1$, then Lemma \ref{thm:irreducible-dominated}, Theorem \ref{thm:diml}, and Lemma \ref{thm:ahlfors-implications} imply that $X$ is not Ahlfors regular. As Theorem \ref{thm:BHR} and Lemma \ref{thm:FK} show that $\HH^s(X)<\infty$, the following question is natural:

\begin{question}
  Let $X$ be a strictly affine strongly irreducible self-affine set satisfying the strong separation condition such that $s=\dimh(X)>1$. Is it possible to have $\HH^s(X)>0$?
\end{question}

Recall that if a Bedford-McMullen carpet is not Ahlfors regular, then, by Theorem \ref{thm:BM-carpet-regular}, it has infinite Hausdorff measure in the dimension. We remark that Przytycki and Urba\'nski \cite[Section~6 and Remark 13]{PrzUrb} have shown that the reducible planar self-affine set $X$ associated to $((x,y) \mapsto (x/2,\lambda y+1),(x,y) \mapsto ((x+1)/2,\lambda y-1))$ where $\lambda>1/2$ has $\HH^s(X)>0$ for $s=\dimh(X)>1$ if the canonical projection of the equidistributed Bernoulli measure onto the orthogonal complement of the sole element in $X_F$ is absolutely continuous and has $L^\infty$-density.

The second question asks whether all slices of self-affine sets are small. As already mentioned, Example \ref{ex:dimaff-dima} together with Corollary \ref{thm:slices} shows that one cannot hope to obtain the same upper bound as in the Marstrand slicing theorem for all slices. Since the example strongly relies on the rigid structure of Bedford-McMullen carpets, a class that is exceptional in several ways, the following question is natural:

\begin{question} \label{q:slices}
  For a ``typical'' strictly affine strongly irreducible self-affine set $X$ satisfying the strong separation condition, is it true that
  \begin{equation*}
    \sup_{\atop{x\in X}{V\in \RP}}\dimh(X\cap(V+x))\leq\max\{0,\dimh(X)-1\}?
  \end{equation*}
  In particular, is it true that $\dima(X)=\dimh(X)$ for a ``typical'' self-affine set $X$ with $\dimh(X)\geq1$?
\end{question}

We speculate that addressing the above question requires further development of the scenery-flow theory introduced by Furstenberg \cite{Furstenberg2008} for self-affine sets, possibly combined with methods introduced by Hochman and Rapaport \cite{HochmanRapaport2021}. The machinery developed here provides a starting point for this direction.

Garc\'ia \cite[Theorem 1.4]{Garcia2020} proved that if a planar self-similar set $X$ not satisfying the weak separation condition is not contained in a line, then $\dima(X) > 1$. Theorem \ref{thm:dima2} shows that if $X$ is a dominated planar self-affine set satisfying the strong separation condition such that $\HH^s(X)=0$ for $s = \dimh(X)$ and $X_F$ is not a singleton, then $\dima(X) \ge 1$. Note that if $X_F$ is not a singleton, then $X$ is not contained in a line. Therefore, the following question is natural:

\begin{question}\label{q:Ass}
  Let $X$ be a dominated planar self-affine set satisfying the strong separation condition such that $\HH^s(X)=0$ for $s = \dimh(X)$ and $X_F$ is not a singleton and $\dimh(X) \le 1$. Is $\dima(X) > 1$?
\end{question}

If for ``typical'' self-affine sets it is possible to obtain a positive answer in Question \ref{q:Ass}, then \cite[Proposition 2.4 and Theorem 5.2]{BaranyKaenmakiRossi2021} imply a negative answer to Question \ref{q:slices}.

The paper is organized as follows. In Section \ref{sec:main-results} we state the main theorems and derive the corollaries announced above. In Section \ref{sec:lower} we prove the lower dimension formula, and in Section \ref{sec:assouad-large} we prove the Assouad dimension formula in the regime $\dimh(X) \ge 1$. In Sections \ref{sec:assouad-small} and \ref{sec:hausdorff-measure} we prove the projective-separation criteria and complete the Ahlfors-regularity characterization. Finally, Section \ref{sec:assouad-affinity} proves the Baire-typical parameter statement.

\section{Preliminaries}

We introduce rigorous definitions and some preliminaries in this section. The reader familiar with the recent progress in the topic may skip the preliminaries and go directly to Section \ref{sec:main-results} where we exhibit the main results.

\subsection{Ahlfors regularity} \label{sec:A-regular}
We say that a Borel measure $\mu$ on $\R^2$ is \emph{Ahlfors $s$-regular} if there exists a constant $C \ge 1$ such that
\begin{equation} \label{eq:ahlfors-regularity-def}
	C^{-1}r^s \le \mu(B(x,r)) \le Cr^s
\end{equation}
for all $x \in \spt(\mu)$ and $0<r<\diam(\spt(\mu))$, where $B(x,r)$ is the closed ball centered at $x$ with radius $r$. A compact set $X \subset \R^2$ is \emph{Ahlfors $s$-regular} if it has a fully supported Ahlfors $s$-regular measure. We also say that a measure or a set is \emph{Ahlfors regular} if it is Ahlfors $s$-regular for some $s \ge 0$. Recall that the \emph{$s$-dimensional Hausdorff measure} $\HH^s$ of a set $X \subset \R^2$ is defined by
\begin{equation*}
  \HH^s(X) = \lim_{\delta \downarrow 0} \HH^s_\delta(X) = \sup_{\delta>0} \HH^s_\delta(X),
\end{equation*}
where
\begin{equation*}
  \HH^s_\delta(X) = \inf\biggl\{ \sum_i \diam(U_i)^s : X \subset \bigcup_i U_i \text{ and } \diam(U_i) \le \delta \biggr\}
\end{equation*}
is the \emph{$s$-dimensional Hausdorff $\delta$-content} of $X$. By \cite[Theorem 6.9]{Mattila1995}, an Ahlfors regular set $X \subset \R^2$ has positive and finite $s$-dimensional Hausdorff measure, $0<\HH^s(X)<\infty$, when $s = \dimh(X) = \inf\{s \ge 0 : \HH^s(X)<\infty\}$ is the Hausdorff dimension of $X$. In fact, a set $X \subset \R^2$ is Ahlfors $s$-regular if and only if there exists a constant $C \ge 1$ such that
\begin{equation*}
	C^{-1}r^s \le \HH^s(X \cap B(x,r)) \le Cr^s
\end{equation*}
for all $x \in X$ and $0<r<\diam(X)$.

If a Borel measure $\mu$ on $\R^2$ is Ahlfors $s$-regular, then we write $\dim(\mu)=s$ and say that $s$ is the \emph{dimension} of $\mu$. More generally, the \emph{upper and lower pointwise dimensions} of $\mu$ at $x \in \R^2$ are
\begin{align*}
\udimloc(\mu,x) &= \limsup_{r \downarrow 0} \frac{\log \mu(B(x,r))}{\log r}, \\
\ldimloc(\mu,x) &= \liminf_{r \downarrow 0} \frac{\log \mu(B(x,r))}{\log r},
\end{align*}
respectively. For basic properties of pointwise dimensions, we refer to the book of Falconer \cite[\S 10]{Falconer1997}. If there exists a constant $s$ such that $\udimloc(\mu,x) = \ldimloc(\mu,x) = s$ for $\mu$-almost all $x \in \R^2$, then we again write $\dim(\mu) = s$ and say that $\mu$ is \emph{exact dimensional}. We remark that in general, most measures do not satisfy this property. But since in the study of self-affine sets, basically all the measures involved are exact dimensional, we use the convention that writing $\dim(\mu)$ implicitly means that the measure $\mu$ is known to be exact dimensional; consult \cite{FengHu2009,BaranyKaenmaki2017,Feng2019-preprint,HochmanSolomyak2017,Rapaport2021,LessaLedrappier-preprint} for examples. Dimensions of measures provide a way to study the Hausdorff dimension of a given Borel set $X$: it is well known that
\begin{equation} \label{eq:variational-principle}
  \dimh(X) = \max\{\esssup_{x \sim \mu}\ldimloc(\mu,x) : \mu \text{ is a finite Borel measure on } X\};
\end{equation}
for example, see \cite[\S 3]{FalconerFraserKaenmaki2020-preprint}. Often with dynamically defined sets, such as in the case of self-affine sets, it is desirable for the approximating measures to adhere to some dynamical properties such as ergodicity and the behavior of entropy and Lyapunov exponents. In this case, the above formula is called the \emph{variational principle}; see \cite{Kaenmaki2004} for its proof in the self-affine case. Under further assumptions, the approximating measures can be seen to agree also with other properties; see \cite[Theorem 1.1]{MorrisShmerkin2019} and \cite[Proposition 2.4]{BaranyKaenmakiRossi2021}. These properties were crucial in \cite{BHR} where the authors were able to determine the Hausdorff dimension of planar self-affine sets in a very general setting; see Theorem \ref{thm:BHR}.

\subsection{Weak tangents}
Let $X \subset \R^2$ be compact. For each $x\in X$ and $r>0$ we define the \emph{magnification} $M_{x,r} \colon \R^2 \to \R^2$ by setting
\begin{equation*}
  M_{x,r}(z) = \frac{z-x}{r}
\end{equation*}
for all $z \in \R^2$. We say that a compact subset $T$ of $B(0,1)$ is a \emph{weak tangent set} of $X$ if there exist sequences $(x_n)_{n\in\N}$ of points in $X$ and $(r_n)_{n \in \N}$ of positive real numbers such that $\lim_{n \to \infty} r_n = 0$ and $M_{x_n,r_n}(X) \cap B(0,1) \to T$ in Hausdorff distance. We denote the collection of weak tangent sets of $X$ by $\Tan(X)$. Note that $0 \in T$ for all $T \in \Tan(X)$.

The \emph{Assouad dimension} of a set $X \subset \R^2$, denoted by $\dima(X)$, is the infimum of all $s \ge 0$ satisfying the following: There exists a constant $C \ge 1$ such that for every $x \in X$ and $0<r<R$ the set $X \cap B(x,R)$ can be covered by at most $C(R/r)^s$ balls of radius $r$ centered at $X$. It is easy to see that
\begin{equation*}
  \dimh(X) \le \dima(X)
\end{equation*}
for all sets $X \subset \R^2$. If $X \subset \R^2$ is compact, then $\dima(X) \ge \dimh(T)$ for all $T \in \Tan(X)$; see \cite[Proposition 6.1.5]{MackayTyson}. The following result of K\"aenm\"aki, Ojala, and Rossi \cite[Proposition 5.7]{KOR} shows that there exists a weak tangent set whose Hausdorff dimension attains the maximal possible value. The result introduces a way to calculate the Assouad dimension of a set by considering its weak tangents.

\begin{lemma} \label{thm:KOR}
  If $X \subset \R^2$ is compact, then $\dima(X) = \max\{\dimh(T) : T \in \Tan(X)\}$.
\end{lemma}

We remark that the first result in this direction is by Furstenberg; see \cite[Proposition 5.1]{Furstenberg2008}. Together with \cite[Proposition 3.13]{KaenmakiRossi2016} it shows that the Assouad dimension is realized as a Hausdorff dimension of a weak tangent set or a finite magnification. The above result is needed to guarantee that the Assouad dimension gets realized on a weak tangent. This is particularly important detail in the study of self-affine sets as such sets often undergo a metamorphosis in approaching the weak tangent; see \cite{BandtKaenmaki2013,KaenmakiKoivusaloRossi2017,KOR,BaranyKaenmakiRossi2021}.

Analogously, the \emph{lower dimension} of a set $X \subset \R^2$, denoted by $\diml(X)$, is the supremum of all $s \ge 0$ satisfying the following: There exists a constant $c>0$ such that for every $x \in X$ and $0<r<R<\diam(X)$ covering the set $X \cap B(x,R)$ requires at least $c(r/R)^{-t}$ balls of radius $r$ centered at $X$. It is easy to see that
\begin{equation*}
  \diml(X) \le \dimh(X)
\end{equation*}
for all $X \subset \R^2$. If $X \subset \R^2$ is compact, then $\diml(X) \le \dimh(T)$ for all $T \in \Tan(X)$; see \cite[Proposition 2.3]{FHKY2019}. The following result of Fraser, Howroyd, K\"aenm\"aki, and Yu \cite[Theorem 1.1]{KOR} shows that there exists a weak tangent set whose Hausdorff dimension attains the minimal possible value. The result introduces a way to calculate the lower dimension of a set by considering its weak tangents.

\begin{lemma} \label{thm:FHKY}
  If $X \subset \R^2$ is compact, then $\diml(X) = \min\{\dimh(T) : T \in \Tan(X)\}$.
\end{lemma}

It is straightforward to see that if $X \subset \R^2$ is Ahlfors $s$-regular, then $\diml(X)=\dima(X)=s$; see \cite[\S 3]{KaenmakiLehrbackVuorinen2013}. The converse does not necessarily hold as the following example illustrates.

\begin{example}
  We construct a set $X \subset [0,1]$ with $\diml(X)=\dima(X)$ and $\HH^s(X)=0$ for $s = \dimh(X)$. Let $\lambda_n=(3n)^{-1}$ for all $n \in \N$. We start the construction from the unit interval $[0,1]$. First, we cut out the middle part of length $\lambda_1$. Next, for each of the two small intervals of length $(1-\lambda_1)/2 = 1/3$, we cut out their middle parts of length $\lambda_2/3$. At step $n$, we cut out the middle parts of relative length $\lambda_n$. We can perform this cutting procedure indefinitely and in the end we obtain a compact set $X$ similarly as with the middle third Cantor set. It is straightforward to see that $\diml(X)=1$ and $X$ has zero Lebesgue measure.
\end{example}

We collect the general implications of Ahlfors regularity in the following lemma:

\begin{lemma} \label{thm:ahlfors-implications}
  If $X \subset \R^2$ is Ahlfors $s$-regular, then $0<\HH^s(X)<\infty$ where $s = \diml(X) = \dimh(X) = \dima(X)$.
\end{lemma}

The important observation here is that if the lower and Assouad dimensions of $X$ differ, then $X$ cannot be Ahlfors regular. For other basic properties of the Assouad and lower dimensions, we refer to the book of Fraser \cite{Fraser2020}.

\subsection{Irreducibility} \label{sec:ProdMatrix}
If $A \in \GL_2(\R)$ is an invertible $2 \times 2$-matrix, then we denote the lengths of the major and minor axis of the ellipse $A(B(0,1))$ by $\alpha_1(A)$ and $\alpha_2(A)$, respectively. Note that $\alpha_1(A) = \|A\|$ and $\alpha_2(A) = \|A^{-1}\|^{-1}$ are the square roots of the non-negative real eigenvalues of the positive semidefinite matrix $A^\top A$. It is also well-known that $|\det(A)| = \alpha_1(A)\alpha_2(A)$ for all $A \in \GL_2(\R)$. Let $\mathbb{RP}^1$ be the real projective line, that is, the set of all lines through the origin in $\R^2$. If $A \in \GL_2(\R)$ and $W \in \RP$, we write $A|W$ for the restriction of $A$ to $W$ and $\|A|W\|$ for its operator norm. For $u = (u_1,u_2)$ and $v = (v_1,v_2)$ in $\R^2$, the wedge product is $u \land v = u_1v_2 - u_2v_1$. The angle between two lines $W_1,W_2 \in \RP$ is
\begin{equation*}
  \sphericalangle(W_1,W_2) = \arcsin\biggl(\frac{|w_1 \land w_2|}{|w_1||w_2|}\biggr),
\end{equation*}
where $w_i \in W_i \setminus \{0\}$ for $i \in \{1,2\}$. The angle $\sphericalangle$ defines a metric on $\RP$, and we write $B(V,r) = \{W \in \RP : \sphericalangle(V,W) < r\}$ for the open ball centered at $V \in \RP$ with radius $r > 0$. If $m$ is a Borel measure on $\RP$, then the push-forward measure $A_*m$ is defined by $(A_*m)(E) = m(A^{-1}E)$ for all Borel sets $E \subset \RP$. We denote the Dirac mass at $V \in \RP$ by $\delta_V$.

\begin{lemma}\label{lem:goto0}
  If $(A_n)_{n \in \N}$ is a sequence of $2\times 2$ matrices such that there exists a non-atomic measure $m$ on $\RP$ for which $\lim_{n \to \infty}(A_n)_*m = \delta_{V}$ in the weak$^*$ topology, then
  \begin{equation*}
    \frac{\alpha_1(A_n^\top)}{\alpha_2(A_n^\top)} = \frac{\alpha_1(A_n)}{\alpha_2(A_n)} \to \infty
  \end{equation*}
  as $n \to \infty$,
  \begin{equation*}
    \lim_{n\to\infty}\frac{\|A_n^\top|W\|}{\alpha_1(A_n^\top)}=|\cos(\sphericalangle(V,W))|
  \end{equation*}
  for all $W\in\RP$, and
  \begin{equation*}
    \sphericalangle(A_n^\top W_1, A_n^\top W_2) \to 0
  \end{equation*}
  as $n \to \infty$ for all $W_1,W_2\in\RP\setminus\{V^\perp\}$.
\end{lemma}

\begin{proof}
  Since $\|\det(A_n)^{-1/2}A_n\|^2 = \alpha_1(A_n)/\alpha_2(A_n)$ and
  \begin{align*}
  |\sphericalangle(A_n^\top W_1, A_n^\top W_2)| &= \arcsin\biggl(\frac{|A_n^\top w_1 \land A_n^\top w_2|}{|A_n^\top w_1||A_n^\top w_2|}\biggr)\\
  &=\arcsin\biggl(\frac{|\det(A_n^\top)|}{\|A_n^\top|W_1\|\|A_n^\top|W_2\|}|\sin(\sphericalangle(W_1,W_2))|\biggr)
  \end{align*}
  for all $n \in \N$, where $w_1 \in W_1$ and $w_2 \in W_2$ are such that $|w_1|=1=|w_2|$, the lemma follows from \cite[Proposition~II.3.1]{BougerolLacroix1985}.
\end{proof}

We are primarily interested in semigroups generated by finite collections of matrices. In this context, it is rather standard practise to use separate alphabet to index the elements in the semigroup. Let $\Sigma = \{ 1,\ldots,N \}^\N$ be the collection of all infinite words obtained from integers $\{ 1,\ldots,N \}$. If $\iii = i_1i_2\cdots \in \Sigma$, then we define $\iii|_n = i_1 \cdots i_n$ for all $n \in \N$. If $\iii = i_1 \cdots i_n$, then we write $\overleftarrow{\iii} = i_n \cdots i_1$. The empty word $\iii|_0$ is denoted by $\varnothing$. Define $\Sigma_n = \{ \iii|_n : \iii \in \Sigma \}$ for all $n \in \N$ and $\Sigma_* = \bigcup_{n \in \N} \Sigma_n \cup \{ \varnothing \}$. Thus $\Sigma_*$ is the collection of all finite words. The length of $\iii \in \Sigma_* \cup \Sigma$ is denoted by $|\iii|$. The concatenation of two words $\iii \in \Sigma_*$ and $\jjj \in \Sigma_* \cup \Sigma$ is denoted by $\iii\jjj$ and the longest common prefix of $\iii$ and $\jjj$ by $\iii\wedge\jjj$. Thus $\jjj = (\iii\wedge\jjj)\jjj'$ for some $\jjj' \in \Sigma_* \cup \Sigma$.
If $\mathsf{A} = (A_1,\ldots,A_N) \in \GL_2(\R)^N$, then we write $A_\iii = A_{i_1} \cdots A_{i_n}$ and
\begin{align*}
  A_{\overleftarrow{\iii}}^\top &= (A_{\overleftarrow{\iii}})^\top = A_{i_1}^\top \cdots A_{i_n}^\top, \\
  A_{\overleftarrow{\iii}}^{-1} &= (A_{\overleftarrow{\iii}})^{-1} = A_{i_1}^{-1} \cdots A_{i_n}^{-1}
\end{align*}
for all $\iii = i_1 \cdots i_n \in \Sigma_n$ and $n \in \N$.

Let $\sigma$ be the left shift operator defined by $\sigma\iii = i_2i_3\cdots$ for all $\iii = i_1i_2\cdots \in \Sigma$. If $\iii \in \Sigma_n$ for some $n$, then we set $[\iii] = \{ \jjj \in \Sigma : \jjj|_n = \iii \}$. The set $[\iii]$ is called a \emph{cylinder set}. The \emph{shift space} $\Sigma$ is compact in the topology generated by the cylinder sets. Moreover, the cylinder sets are open and closed in this topology and they generate the Borel $\sigma$-algebra. Let $\MM_\sigma(\Sigma)$ be the collection of all $\sigma$-invariant Borel probability measures on $\Sigma$. We say that a measure $\mu$ on $\Sigma$ is \emph{fully supported} if every cylinder set has positive measure, $\mu([\iii])>0$ for all $\iii \in \Sigma_*$. If $(p_1,\ldots,p_N)$ is a probability vector, then the measure $\mu \in \MM_\sigma(\Sigma)$ for which
\begin{equation*}
  \mu([\iii])=p_{i_1}\cdots p_{i_{n}}
\end{equation*}
for all $\iii=i_1\cdots i_n\in\Sigma_n$ and $n \in \N$ is called a \emph{Bernoulli measure}. Note that a Bernoulli measure obtained from a probability vector $(p_1,\ldots,p_N)$ is fully supported if and only if $p_i>0$ for all $i \in \{1,\ldots,N\}$.

If $\A = (A_1,\ldots,A_N) \in \GL_2(\R)^N$ and $\mu \in \MM_\sigma(\Sigma)$ is a fully supported Bernoulli measure obtained from a probability vector $(p_1,\ldots,p_N)$, then the associated \emph{Furstenberg measure} is a Borel probability measure $\mu_F$ on $\RP$ satisfying
\begin{equation} \label{eq:furstenberg-measure-def}
  \mu_F=\sum_{i=1}^Np_i(A_i^{-1})_*\mu_F.
\end{equation}
We say that $\A$ is \emph{irreducible} if there does not exist $V \in \RP$ such that $A_iV=V$ for all $i \in \{ 1,\ldots,N \}$; otherwise $\A$ is \emph{reducible}. The tuple $\A$ is \emph{strongly irreducible} if there does not exist a finite set $\mathcal{V} \subset \RP$ such that $A_i\mathcal{V}=\mathcal{V}$ for all $i\in\{1,\ldots,N\}$. A matrix $A \in \GL_2(\R)$ is called \emph{proximal} if it has two real eigenvalues with different absolute values. We say that $\mathsf{A}$ is \emph{strictly affine} if there is $\iii \in \Sigma_*$ such that $A_\iii$ is proximal. If $\A$ is strictly affine, then the set of \emph{Furstenberg directions} is
\begin{equation*}
  X_F=\{A\R^2 \in \RP : A \in \overline{\{ cA_{\overleftarrow{\iii}}^{-1} : c \in \R \text{ and } \iii \in \Sigma_* \}} \text{ has rank one}\}.
\end{equation*}
The following lemma demonstrates the connection between the Furstenberg measure and directions.

\begin{lemma} \label{thm:furstenberg-delta}
  If $\A \in \GL_2(\R)^N$ is strictly affine and strongly irreducible and $\mu \in \MM_\sigma(\Sigma)$ is a fully supported Bernoulli measure, then the associated Furstenberg measure $\mu_F$ is unique and non-atomic with $\dim(\mu_F)>0$, the support of $\mu_F$ is $X_F$, and there exists a measurable function $\vartheta_2 \colon \Sigma \to X_F$ such that $\mu_F = \int_\Sigma \delta_{\vartheta_2(\iii)} \dd\mu(\iii) = (\vartheta_2)_*\mu$ and
  \begin{equation*}
    (A_{\overleftarrow{\iii|_n}}^{-1})_* \mu_F \to \delta_{\vartheta_2(\iii)}
  \end{equation*}
  in the weak$^*$ topology for $\mu$-almost all $\iii \in \Sigma$ as $n \to \infty$. In particular, for every $V \in X_F$ there exists a sequence $(n_k)_{k \in \N}$ of integers and for each $k \in \N$ there is a word $\jjj_k \in \Sigma_{n_k}$ such that
  \begin{equation*}
    (A_{\overleftarrow{\jjj_k}}^{-1})_* \mu_F \to \delta_V.
  \end{equation*}
  in the weak$^*$ topology as $k \to \infty$.
\end{lemma}

\begin{proof}
  By \cite[Theorem~II.4.1 and Corollary~VI.4.2]{BougerolLacroix1985} and \cite[Theorem 1.1]{HochmanSolomyak2017}, the Furstenberg measure $\mu_F$ is unique, non-atomic, and satisfies $\dim(\mu_F)>0$. The fact that the support of $\mu_F$ is $X_F$ follows from \cite[proof of Lemma 2.3]{BaranyKaenmakiRossi2021} and the existence of the function $\vartheta_2 \colon \Sigma \to X_F$ from \cite[Proposition II.3.3 and Theorem II.4.1]{BougerolLacroix1985}. To show the last claim, note that, by the definition of the support, there exists a sequence $(\iii_k)_{k \in \N}$ of words in $\Sigma$ such that $\vartheta_2(\iii_k)\to V$. Let $d_W$ be the Wasserstein distance (or any other metric inducing the weak$^*$ topology). For each $k \in \N$ choose $n_k \in \N$ such that
  \begin{equation*}
    d_W((A_{\overleftarrow{\iii_k|_{n_k}}}^{-1})_* \mu_F,\delta_{\vartheta_2(\iii_k)}) < \frac{1}{k},
  \end{equation*}
  where $\iii_k|_{n_k} \in \Sigma_{n_k}$. The choice of $n_k$ is justified by weak$^*$ convergence. Therefore,
  \begin{equation*}
    d_W((A_{\overleftarrow{\iii_k|_{n_k}}}^{-1})_* \mu_F,\delta_{V}) \le \frac{1}{k} + d_W(\delta_{\vartheta_2(\iii_k)},\delta_V) \to 0
  \end{equation*}
  as $k \to \infty$, and we can choose $\jjj_k$ to be $\iii_k|_{n_k}$.
\end{proof}

Let $R\colon \RP\to \RP$ be such that $R(V)=V^\perp$ for all $V\in\RP$ and write $\mu^\perp_F=R_*\mu_F$. Suppose that $\A \in \GL_2(\R)^N$ is strictly affine and strongly irreducible, and $\mu \in \MM_\sigma(\Sigma)$ is a fully supported Bernoulli measure. Then Lemma \ref{thm:furstenberg-delta} together with the facts that $R^{-1}=R$ and $A^\top V^\perp = (A^{-1}V)^\perp$ for all $A \in \GL_2(\R)$ and $V \in \RP$ implies that
\begin{equation} \label{eq:furstenberg-delta-perp}
  (A_{\overleftarrow{\iii|_n}}^\top)_* \mu_F^\perp \to \delta_{\vartheta_2(\iii)^\perp}
\end{equation}
in the weak$^*$ topology for $\mu$-almost all $\iii \in \Sigma$ as $n \to \infty$, where $\vartheta_2$ is as in Lemma \ref{thm:furstenberg-delta}. We remark that an explicit definition for the function $\vartheta_2$ can be found in \cite[\S 2.3]{BaranyKaenmakiRossi2021}.

Finally, let us analyse the defined reducibility conditions. The following lemma, which classifies the conditions, follows immediately from \cite[proof of Lemma 2.2]{BaranyKaenmakiRossi2021}.

\begin{lemma} \label{thm:A-classification}
  If $\A \in \GL_2(\R)^N$ is strictly affine, then precisely one of the following conditions hold:
  \begin{enumerate}
    \item $\A$ is strongly irreducible,
    \item $\A$ is irreducible but not strongly, i.e., the matrices in $\A$ are simultaneously diagonal or antidiagonal in some basis so that there is at least one antidiagonal matrix,
    \item $\A$ is reducible, i.e., the matrices in $\A$ are simultaneously upper triangular in some basis.
  \end{enumerate}
\end{lemma}

Note that if $\A$ consists only of antidiagonal matrices, then their second iterates are diagonal and hence reducible.

\subsection{Domination}
We say that $\mathsf{A} = (A_1,\ldots,A_N) \in \GL_2(\R)^N$ is \emph{dominated} if there exist constants $C>0$ and $0<\tau<1$ such that
\begin{equation}\label{eq:domconst}
	\alpha_2(A_\iii) \le C\tau^{|\iii|}\alpha_1(A_\iii)
\end{equation}
for all $\iii \in \Sigma_*$. By \cite[Corollary 2.4]{BaranyKaenmakiMorris2018}, a dominated tuple is strictly affine. We call a proper subset $\CC\subset\RP$ a \emph{multicone} if it is a finite union of closed projective intervals. We say that $\CC\subset\RP$ is a \emph{strongly invariant multicone} for $\mathsf{A}$ if it is a multicone and $A_i\CC\subset\CC^o$ for all $i \in \{1,\ldots,N\}$, where $\CC^o$ is the interior of $\CC$. For example, the first quadrant is strongly invariant for any tuple of positive matrices. By \cite[Theorem~B]{BochiGourmelon2009}, $\mathsf{A}$ has a strongly invariant multicone if and only if $\mathsf{A}$ is dominated. Furthermore, if $\CC\subset\RP$ is a strongly invariant multicone for $\A$, then $\overline{\RP \setminus \CC}$ and $\{V^\perp : V \in \overline{\RP \setminus \CC}\}$ are strongly invariant multicones for $\A^{-1} = (A_1^{-1},\ldots,A_N^{-1})$ and $\A^\top = (A_1^\top,\ldots,A_N^\top)$, respectively.

Observe that if $\A$ is dominated and $\CC\subset\RP$ is a strongly invariant multicone for $\mathsf{A}$, then the set of Furstenberg directions is the compact set
\begin{equation} \label{eq:furstenberg-directions-dominated}
  X_F = \bigcap_{n=1}^\infty \bigcup_{\iii \in \Sigma_n} A_{\overleftarrow{\iii}}^{-1} \overline{\RP \setminus \CC}.
\end{equation}
If $\Pi \colon \Sigma \to \RP$ is the \emph{canonical projection} defined by the relation
\begin{equation} \label{eq:furstenberg-directions-dominated2}
  \{\Pi(\iii)\} = \bigcap_{n=1}^\infty A_{\overleftarrow{\iii|_n}}^{-1} \overline{\RP \setminus \CC},
\end{equation}
then it is easy to see that $X_F = \bigcup_{\iii \in \Sigma} \Pi(\iii)$. The canonical projection is continuous by \cite[Lemma 2.4]{BaranyRams2018}. Note that the definition of $\Pi$ does not depend on the choice of the strongly invariant multicone $\CC$ and $X_F$ is perfect unless it is a singleton.

\begin{lemma} \label{thm:bochi-morris}
  If $\mathsf{A} \in \GL_2(\R)^N$ is dominated, then there exists a constant $D \ge 1$ such that
  \begin{equation*}
    \|A_\iii^\top|V^\perp\|\leq\alpha_1(A_\iii)=\alpha_1(A_\iii^\top)\leq D\|A_\iii^\top|V^\perp\|
  \end{equation*}
  for all $\iii\in\Sigma_*$ and $V\in X_F$. Furthermore, if $\iii \in \Sigma$ and $V = \Pi(\iii)$, then
  \begin{equation*}
    D^{-1}\|A_{\overleftarrow{\iii|_n}}|V\| \le \alpha_2(A_{\overleftarrow{\iii|_n}}) = \alpha_2(A_{\overleftarrow{\iii|_n}}^\top) \le \|A_{\overleftarrow{\iii|_n}}|V\|
  \end{equation*}
  for all $n \in \N$.
\end{lemma}

\begin{proof}
  To prove the first claim, fix $V \in \bigcup_{i=1}^N A_i^{-1} \overline{\RP \setminus \CC}$ and let $v \in V^\bot$ be such that $|v|=1$. Notice that $X_F \subset \bigcup_{i=1}^N A_i^{-1} \overline{\RP \setminus \CC}$ and $A_\iii^\top \{V^\bot : V \in \overline{\RP \setminus \CC}\} \subset \{V^\bot : V \in \bigcup_{i=1}^N A_i^{-1}\overline{\RP \setminus \CC}\}$ for all $\iii \in \Sigma_* \setminus \{\varnothing\}$. By \cite[Lemma 2.2]{BochiMorris2015}, there exists a constant $D \ge 1$ such that
  \begin{equation*}
    \|A_\iii^\top|V^\bot\| = |A_\iii^\top v| \ge D^{-1}\|A_\iii^\top\| = D^{-1}\|A_\iii\|
  \end{equation*}
  for all $\iii \in \Sigma_*$.

  To show the second claim, let $V \in \RP$ be the only element in $\bigcap_{n=1}^\infty A_{\overleftarrow{\iii|_n}}^{-1}\overline{\RP \setminus \CC}$ and let $v \in V$ be such that $|v|=1$. Notice that $A_\jjj^{-1}\overline{\RP \setminus \CC} \subset \bigcup_{i=1}^N A_i^{-1}\overline{\RP \setminus \CC}$ for all $\jjj \in \Sigma_* \setminus \{\varnothing\}$. Fixing $n \in \N$ we see that $A_{\overleftarrow{\iii|_n}}v \in \bigcap_{k=1}^\infty A_{\overleftarrow{\sigma^n\iii|_k}}^{-1}\overline{\RP \setminus \CC} \subset \bigcup_{i=1}^N A_i^{-1}\overline{\RP \setminus \CC}$. Therefore, again by \cite[Lemma 2.2]{BochiMorris2015}, there exists a constant $D \ge 1$ such that
  \begin{equation*}
    1 = |v| = |A_{\overleftarrow{\iii|_n}}^{-1}A_{\overleftarrow{\iii|_n}}v| \ge D^{-1}\|A_{\overleftarrow{\iii|_n}}^{-1}\||A_{\overleftarrow{\iii|_n}}v| = D^{-1}\|A_{\overleftarrow{\iii|_n}}^{-1}\|\|A_{\overleftarrow{\iii|_n}}|V\|
  \end{equation*}
  which finishes the proof.
\end{proof}

The proof of the following lemma is a simple consequence of \cite[Lemma~5.2]{BaranyJordanKaenmakiRams2021}.

\begin{lemma}\label{lem:subsystem}
  If $\mathsf{A} \in \GL_2(\R)^N$ is strictly affine and strongly irreducible, then there exists a finite set $J\subset\Sigma_*$ such that the tuple $(A_{\iii})_{\iii\in J}$ is dominated and strongly irreducible.
\end{lemma}

If $V,W \in \RP$, then the \emph{projection} $\proj_V^W \colon \R^2 \to V$ is the linear map such that $\proj_V^W|_V=\mathrm{Id}|_V$ and $\ker(\proj_V^W)=W$. The \emph{orthogonal projection} $\proj_V^{V^\bot}$ onto the subspace $V$ is denoted by $\proj_V$. Note that $(\proj_{V}A)^\top = A^\top \proj_{V}$ and hence,
\begin{equation} \label{eq:transpose-projection}
  \|A^\top|V^\bot\| = \|\proj_{V^\bot}A\|
\end{equation}
for all $A \in \GL_2(\R)$ and $V \in \RP$. Recall that a $2 \times 2$-matrix $A$ has rank one if and only if there exist $v,w \in \R^2 \setminus \{(0,0)\}$ such that $A = vw^\top$ with $\im(A)=\linspan(v)$ and $\ker(A)=\linspan(w)^\bot$. It is easy to see, consult e.g.\ \cite[Lemma 2.1]{KaenmakiNissinen-preprint}, that in such a case,
\begin{equation} \label{eq:rank-one-projection-nilpotent}
  A =
  \begin{cases}
    \langle v,w \rangle\proj_{\im(A)}^{\ker(A)}, &\text{if $A$ is not nilpotent}, \\
    |v||w|R\proj_{\ker(A)^\perp}, &\text{if $A$ is nilpotent},
  \end{cases}
\end{equation}
where $R \in O(2)$ is a rotation by an angle $\pi/2$. In particular, $A(X)$ is bi-Lipschitz equivalent to $\proj_{\ker(A)^\perp}(X)$ for all $X \subset \R^2$. The following lemma guarantees that nilpotent matrices do not appear in the dominated case.

\begin{lemma} \label{thm:nilpotent}
  If $\A \in \GL_2(\R)^N$ is dominated, then the closure of $\{cA_\iii : c \in \R \text{ and }\iii \in \Sigma_*\}$ does not contain non-zero nilpotent elements. In other words, rank one matrices in the above closure are all projections.
\end{lemma}

\begin{proof}
  Let us assume that there exists a non-zero nilpotent matrix $P$ in the closure of $\{cA_\iii : c \in \R \text{ and }\iii \in \Sigma_*\}$. By definitions, $P^2=0$ and there exists sequences $(\iii_n)_{n \in \N}$ of finite words in $\Sigma_*$ and $(c_n)_{n \in \N}$ of non-zero real numbers such that $c_n A_{\iii_n} \to P$ as $n \to \infty$. The domination guarantees that, by possibly going through a sub-sequence, $A_{\iii_n}/\|A_{\iii_n}\|$ converges to a rank one matrix and so $\lim_{n\to\infty}c_n/\|A_{\iii_n}\|\in\R\setminus\{0\}$. Thus, without loss of generality, we may assume that $c_n = \|A_{\iii_n}\|$.

  The domination guarantees that, by possibly taking another sub-sequence, also $A_{\iii_n}^2/\|A_{\iii_n}^2\|$ converges to a rank one matrix $Q$. By \cite[Corollary 2.4]{BaranyKaenmakiMorris2018}, there exists a constant $C \ge 1$ such that $\|A_{\iii_n}^2\| \le \|A_{\iii_n}\|^2 \le C\|A_{\iii_n}^2\|$ for all $n \in \N$. Hence, $A_{\iii_n}^2/\|A_{\iii_n}^2\|$ must converge to a constant times $P^2$. Since $P^2=0$, this contradicts for $Q$ being rank one.
\end{proof}

A dominated tuple $\A$ is not necessarily irreducible and an irreducible tuple $\A$ is not necessarily dominated. For example, consider a tuple of diagonal matrices and a tuple of diagonal and antidiagonal matrices. The following lemma shows that together the properties imply strong irreducibility.

\begin{lemma} \label{thm:irreducible-dominated}
  If $\A \in \GL_2(\R)^N$ is dominated and irreducible, then $\A$ is strongly irreducible.
\end{lemma}

\begin{proof}
  Since $\mathsf{A}$ is dominated, it contains only proximal elements and is hence strictly affine; see \cite[Corollary 2.4]{BaranyKaenmakiMorris2018}. Therefore, as $\A$ is irreducible, it suffices to show that the condition (2) in Lemma \ref{thm:A-classification} does not hold. Let $\CC \subset \RP$ be a strongly invariant multicone for $\A$. If $\A$ contains an antidiagonal matrix $A$, then $A^2$ is a constant times the identity matrix and $A^2\CC = \CC$ which is a contradiction. Hence, $\A$ is strongly irreducible.
\end{proof}

The following lemma studies dominated and reducible tuples for which the set of Furstenberg directions is non-trivial.

\begin{lemma} \label{thm:need-for-bhr-prop66}
  If $\A = (A_1,\ldots,A_N) \in \GL_2(\R)^N$ is dominated and reducible such that $X_F$ is not a singleton, then, possibly after a change of basis,
  \begin{equation} \label{eq:triangform}
    A_i =
    \begin{pmatrix}
      a_i & b_i \\
      0   & d_i
    \end{pmatrix}
  \end{equation}
  with $0 < |d_i| < |a_i| < 1$ for all $i \in \{1,\ldots,N\}$, and the matrices are not simultaneously diagonalizable.
\end{lemma}

\begin{proof}
  If $X$ is reducible, then, by Lemma \ref{thm:A-classification}(3), the matrices $A_i$ are simultaneously upper triangular in some basis and have the form \eqref{eq:triangform}. If there exists $i\in\{1,\ldots,N\}$ such that $|d_i|\geq|a_i|$, then it is easy to see that for any subspace $W\in\RP$, we have $A_i^nW\to\mathrm{span}(e_1)$ as $n \to \infty$, where $e_1 = (1,0)$. Thus, since the matrices are dominated, any strongly invariant multicone $\CC$ must contain $\mathrm{span}(e_1)$ as an interior point. But such a cone cannot be strongly invariant for matrices of the form \eqref{eq:triangform} unless $|d_i|>|a_i|$ for all $i \in \{1,\ldots,N\}$, which contradicts the assumption that $X_F$ is not a singleton. Hence, $|d_i|<|a_i|$ for all $i \in \{1,\ldots,N\}$. Finally, as $X_F$ is not a singleton, the matrices cannot be simultaneously diagonalisable.
\end{proof}

\subsection{Equilibrium states} \label{sec:eq-states}

For each $A \in \GL_2(\R)$ and $s \ge 0$ we define the \emph{singular value function} by setting
\begin{equation*}
  \fii^s(A) =
  \begin{cases}
    \|A\|^s = \alpha_1(A)^s, &\text{if } 0 \le s \le 1, \\
    \|A\|^{2-s}|\det(A)|^{s-1} = \alpha_1(A)\alpha_2(A)^{s-1}, &\text{if } 1 < s \le 2, \\
    |\det(A)|^{s/2} = \alpha_1(A)^{s/2}\alpha_2(A)^{s/2}, &\text{if } s > 2.
  \end{cases}
\end{equation*}
The value $\fii^s(A)$ represents a measurement of the $s$-dimensional volume of the ellipse $A(B(0,1))$. Since the determinant is multiplicative and the operator norm is sub-multiplicative, we have $\varphi^s(AB) \leq \varphi^s(A)\varphi^s(B)$ for all $A,B \in \GL_2(\R)$ and $s \ge 0$. For each $\A \in \GL_2(\R)^N$ and $s \ge 0$ we define the \emph{pressure} by setting
\begin{equation*}
  P(\A,s) = \lim_{n \to \infty} \frac{1}{n}\log\sum_{\iii \in \Sigma_n} \fii^s(A_\iii).
\end{equation*}
As the singular value function is sub-multiplicative, the sequence $(\log\sum_{\iii \in \Sigma_n} \fii^s(A_\iii))_{n \in \N}$ is sub-additive and hence, the limit above exists by Fekete's lemma. It is also easy to see that the pressure $P(\A,s)$ is continuous and strictly decreasing as a function of $s$ with $P(\A,0) \ge 0$ and $\lim_{s \to \infty}P(\A,s) = -\infty$. We may thus define the \emph{affinity dimension} by setting $\dimaff(\A)$ to be the unique $s \ge 0$ for which $P(\A,s)=0$.

\begin{lemma} \label{thm:equilibrium-state-existence}
  If $\mathsf{A} \in \GL_2(\R)^N$ is dominated or irreducible, then there exists $\mu_K \in \MM_\sigma(\Sigma)$ such that
  \begin{equation*}
    C^{-1}\fii^s(A_\iii) \le \mu_K([\iii]) \le C\fii^s(A_\iii)
  \end{equation*}
  for all $\iii \in \Sigma_*$, where $s = \dimaff(\A)$.
\end{lemma}

\begin{proof}
  If $\A$ is irreducible, then the existence of the claimed measure $\mu_K \in \MM_\sigma(\Sigma)$ follows immediately from \cite[Proposition 1.2 and \S 3]{FengKaenmaki2011}. If $\A$ is reducible but dominated, then the claim follows by \cite[Proposition 2.6]{BaranyKaenmakiMorris2018}.
\end{proof}

The measure $\mu_K \in \MM_\sigma(\Sigma)$ in Lemma \ref{thm:equilibrium-state-existence} is called the \emph{equilibrium state} as its canonical projection on the self-affine set $X$ is the natural canditate to attain the maximum in \eqref{eq:variational-principle}. Equilibrium states exist also without domination or irreducibility but in this case, they satisfy sligthly weaker properties than described in Lemma \ref{thm:equilibrium-state-existence} and are not necessarily unique; see \cite{Kaenmaki2004} and \cite{FengKaenmaki2011}.

Suppose that $\mathsf{A} \in \GL_2(\R)^N$ is dominated. For each $s \ge 0$ define a function $g_s \colon \Sigma \to \R$ by setting
\begin{equation*}
  g_s(\iii) =
  \begin{cases}
    \log\|A_{\iii|_1}^\top|\Pi(\sigma\iii)^\perp\|^s, &\text{if } 0 \le s \le 1, \\
    \log\|A_{\iii|_1}^\top|\Pi(\sigma\iii)^\perp\|^{2-s}|\det(A_{\iii|_1})|^{s-1}, &\text{if } 1 < s \le 2, \\
    \log|\det(A_{\iii|_1})|^{s/2}, &\text{if } s > 2,
  \end{cases}
\end{equation*}
for all $\iii \in \Sigma$, where $\Pi$ is as in \eqref{eq:furstenberg-directions-dominated2}. Notice that $g_s$ is H\"older continuous. The \emph{Perron-Frobenius operator} $\LL$ for $s$ is the positive linear operator defined by setting
\begin{equation*}
  \mathcal{L}f(\iii) = \sum_{i=1}^N \exp(g_s(i\iii)) f(i\iii)
\end{equation*}
for all continuous functions $f \colon \Sigma \to \R$. Observe that, by Lemma \ref{thm:bochi-morris}, there exists a constant $D \ge 1$ such that
\begin{equation*}
  \log\|A_{\overleftarrow{\iii|_n}}^\top\| - \log D \le \log\|A_{\overleftarrow{\iii|_n}}^\top|\Pi(\sigma^n\iii)^\bot\| = \sum_{k=0}^{n-1} \log\|A_{\sigma^k\iii|_1}^\top|\Pi(\sigma^{k+1}\iii)^\bot\|
\end{equation*}
and hence, the \emph{Birkhoff sum} of $g_s$, $\sum_{k=0}^{n-1} g_s(\sigma^k\iii)$, satisfies
\begin{equation} \label{eq:birkhoff-sum}
  \log\fii^s(A_{\overleftarrow{\iii|_n}}^\top) - \log D \le \sum_{k=0}^{n-1} g_s(\sigma^k\iii) \le \log\fii^s(A_{\overleftarrow{\iii|_n}}^\top)
\end{equation}
for all $\iii \in \Sigma$ and $n \in \N$. The following lemma is a simple consequence of the classical Ruelle's Perron-Frobenius Theorem.

\begin{lemma} \label{thm:perron-frobenius}
  If $\mathsf{A} \in \GL_2(\R)^N$ is dominated and $\LL$ is the Perron-Frobenius operator for $\dimaff(\mathsf{A})$, then there exist a unique continuous function $h \colon \Sigma \to (0,\infty)$ and a unique Borel probability measure $\nu$ on $\Sigma$ such that
  \begin{equation*}
    \mathcal{L}h = h, \quad \int_{\Sigma} h(\iii) \dd\nu(\iii) = 1,
  \end{equation*}
  and
  \begin{equation*}
    \lim_{n \to \infty} \sup_{\iii \in \Sigma} \biggl|\LL^nf(\iii) - h(\iii) \int_{\Sigma} f(\jjj) \dd\nu(\jjj)\biggr| = 0
  \end{equation*}
  for all continuous functions $f \colon \Sigma \to \R$.
\end{lemma}

\begin{proof}
  Write $s = \dimaff(\A)$. By \cite[Theorem 1.7]{Bowen}, there exist a unique continuous function $h \colon \Sigma \to (0,\infty)$ and a unique Borel probability measure $\nu$ on $\Sigma$ such that
  \begin{equation*}
    \mathcal{L}h = \lambda h, \quad \mathcal{L}^*\nu = \lambda\nu, \quad \int_{\Sigma} h(\iii) \dd\nu(\iii) = 1,
  \end{equation*}
  and
  \begin{equation*}
    \lim_{n \to \infty} \sup_{\iii \in \Sigma} \biggl|\lambda^{-n}\LL^nf(\iii) - h(\iii) \int_{\Sigma} f(\jjj) \dd\nu(\jjj)\biggr| = 0
  \end{equation*}
  for all continuous functions $f \colon \Sigma \to \R$, where $\lambda = \LL^*\nu(\Sigma) > 0$ and
  \begin{equation*}
    \log\lambda = \lim_{n \to \infty} \frac{1}{n}\log\sum_{\iii \in \Sigma_n} \exp\sup_{\jjj \in [\iii]} \biggl(\sum_{k=1}^{n-1}g_s(\sigma^k\jjj)\biggr).
  \end{equation*}
  Therefore, by \eqref{eq:birkhoff-sum} and the choice of $s$, we have $\lambda = \exp(P(\A^\top,s)) = \exp(P(\A,s)) = 1$ which completes the proof.
\end{proof}

\subsection{Self-affine set} \label{sec:self-affine-def}
We consider a tuple $\Phi = (A_1+v_1,\ldots,A_N+v_N)$ of contractive invertible affine self-maps on $\R^2$, where we have written $A+v$ to denote the affine map $x \mapsto Ax+v$ defined on $\R^2$ for all $2 \times 2$ matrices $A$ and translation vectors $v \in \R^2$. Such a tuple $\Phi$ is called an \emph{affine iterated function system}. We also write $\fii_i = A_i+v_i$ for all $i \in \{1,\ldots,N\}$ and $\fii_\iii = \fii_{i_1} \circ \cdots \circ \fii_{i_n}$ for all $\iii = i_1 \cdots i_n \in \Sigma_n$ and $n \in \N$. Note that the associated tuple of matrices $(A_1,\ldots,A_N)$ is an element of $\GL_2(\R)^N$ and satisfies $\max_{i \in \{1,\ldots,N\}}\|A_i\|<1$. It is a classical result of Hutchinson \cite{Hutchinson1981} that for each affine iterated function system there exists a unique non-empty compact set $X \subset \R^2$, called the \emph{self-affine set}, such that
\begin{equation} \label{eq:self-affine-set-def}
  X = \bigcup_{i=1}^N \fii_i(X).
\end{equation}
We use the convention that whenever we speak about a self-affine set $X$, then it is automatically accompanied with a tuple of affine maps which defines it. This makes it possible to write that e.g.\ ``$X$ is dominated'' which obviously then means that ``the associated tuple $\A$ of matrices is dominated''. Similarly, by $\dimaff(X)$ we mean the affinity dimension $\dimaff(\A)$ defined in \S \ref{sec:eq-states}.

We are interested in understanding the geometry of self-affine sets. Relying on \eqref{eq:self-affine-set-def}, the self-affine set $X$ can naturally be covered by the sets $\fii_\iii(B)$, where $B$ is a ball containing $X$. Observe that each ellipse $\fii_\iii(B)$ can be covered by one ball of radius $\alpha_1(A_\iii)\diam(B)$ or by $\alpha_1(A_\iii)/\alpha_2(A_\iii)$ many balls of radius $\alpha_2(A_\iii)\diam(B)$. This motivates us to study the limiting behavior of the sums $\sum_{\iii \in \Sigma_n} \fii^s(A_\iii)$ and indeed, it is straightforward to see that $\dimh(X) \le \min\{2,\dimaff(X)\}$.

Every affine iterated function system is associated with the \emph{canonical projection} $\pi\colon \Sigma \to X$ which is defined by $\pi(\iii) = \sum_{n=1}^\infty A_{\iii|_{n-1}} v_{i_n}$ for all $\iii = i_1i_2\cdots \in \Sigma$. It is easy to see that $\pi$ is continuous and the image of $\Sigma$ is the self-affine set, $\pi(\Sigma) = X$. Separation conditions allow simple interplay between $\Sigma$ and $X$. We say that $X$ satisfies the \emph{strong separation condition} if $\fii_i(X) \cap \fii_j(X) = \emptyset$ whenever $i \ne j$. In this case, we have
\begin{equation} \label{eq:SSC}
  \delta = \min_{i \ne j}\dist(\fii_i(X),\fii_j(X)) > 0.
\end{equation}
As $\pi([\iii]) = \fii_\iii(X)$ for all $\iii \in \Sigma_*$, the strong separation condition is characterized by the requirement that the canonical projection is one-to-one. We say that $X$ satisfies the \emph{open set condition} if there exists an open set $U \subset \R^2$ such that $\fii_i(U) \subset U$ for all $i \in \{1,\ldots,N\}$ and $\fii_i(U) \cap \fii_j(U) = \emptyset$ whenever $i \ne j$. If such a set $U$ also intersects $X$, then we say that $X$ satisfies the \emph{strong open set condition}. Observe that the strong separation condition implies the strong open set condition.

Let us first survey known results for self-similar sets which are a special case of self-affine sets. If $(\lambda_1O_1+v_1,\ldots,\lambda_NO_N+v_N)$, where $0<\lambda_i<1$ and $O_i \in O(2)$ for all $i \in \{1,\ldots,N\}$, is a tuple of contractive similarities on $\R^2$, then we call the associated self-affine set $X$ \emph{self-similar}. In this case, the affinity dimension is called \emph{similarity dimension} and we denote it by $\dimsim(X)$. Notice that $\dimsim(X)$ is the unique $s \ge 0$ for which $\sum_{i=1}^N \lambda_i^s = 1$. Let us endow the group of all affine maps with the topology of pointwise convergence: the distance $|\fii-\psi|$ between two affine maps $\fii$ and $\psi$ is the supremum of $|\fii(x)-\psi(x)|$ over the unit ball. Define
\begin{equation*}
  \Sigma(x,r) = \{ \iii \in \Sigma_* : \diam(\fii_\iii(X)) \le r < \diam(\fii_{\iii^-}(X)) \text{ and } \fii_\iii(X) \cap B(x,r) \ne \emptyset \}
\end{equation*}
for all $x \in \R^2$ and $r>0$.

\begin{theorem} \label{thm:self-similar-osc-equiv}
  If $X$ is a planar self-similar set, then the following conditions are equivalent:
  \begin{enumerate}
    \item $X$ satisfies the open set condition,
    \item $X$ satisfies the strong open set condition,
    \item $\sup \{\#\Sigma(x,r) : x \in X \text{ and } r>0 \} < \infty$,
    \item the identity is not in the closure of $\{\fii_\iii^{-1} \circ \fii_\jjj : \iii,\jjj \in \Sigma_* \text{ such that }\iii\ne\jjj\}$,
    \item there is $\eta>0$ such that $|\fii_\iii-\fii_\jjj| \ge \eta\diam(\fii_\iii(X))$ for all $\iii,\jjj \in \Sigma_*$ with $\iii\ne\jjj$,
    \item $\HH^s(X)>0$ where $s=\dimsim(X)$,
    \item $X$ is Ahlfors $s$-regular where $s = \dimsim(X)$,
  \end{enumerate}
\end{theorem}

\begin{proof}
  Notice that (2) $\Rightarrow$ (1) is a triviality and (7) $\Rightarrow$ (6) follows from Lemma \ref{thm:ahlfors-implications}. Hutchinson \cite[\S 5.3]{Hutchinson1981} proved the implication (1) $\Rightarrow$ (3) $\Rightarrow$ (7), Bandt and Graf \cite{BandtGraf1992} showed that (6) $\Leftrightarrow$ (4) $\Leftrightarrow$ (5), and finally, Schief \cite[Theorem 2.1]{Schief1994} verified the remaining implication (6) $\Rightarrow$ (2).
\end{proof}

Recall that if $X$ is a self-similar set, then $\HH^s(X)<\infty$ for $s=\dimh(X)$, regardless of any separation conditions; see \cite[Theorem 4]{Falconer1989}. We say that $X$ satisfies the \emph{weak separation condition} if
\begin{equation*}
  \sup \{\#\Phi(x,r) : x \in X \text{ and } r>0 \} < \infty,
\end{equation*}
where
\begin{equation*}
  \Phi(x,r) = \{ \fii_\iii : \diam(\fii_\iii(X)) \le r < \diam(\fii_{\iii^-}(X)) \text{ and } \fii_\iii(X) \cap B(x,r) \ne \emptyset \}
\end{equation*}
for all $x \in \R^2$ and $r>0$. This notion was introduced by Lau and Ngai \cite{LauNgai1999}, and the above equivalent formulation is due to Zerner \cite{Zerner1996}. The open set condition holds if and only if the weak separation condition is satisfied and $\fii_\iii \ne \fii_\jjj$ for all $\iii,\jjj \in \Sigma_*$ with $\iii \ne \jjj$.

\begin{theorem} \label{thm:self-similar-wsc-equiv}
  If $X$ is a planar self-similar set, then the following conditions are equivalent:
  \begin{enumerate}
    \item $X$ satisfies the weak separation condition,
    \item the identity is not an accumulation point of $\{\fii_\iii^{-1} \circ \fii_\jjj : \iii,\jjj \in \Sigma_* \text{ such that }\iii\ne\jjj\}$,
    \item there is $\eta>0$ such that $|\fii_\iii-\fii_\jjj| \ge \eta\diam(\fii_\iii(X))$ for all $\iii,\jjj \in \Sigma_*$ with $\fii_\iii\ne\fii_\jjj$.
  \end{enumerate}
  Furthermore, the following conditions follow from the above conditions and, if $X$ is not contained in a line and $\dimh(X) \le 1$, or alternatively, if $\dimh(X) < 1$, then all the conditions are equivalent:
  \begin{enumerate}[resume]
    \item $\HH^s(X)>0$ where $s=\dimh(X)$,
    \item $X$ is Ahlfors regular,
    \item $\diml(X) = \dimh(X) = \dima(X)$.
  \end{enumerate}
\end{theorem}

\begin{proof}
  It follows from Angelevska, K\"aenm\"aki, and Troscheit \cite[Theorem 3.2]{AngelevskaKaenmakiTroscheit2020} that (1) $\Leftrightarrow$ (2) $\Leftrightarrow$ (3). Furthermore, by \cite[Theorem 3.1]{AngelevskaKaenmakiTroscheit2020}, we have (4) $\Leftrightarrow$ (5). Note also that \cite[Proposition 3.3]{AngelevskaKaenmakiTroscheit2020} verifies the implication (1) $\Rightarrow$ (4). The implication (5) $\Rightarrow$ (6) follows immediately from Lemma \ref{thm:ahlfors-implications}. Finally, Fraser, Henderson, Olson, and Robinson \cite[Theorems 3.1 and 3.2]{FraserHendersonOlsonRobinson2015} proved that if $X$ does not satisfy (2), then $\dima(X) \ge 1$, and Garc\'ia \cite[Theorem 1.4]{Garcia2020} demostrated that if, in addition, $X$ is not contained in a line, then $\dima(X) > 1$. Therefore, under the mentioned extra assumptions, we have the implication (6) $\Rightarrow$ (2).
\end{proof}

The assumption $\dimh(X) \le 1$ in the above theorem is essential: a slight modification of \cite[Proposition 3.3]{FarkasFraser2015} shows that for each $1 < s \le 2$ there exists a planar Ahlfors $s$-regular self-similar set not satisfying the weak separation condition. Furthermore, a line is an Ahlfors $1$-regular set and it can be expressed as a self-similar set not satisfying the weak separation condition. This shows that none of the conditions in the second group imply the conditions in the first group without the extra assumption.

Let us next survey dimension results for a special case of self-affine sets, Bedford-McMullen carpets, which are constructed by affine maps sharing a common diagonal matrix as a linear part. Let $q > p \ge 2$ and $N \in \{2,\ldots,pq\}$ be integers, and $I \subset \{0,\ldots,p-1\} \times \{0,\ldots,q-1\}$ a set of $N$ elements. A \emph{Bedford-McMullen carpet} is the self-affine set $X \subset [0,1]^2$ associated to a tuple $(\fii_1,\ldots,\fii_N)$ of affine maps which all have the same linear part $\diag(\tfrac{1}{p},\tfrac{1}{q})$ and the translation part is from the set $\{(\tfrac{j}{p},\tfrac{k}{q}) \in [0,1]^2 : (j,k) \in I\}$. We assume that each map in the tuple appears there only once. Write $n_j = \#\{k : (j,k) \in I\}$ to denote the number of sets $\fii_i([0,1)^2)$ the vertical line $\{(\tfrac{j}{p},y) : y \in \R\}$ intersects. We say that the Bedford-McMullen carpet $X$ has \emph{uniform vertical fibers} if there is $n \in \N$ such that $n_j = n$ for all $j$ with $n_j \ne 0$.

\begin{theorem} \label{thm:BM-carpet-regular}
  If $X$ is a Bedford-McMullen carpet, then the following conditions are equivalent:
  \begin{enumerate}
    \item $X$ has uniform vertical fibers,
    \item $\HH^s(X) < \infty$ where $s = \dimh(X)$,
    \item $X$ is Ahlfors regular,
    \item ${\diml(X)=}\dimh(X)=\dima(X)$.
  \end{enumerate}
\end{theorem}

\begin{proof}
  The implication (1) $\Rightarrow$ (3) follows from McMullen \cite{McMullen1984}. Lemma \ref{thm:ahlfors-implications} shows (3) $\Rightarrow$ (2) and (3) $\Rightarrow$ (4). Finally, Peres \cite[Theorem 1]{Peres1994} has shown the implication (2) $\Rightarrow$ (1) and Fraser \cite[Corollary 2.14]{Fraser14}, extending the result of Mackay \cite[Theorem 1.1]{Mackay2011}, proved the implication (4) $\Rightarrow$ (1).
\end{proof}

Let us then turn to the general self-affine case. Recall that the set of all irreducible tuples $\A \in \GL_2(\R)^N$ is open, dense, and full Lebesgue measure in $\GL_2(\R)^N$. In fact, the set of all reducible tuples $\A \in \GL_2(\R)^N$ is a finite union of $(4N-1)$-dimensional algebraic varieties; see \cite[Propositions 3.4 and 3.6]{KaenmakiLi2017}. Recall also that the set of all dominated tuples $\A \in \GL_2(\R)^N$ is open in $\GL_2(\R)^N$.

Theorem \ref{thm:self-similar-osc-equiv} shows that on self-similar sets the open set condition and the strong open set condition are equivalent. On strictly affine strongly irreducible planar self-affine sets, the open set condition is not a sufficient assumption for any meaningful dimension result; see \cite[Example 5.5]{MorrisShmerkin2019}. Nevertheless, the strong open set condition still has a role. The following breakthrough result is proven by B\'ar\'any, Hochman, and Rapaport \cite[Theorems 1.1 and 7.1]{BHR}:

\begin{theorem} \label{thm:BHR}
  If $X$ is a strictly affine strongly irreducible planar self-affine set satisfying the strong open set condition, then
  \begin{align*}
    \dimh(X) &= \min\{2,\dimaff(X)\}, \\
    \dimh(\proj_{V^\bot}(X)) &= \min\{1,\dimaff(X)\}
  \end{align*}
  for all $V \in \RP$.
\end{theorem}

Recall that if a planar self-affine set $X$ is dominated, then, by \cite[Corollary 2.4]{BaranyKaenmakiMorris2018}, it is strictly affine. Therefore, by Lemma \ref{thm:irreducible-dominated}, dominated irreducible planar self-affine sets satisfy the hypothesis of Theorem \ref{thm:BHR}. Note that if $X$ is irreducible, then $X_F$ is not a singleton. It turns out that, under domination, the assumption that $X_F$ is not a singleton is enough. Indeed, by recalling Lemma \ref{thm:need-for-bhr-prop66}, we may rely on \cite[Proposition 6.6]{BHR}\footnote{Note that the formulation of \cite[Proposition 6.6]{BHR} has a small mistake: the proposition should exclude the span of $(1,0)$, not the span of $(0,1)$.} to arrive at the following theorem:

\begin{theorem} \label{thm:BHR2}
  If $X$ is a dominated planar self-affine set satisfying the strong open set condition such that $X_F$ is not a singleton, then
  \begin{align*}
    \dimh(X) &= \min\{2,\dimaff(X)\}, \\
    \dimh(\proj_{V^\bot}(X)) &= \min\{1,\dimaff(X)\}
  \end{align*}
  for all $V \in \RP \setminus \II$, where $\II = \{W \in \RP : W = A_iW \text{ for all } i \in \{1,\ldots,N\} \}$ and contains at most one element.
\end{theorem}

We remark that Hochman and Rapaport \cite{HochmanRapaport2021} have recently generalized the above results. They showed that the strong open set condition can be replaced by exponential separation, a separation condition which allows overlapping. Our standing assumption is the strong separation condition and therefore Theorems \ref{thm:BHR} and \ref{thm:BHR2} suffice for us.

\section{Main results} \label{sec:main-results}

As explained in the introduction, our goal is to extract finer geometric information in the setting of Theorems \ref{thm:BHR} and \ref{thm:BHR2}, in analogy with known phenomena for self-similar sets and Bedford-McMullen carpets. We first state the main implications, then derive the corollaries announced in the introduction. Our first observation in this direction follows immediately from the following lemma.

\begin{lemma} \label{thm:FK}
  If $X$ is a dominated or irreducible planar self-affine set, then $\HH^s(X)<\infty$ where $s = \dimaff(X)$.
\end{lemma}

\begin{proof}
  To prove the first claim, let $B$ be a closed ball containing $X$ and $\iii \in \Sigma_*$. To cover the ellipsis $\fii_\iii(B)$, we need approximately one ball of radius $\alpha_1(A_\iii)$ or $\alpha_1(A_\iii)/\alpha_2(A_\iii)$ many balls of radius $\alpha_2(A_\iii)$. Thus, by the definitions of the Hausdorff measure and the singular value function, there exists a constant $c>0$ such that
  \begin{equation*}
    \HH^s(X) \le c\lim_{n \to \infty} \sum_{\iii \in \Sigma_n} \fii^s(A_\iii).
  \end{equation*}
  By Lemma \ref{thm:equilibrium-state-existence}, there exist a measure $\mu_K \in \MM_\sigma(\Sigma)$ and a constant $C \ge 1$ such that
  \begin{equation*}
    C^{-1}\fii^s(A_\iii) \le \mu_K([\iii]) \le C\fii^s(A_\iii)
  \end{equation*}
  for all $\iii \in \Sigma_*$. The claim follows.
\end{proof}

Let us now begin to state the theorems which were concisely summarized in the beginning of the introduction. Our first result determines the lower dimension of self-affine sets. It will be proved in Section \ref{sec:lower}. We emphasize that the result does not assume domination.

\begin{theorem} \label{thm:diml}
	If $X$ is a strictly affine strongly irreducible planar self-affine set satisfying the strong separation condition, then $\diml(X) = \min\{ 1, \dimh(X) \}$.
\end{theorem}

We remark that a Bedford-McMullen carpet $X$ with $\dimh(X) \le 1$ not having uniform vertical fibers serves as a counter-example for the above result in the reducible case; see Theorem \ref{thm:BM-carpet-regular} and, more precisely, \cite[Corollary 2.14]{Fraser14}. Recalling Lemma \ref{thm:ahlfors-implications}, it follows from Theorem \ref{thm:diml} that if $\dimh(X)>1$, then $X$ is not Ahlfors regular.

The following theorem determines the Assouad dimension of self-affine sets and it will be proved in Section \ref{sec:assouad-large}. It generalizes the result of B\'ar\'any, K\"aenm\"aki, and Rossi \cite[Theorem 3.2]{BaranyKaenmakiRossi2021} which uses a projection condition, a very restrictive assumption to guarantee that the projection of the self-affine set is a line segment for sufficiently many directions, to overcome several technical difficulties. Recall also the result of Fraser \cite[Theorem 2.12]{Fraser14} for self-affine carpets.

\begin{theorem} \label{thm:dima}
	If $X$ is a dominated planar self-affine set satisfying the strong separation condition such that $\dimh(X) \ge 1$ and $X_F$ is not a singleton, then
  \begin{equation*}
    \dima(X) = 1 + \max_{\atop{x\in X}{V\in X_F}}\dimh(X\cap(V+x)) < 2.
  \end{equation*}
\end{theorem}

In the following example, we show that, under the assumptions of Theorem \ref{thm:dima}, it is possible to have $\diml(X) < \dimh(X) = \dimaff(X) < \dima(X)$. This observation answers one of the open questions posed in \cite[Question 17.5.2]{Fraser2020}. The example strongly relies on known results on Bedford-McMullen carpets and, to our knowledge, it is currently the only example demonstrating $\dimaff(X) < \dima(X)$ in the case $\dimh(X) \ge 1$ under irreducibility. If $\dimh(X) < 1$, then this phenomenon is studied in more detail in Corollary \ref{thm:affinity-assouad}.

\begin{example} \label{ex:dimaff-dima}
  Let $q > p \ge 2$ and $N \in \{2,\ldots,pq\}$ be integers, and $I \subset \{0,\ldots,p-1\} \times \{0,\ldots,q-1\}$ a set of $N$ elements. Let $A = \diag(\tfrac{1}{p},\tfrac{1}{q})$ and $B_\eps \in \GL_2(\R)$ be a matrix with positive entries such that $\|B_\eps\| < \eps$ for all $\eps > 0$. It follows that $\fii^s(B_\eps) \downarrow 0$ as $\eps \downarrow 0$ for all $s \ge 0$. Furthermore, the tuple $(A,B_\eps) \in \GL_2(\R)^2$ is dominated and irreducible for all $\eps > 0$. Write $\A = (A,\ldots,A) \in \GL_2(\R)^N$ and note that $P(\A,s) = \log(N\fii^s(A))$. Hence,
  \begin{equation*}
    \dimaff(\A) =
    \begin{cases}
      \frac{\log N}{\log p}, &\text{if } N \in \{2,\ldots,p\}, \\
      1 + \frac{\log N/p}{\log q}, &\text{if } N \in \{p+1,\ldots,pq\}.
    \end{cases}
  \end{equation*}
  Observe also that if $\A_\eps = (A,\ldots,A,B_\eps) \in \GL_2(\R)^{N+1}$, then
  \begin{align*}
    P(\A_\eps,s) &= \lim_{n \to \infty} \frac{1}{n}\log\sum_{k=0}^n \sum_{1 \le i_1 < \cdots < i_k \le n} N^{n-k} \fii^s(A^{i_1-1}B_\eps A^{i_2-i_1-1}B_\eps \cdots B_\eps A^{n-i_k-1}) \\
    &\le \lim_{n \to \infty} \frac{1}{n}\log\sum_{k=0}^n \sum_{1 \le i_1 < \cdots < i_k \le n} N^{n-k} \fii^s(A)^{n-k} \fii^s(B_\eps)^k \\
    &= \log(N\fii^s(A)+\fii^s(B_\eps)).
  \end{align*}
  Therefore, if $s(\eps)$ is such that $\log(N\fii^{s(\eps)}(A)+\fii^{s(\eps)}(B_{\eps})) = 0$, we see that $\dimaff(\A) \le \dimaff(\A_\eps) \le s(\eps)$ for all $\eps>0$ and $s(\eps) \downarrow \dimaff(\A)$ as $\eps \downarrow 0$.

  Recall that the Bedford-McMullen carpet is the self-affine set $X \subset [0,1]^2$ associated to a tuple $(\fii_1,\ldots,\fii_N)$ of affine maps which all have the same linear part $A$ and the translation part is from the set $\{(\tfrac{j}{p},\tfrac{k}{q}) \in [0,1]^2 : (j,k) \in I\}$. We assume that $X$ satisfies the strong separation condition. Write $n_j = \#\{k : (j,k) \in I\}$ to denote the number of sets $\fii_i([0,1)^2)$ the vertical line $\{(\tfrac{j}{p},y) : y \in \R\}$ intersects. By \cite[Theorem 1.1]{Mackay2011}, we have
  \begin{equation*}
    \dima(X) = \frac{\log \#\{j \in \{1,\ldots,p\} : n_j \ne 0\}}{\log p} + \max_{j \in \{1,\ldots,p\}} \frac{\log n_j}{\log q}.
  \end{equation*}
  For example, if $q=5$, $p=4$, and $N=5$, then, by choosing the translation vectors such that $n_1=3$, $n_2=0$, and $n_3=1=n_4$, we have $1 < \dimaff(\A) < \dima(X)$.
  Observe that several other choices also lead to these strict inequalities. We may now choose $\eps>0$ such that $s(\eps) < \dima(X)$.

  Define a contractive affine map $\fii_{N+1} \colon \R^2 \to \R^2$ by setting $\fii_{N+1}(x) = B_\eps x + v$, where $v \in \R^2$ is chosen such that $\fii_{N+1}([0,1]^2) \cap \bigcup_{i=1}^N \fii_i([0,1]^2) = \emptyset$. Let $X'$ be the dominated irreducible planar self-affine set satisfying the strong separation condition associated to the tuple $(\fii_1,\ldots,\fii_N,\fii_{N+1})$. By Theorems \ref{thm:BHR} and \ref{thm:diml}, we have $\dimh(X') = \dimaff(\A_\eps) \ge \dimaff(\A) > 1 = \diml(X')$. Therefore, as $X \subset X'$,
  \begin{equation*}
    \diml(X') < \dimh(X') = \dimaff(\A_\eps) \le s(\eps) < \dima(X) \le \dima(X')
  \end{equation*}
  as wished.
\end{example}

We say that a strictly affine planar self-affine set $X$ satisfies a \emph{projective open set condition} if there is an open and bounded non-empty set $U \subset \R^2$ such that $\fii_i(U) \subset U$ for all $i \in \{1,\ldots,N\}$ and $\proj_{V^\bot}(\fii_i(U)) \cap \proj_{V^\bot}(\fii_j(U)) = \emptyset$ for all $V \in X_F$ whenever $i \ne j$. Inspired by Theorem \ref{thm:self-similar-osc-equiv}, we say that $X$ satisfies a \emph{projective separation} if there exists $\eta>0$ such that for every $V \in X_F$ and $\iii,\jjj \in \Sigma_*$ with $\iii\ne\jjj$ it holds that
\begin{equation} \label{eq:POSC}
  |\proj_{V^\perp} \circ \fii_\iii-\proj_{V^\perp} \circ \fii_\jjj| \ge \eta\diam(\proj_{V^\perp}(\fii_\iii(X))).
\end{equation}
In Propositions \ref{thm:POSC-PSC} and \ref{lem:missing}, we show that the projective open set condition implies the projective separation. One has to be careful when using \eqref{eq:POSC} as, for example, a Bedford-McMullen carpet can be contained in a line parallel to the sole element of $X_F$ in which case the right hand-side in \eqref{eq:POSC} is zero. Nevertheless, if $X$ and $X_F$ are not singletons, then there always exists $V \in X_F$ such that $\diam(\proj_{V^\perp}(X)) > 0$. In fact, in \eqref{eq:posdiam}, we will see that if $X$ is dominated with at least two points and $X_F$ is not a singleton, then the right hand-side in \eqref{eq:POSC} is uniformly bounded away from zero. We also define
\begin{equation} \label{eq:POSC-Sigma-def}
\begin{split}
  \Sigma(V,x,r) = \{\iii\in\Sigma_* : \;&\diam(\proj_{V^\perp}(\fii_\iii(X))) < r \le \diam(\proj_{V^\perp}(\fii_{\iii^-}(X))) \\
  &\text{ and }\proj_{V^\perp}(\fii_\iii(X)) \cap B(\proj_{V^\perp}(x),r) \ne \emptyset\}
\end{split}
\end{equation}
for all $V \in X_F$, $x \in X$, and $r>0$. The projective separation is characterized in the following theorem. The result is analogous to Theorems \ref{thm:self-similar-osc-equiv} and \ref{thm:self-similar-wsc-equiv} in the self-similar case. Its proof is given in Sections \ref{sec:assouad-small} and \ref{sec:hausdorff-measure}.

\begin{theorem} \label{thm:ahl}
	If $X$ is a dominated planar self-affine set satisfying the strong separation condition such that $X_F$ is not a singleton and $\dimh(X) \le 1$, then the following conditions are equivalent:
	\begin{enumerate}
    \item\label{(1)} $X$ satisfies the projective separation.
    \item\label{(11)} $\sup\{\#\Sigma(V,x,r) : V \in X_F, \text{ } x \in X, \text{ and } r>0\} < \infty$,
    \item\label{(4)} $X$ is Ahlfors regular,
		\item\label{(2)} $\HH^{s}(X)>0$ where $s=\dimh(X)$,
    \item\label{(3)} $\min_{V \in X_F} \HH^{s}_\infty(\proj_{V^\perp}(X)) > 0$ where $s=\dimh(X)$,
    \item\label{(5)} $\proj_{V^\perp}(X)$ is Ahlfors regular for all $V \in X_F$.
	\end{enumerate}
  Furthermore, if $\dimh(X)<1$, then the following condition can be added to the list:
  \begin{enumerate}[resume]
    \item\label{(6)} $\diml(X)=\dimh(X)=\dima(X)$.
  \end{enumerate}
\end{theorem}

If the self-affine set $X$ has positive Hausdorff measure, then Theorem \ref{thm:ahl} guarantees that there are no exact overlaps in the projections onto orthogonal complements of the Furstenberg directions. In other words, if under the assumptions of Theorem \ref{thm:ahl} it holds that $\HH^s(X)>0$ for $s=\dimh(X)$, then $\proj_{V^\perp}\fii_\iii \ne \proj_{V^\perp}\fii_\jjj$ for all $V \in X_F$ and $\iii,\jjj \in \Sigma_*$ with $\iii \ne \jjj$.

\begin{theorem}\label{thm:dima2}
	If $X$ is a dominated planar self-affine set satisfying the strong separation condition, but not the projective separation, such that $X_F$ is not a singleton, then $\dima(X) \ge 1$.
\end{theorem}

It would be interesting to know if the above theorem can be improved to show $\dima(X) > 1$. The theorem would then be analogous to the result of Garc\'ia \cite[Theorem 1.4]{Garcia2020} in the self-similar case; recall also Theorem \ref{thm:self-similar-wsc-equiv}.

Our final result, Theorem \ref{thm:baire-typical}, shows that in a topological sense typical self-affine sets satisfy the assumptions of Theorem \ref{thm:dima2}. The proof of Theorem \ref{thm:baire-typical} will be postponed until Section \ref{sec:assouad-affinity}. Let $\A = (A_1,\ldots,A_N) \in \GL_2(\R)^N$ and consider the affine iterated function systems $\Phi_{\mathsf{v}} = (A_1+v_1,\ldots,A_N+v_N)$ parametrized by the translation vector $\mathsf{v} = (v_1,\ldots,v_N) \in (\R^2)^N$. Let $\pi_{\mathsf{v}} \colon \Sigma \to X_{\mathsf{v}}$ be the associated canonical projection onto the self-affine set $X_{\mathsf{v}}$. If $\A$ is strictly affine, fix $\delta>0$ and define
\begin{equation} \label{eq:baire-N-definition}
\begin{split}
  \NN(\A) = \NN_\delta(\A) = \{\mathsf{v} \in (\R^2)^N : \;&X_{\mathsf{v}} \text{ satisfies the strong separation condition } \\
  &\text{with } \delta>0 \text{ in \eqref{eq:SSC} and there are } V \in X_F \\
  &\text{and } \iii,\jjj \in \Sigma \text{ with } \iii|_1 \ne \jjj|_1 \text{ such that}  \\
  &\proj_{V^\bot}(\pi_{\mathsf{v}}(\iii)) = \proj_{V^\bot}(\pi_{\mathsf{v}}(\jjj))\}.
\end{split}
\end{equation}
Since the same $\delta>0$ works for all $\mathsf{v} \in \NN(\A)$ in \eqref{eq:SSC}, it is easy to see that $\NN(\A)$ is complete. Recall that a \emph{residual} set is an intersection of countably many sets with dense interiors.

\begin{theorem} \label{thm:baire-typical}
  If $\A \in \GL_2(\R)^N$ is strictly affine such that $\max_{i \in \{1,\ldots,N\}} \|A_i\|<\frac12$, then there exists a residual set $\RR(\A) \subset \NN(\A)$ such that for each $\mathsf{v} \in \RR(\A)$ the planar self-affine set $X_{\mathsf{v}}$ does not satisfy the projective separation.
\end{theorem}

Let us next prove Corollaries \ref{thm:ahlfors-equiv}--\ref{thm:affinity-assouad} by relying on the above stated results:

\begin{proof}[Proof of Corollary \ref{thm:ahlfors-equiv}]
  Since $X$ is irreducible and dominated, it is also strongly irreducible; see Lemma \ref{thm:irreducible-dominated}. Therefore, it follows from Theorem \ref{thm:diml} and Lemma \ref{thm:ahlfors-implications} that if $s = \dimh(X) > 1$, then $X$ is not Ahlfors $s$-regular. Recalling Lemma \ref{thm:ahlfors-implications}, we have thus proven that if $X$ is Ahlfors $s$-regular, then $0 \le s \le 1$ and $0<\HH^s(X)<\infty$. Conversely, as $X$ is strongly irreducible and dominated, Theorem \ref{thm:ahl} shows that $\HH^s(X)>0$ where $s=\dimh(X) \le 1$ implies that $X$ is Ahlfors $s$-regular.
\end{proof}

\begin{proof}[Proof of Corollary \ref{thm:assouad-proj}]
  By \eqref{eq:orponen}, we have $\dima(\proj_{V^\bot}(X)) \ge \min\{1,\dima(X)\}$ for all $V \in \RP \setminus E$, where the set $E \subset \RP$ satisfies $\dimh(E)=0$. If $\dima(X) \ge 1$, then, as $\dima(\proj_{V^\bot}(X)) \le 1$ for all $V \in \RP$, the claim is a direct consequence of this result. We may thus assume that $\dimh(X) \le 1$. If $X$ is not Ahlfors regular, then Theorem \ref{thm:ahl} implies that $X$ does not satisfy the projective separation. Therefore, by Theorem \ref{thm:dima2}, we have $\dima(X) \ge 1$ and we have shown the first claim. Furthermore, if $X$ is Ahlfors regular, then, by Theorem \ref{thm:ahl}, $\proj_{V^\bot}(X)$ is Ahlfors regular for all $V \in X_F$. Therefore, Lemma \ref{thm:ahlfors-implications} and the fact that the Hausdorff dimension cannot increase under Lipschitz maps, gives $\dima(\proj_{V^\bot}(X)) = \dimh(\proj_{V^\bot}(X)) \le \min\{1,\dimh(X)\} \le \min\{1,\dima(X)\}$ for all $V \in X_F$ and finishes the proof.
\end{proof}

\begin{proof}[Proof of Corollary \ref{thm:slices}]
  This is a direct consequence of Theorem \ref{thm:dima}.
\end{proof}

\begin{proof}[Proof of Corollary \ref{thm:affinity-assouad}]
  Let $\A \in \GL_2(\R)^N$ be dominated such that $\max_{i \in \{1,\ldots,N\}} \|A_i\|<\frac12$, $X_F$ is not a singleton, and $\dimaff(\A)<1$. Let $\mathsf{v} \in (\R^2)^N$ be a translation vector such that $X_{\mathsf{v}}$ satisfies the strong separation condition. If $X_{\mathsf{v}}$ satisfies the projective separation, then, by Theorem \ref{thm:ahl}, $X_{\mathsf{v}}$ is Ahlfors regular. On the other hand, if $X_{\mathsf{v}}$ does not satisfy the projective separation, then, by Theorem \ref{thm:dima2}, $\dimaff(\A) < 1 \le \dima(X_{\mathsf{v}})$. By Theorem \ref{thm:baire-typical}, there is a residual set $\RR(\A) \subset \NN(\A)$ such that for every $\mathsf{v} \in \RR(\A)$ the associated planar self-affine set $X_{\mathsf{v}}$ satisfies the strong separation condition, but not the projective separation. Therefore, $\dimaff(\A) < 1 \le \dima(X_{\mathsf{v}})$ for all $\mathsf{v} \in \RR(\A)$. As $\RR(\A)$ is residual, we see that $\RR(\A) = \bigcap_{q \in \N} \RR_q(\A)$, where $\RR_q(\A)^o$ is dense in $\NN(\A)$. If $\RR(\A)$ is countable, say $\RR(\A) = \{\mathsf{v}_1,\mathsf{v}_2,\ldots\}$, then we have $\bigcap_{q,k \in \N} \RR_q(\A)^o \cap (\NN(\A) \setminus \{\mathsf{v}_k\}) = \emptyset$ contradicting the Baire category theorem. Therefore, $\RR(\A)$ is uncountable.
  \begin{figure}[t]
  \begin{tikzpicture}[scale=2.8]
    \draw [->] (-1.5,0) -- (1.5,0);
    \draw [->] (0,-1.2) -- (0,1.2);
    \draw (0,0) circle (1);
    \begin{scope}[shift={(-0.45,-0.5)}]
      \filldraw[black!10, rotate=70] (0,0) ellipse (0.32 and 0.15);
      \draw[rotate=70] (0,0) ellipse (0.32 and 0.15);
    \end{scope}
    \begin{scope}[shift={(0.05,0.2)}]
      \filldraw[black!10, rotate=15] (0,0) ellipse (0.3 and 0.1);
      \draw[rotate=15] (0,0) ellipse (0.3 and 0.1);
    \end{scope}
    \begin{scope}[shift={(0.55,0.55)}]
      \filldraw[black!10, rotate=50] (0,0) ellipse (0.2 and 0.1);
      \draw[rotate=50] (0,0) ellipse (0.2 and 0.1);
    \end{scope}
    \draw[dashed] (-0.63,{-0.81-0.4}) -- (-0.63,-0.19) -- ({-0.27+0.4},-0.19);
    \draw[dashed] ({-0.63-0.4},-0.81) -- (-0.27,-0.81) -- (-0.27,{-0.19+0.4});

    \draw[dashed] (-0.24,{0.075-0.4}) -- (-0.24,0.33) -- ({0.34+0.4},0.33);
    \draw[dashed] ({-0.24-0.4},0.075) -- (0.34,0.075) -- (0.34,{0.33+0.4});

    \draw[dashed] (0.4,{0.38-0.4}) -- (0.4,0.72) -- ({0.7+0.4},0.72);
    \draw[dashed] ({0.4-0.4},0.38) -- (0.7,0.38) -- (0.7,{0.72+0.4});
  \end{tikzpicture}
  \caption{Schematic construction used in the proof of Corollary \ref{thm:affinity-assouad}. The first-level pieces are arranged so that projected overlaps are uniformly excluded in Furstenberg directions, which yields projective separation and then Ahlfors regularity via Theorem \ref{thm:ahl}.} \label{fig:ahlfors}
  \end{figure}

  It remains to show that such self-affine sets with projective separation have non-empty interior in the set of parameters formed by the translation vectors and elements of the matrices. Fix a dominated tuple $\A=(A_1,\ldots,A_N) \in \GL_2(\R)^N$ and let $\CC \subset \RP$ be a strongly invariant multicone for $\A$. Let us choose the tuple $\mathsf{v}=(v_1,\ldots,v_N)$ of translation vectors such that $v_i\in B(0,1)^o$ and $\mathrm{span}(v_i-v_j)\in\CC^o$ whenever $i \ne j$. Thus, by choosing sufficiently small positive constants $c_1,\ldots,c_N$ and defining $\fii_i \colon \R^2 \to \R^2$ by setting $\varphi_i(x)=c_iA_ix+v_i$, we have $\varphi_i(B(0,1))\subset B(0,1)^o$ for all $i \in \{1,\ldots,N\}$ and $\varphi_i(x)-\varphi_j(y)\in\CC^o$ for all $x,y \in B(0,1)$ whenever $i \ne j$. We consider the self-affine set associated to $(\fii_1,\ldots,\fii_N)$. It is evident that the mentioned properties are open in the set of parameters. Moreover, since $X_F\subset\RP\setminus\CC$ by compactness, there exists a constant $\eta>0$ such that
  \begin{equation*}
    |\proj_{V^\perp}(\varphi_i(x))-\proj_{V^\perp}(\varphi_j(y))| \ge \eta
  \end{equation*}
  for all $V\in X_F$, $i,j \in \{1,\ldots,N\}$ with $i \ne j$, and $x,y \in X$. The projective separation \eqref{eq:POSC} follows. For illustration in the case where all the matrices have positive entries, see Figure \ref{fig:ahlfors}.
\end{proof}

\section{Lower dimension of self-affine sets} \label{sec:lower}

In this section, we prove Theorem \ref{thm:diml}. The proof relies on analysis on weak tangent sets and it is worth emphasizing that it does not require domination. The upper bound is done in Proposition \ref{lem:ublow} below. The main idea is to find more and more narrow parts of the self-affine set which eventually result in a weak tangent set contained in a line. Proposition \ref{lem:lblow} below gives the lower bound. Since, by Theorem \ref{thm:BHR}, the dimension of a self-affine set is preserved under projections, the task there is to compare weak tangent sets to projections. Let us denote the convex hull of a set $A \subset \R^2$ by $\conv(A)$ and the boundary of $A$ by $\partial A$. We say that a set $A \subset \R^2$ has \emph{positive length} if $\HH^1(A) > 0$.

\begin{lemma}\label{lem:linseg}
	If $X$ is a planar self-affine set satisfying the strong separation condition, then $\partial\conv(X)$ contains at most countably many line segments with positive length.
\end{lemma}

\begin{proof}
Write $S(V,t) = (V+t)\cap\partial\conv(X)$ for all $t \in V^\perp$ and $V \in \RP$. Let
\begin{equation}\label{eq:lineseg}
  \mathcal{I} = \{(V,t)\in\RP\times\R^2 : \HH^1(S(V,t)) > 0 \text{ and } t \in V^\perp\}
\end{equation}
and observe that if $(V,t) \in \mathcal{I}$, then, by the convexity of $\conv(X)$, $S(V,t)$ is a proper line segment in $\partial\conv(X)$. Let us show that $\mathcal{I}$ is a countable set. If $W\in\RP$ is the $x$-axis, then, again by the convexity of $\conv(X)$, it is easy to see that there are at most two points $t_1,t_2 \in W$ such that $(W^\perp,t_1)\in\mathcal{I}$ and $(W^\perp,t_2)\in\mathcal{I}$. It is thus enough to show that $\mathcal{I}'=\{(V,t)\in\mathcal{I} : V\neq W^\perp\}$ is a countable set.

Observe that $\proj_W(S(V,t))$ has positive length for all $(V,t)\in\mathcal{I}'$. Thus, for every $(V,t)\in\mathcal{I}'$ there exists $q\in\Q$ such that $q\in\proj_W(S(V,t))$. Moreover, using the convexity of $\conv(X)$ and the definition of $\mathcal{I}'$ once more, for every $q\in\Q$ there are at most two $(V_1,t_1),(V_2,t_2)\in\mathcal{I}'$ such that $q\in\proj_W(S(V_j,t_j))$ for both $j\in \{1,2\}$. Hence, $\#\mathcal{I}'$ is indeed at most countable.
\end{proof}

\begin{lemma}\label{lem:seppoint}
	If $X$ is a planar self-affine set satisfying the strong separation condition, then for every except possibly countably many $V\in\RP$ there exists $y\in X$ such that $(V+y)\cap X=\{y\}$ and $X\setminus\{y\}$ is contained in one of the open half-planes defined by $V+y$.
\end{lemma}

\begin{proof}
  Let $\mathcal{S}=\{V\in\RP : (V+t)\cap\partial\conv(X)$ has positive length for some $t \in V^\perp\}$ and notice that $\mathcal{S}$ is at most countable by Lemma \ref{lem:linseg}. Fix $V\in\RP\setminus\mathcal{S}$ and let $v\in V^\perp$ be such that $|v|=1$. Since the set $\proj_{V^\perp}(X)$ is compact, there exist unique $t_1,t_2\in\R$ such that $\proj_{V^\perp}(X)\subset\{tv : t_1\leq t\leq t_2\}$ and $t_1v, t_2v \in \proj_{V^\perp}(X)$. In particular, $\{tv : t_1\leq t\leq t_2\}=\proj_{V^\perp}(\conv(X))$. To finish the proof, it suffices to show that $\proj_{V^\perp}^{-1}(t_1v)\cap X$ is a singleton.
	
	Suppose to the contrary that there are $x_1, x_2\in\proj_{V^\perp}^{-1}(t_1v)\cap X$ such that $x_1 \ne x_2$. But then $x_1,x_2\in\partial\conv(X)$ and in particular, the line segment connecting $x_1,x_2$ must be also contained in $\partial\conv(X)$. Since this line segment is parallel to $V$, this is a contradiction as $V\notin\mathcal{S}$.
\end{proof}

For $x \in \R^2$, $V \in \RP$, $v \in V$ such that $|v|=1$, and $0 \le \delta \le 1$, we set
\begin{align*}
  C(x,v,\delta) &= \{y \in \R^2 : (y-x) \cdot v > \sqrt{1-\delta^2}|y-x|\}, \\
  H(x,V,\delta,\eps) &= C(x+\eps v,v,\delta) \cup C(x-\eps v,-v,\delta).
\end{align*}
For illustration, see Figure \ref{fig:almostcone}. Observe that $\bigcap_{\delta>0} H(x,V,\delta,\eps) = (V+x) \setminus B(x,\eps)$.

\begin{figure}[t]
\begin{tikzpicture}[scale=1.0]
\begin{scope}[rotate=10]
  \draw[fill=black!10!white, draw=white] (4,1) -- (1,0) -- (4,-1);
  \draw[fill=black!10!white, draw=white] (-4,1) -- (-1,0) -- (-4,-1);
  \draw[dotted] (-5,0) -- (5,0) node [right] {$V+x$};
  \filldraw (0,0) circle (1pt) node [below=3] {$x$};
  \draw (4,1) -- (1,0) -- (4,-1);
  \draw (-4,1) -- (-1,0) -- (-4,-1);
  \node at (0.5,0.3) {$\eps$};
  \node at (4.3,0.5) {$\delta$};
\end{scope}
\end{tikzpicture}
\caption{Illustration of $H(x,V,\delta,\eps)$ as the union of two opposite cones around the line $V+x$, separated from $x$ by the offset $\eps$. As $\delta \downarrow 0$, these cones collapse to the punctured line $(V+x)\setminus B(x,\eps)$.} \label{fig:almostcone}
\end{figure}

\begin{lemma}\label{lem:almostcone}
  If $X$ is a planar self-affine set satisfying the strong separation condition, then there exists an at most countable set $\mathcal{S} \subset \RP$ such that for every $V\in\RP\setminus \mathcal{S}$ there exist $y\in X$ such that for every $\eps>0$ there is $\Delta>0$ such that $X \subset \R^2 \setminus H(y,V,\delta,\varepsilon)$ for all $0<\delta<\Delta$.
\end{lemma}

\begin{proof}
  Let $\mathcal{S}$ be the at most countable subset of $\RP$ for which the conclusion of Lemma \ref{lem:seppoint} fails. Let us argue by contradiction that there exists $V\in\RP\setminus\mathcal{S}$ such that for every $y \in X$ there is $\eps>0$ so that for every  $\Delta>0$ there is $0<\delta<\Delta$ for which $X \cap H(y,V,\delta,\eps)\neq\emptyset$. Observe that this implies $X \cap H(y,V,\Delta,\eps)\neq\emptyset$. Let $y \in X$ be the point given by Lemma~\ref{lem:seppoint} such that $(V+y)\cap X=\{y\}$ and $\eps>0$ as above. 
  
  For each $n \in \N$ let us now choose a point $x_n \in X\cap H(y,V,\tfrac{1}{n},\eps)$. By the compactness of $X$, going into a subsequence if required, the sequence $(x_n)_{n \in \N}$ converges to a point $x \in X$. As the sequence of sets $(H(y,V,\tfrac{1}{n},\eps))_{n \in \N}$ is decreasing, we must have $x\in\bigcap_{n \in \N} H(y,V,\tfrac{1}{n},\eps) = (V+y) \setminus B(y,\varepsilon)$. Therefore, $y \ne x \in (V+y) \cap X$ which is a contradiction.
\end{proof}

We remark that in general it is not possible to choose $\eps=0$ above as a planar self-affine set can be contained for example in a parabola; see \cite{BandtKravchenko2011,FengKaenmaki2018}. The following proposition gives the upper bound in Theorem~\ref{thm:diml}.

\begin{proposition} \label{lem:ublow}
	 If $X$ is a strictly affine strongly irreducible planar self-affine set satisfying the strong separation condition, then $\diml(X)\leq \min\{ 1, \dimh(X) \}$.
\end{proposition}

\begin{proof}
  Since $\diml(X) \le \dimh(X)$, it is enough to show that $\diml(X) \le 1$. Therefore, by Lemma \ref{thm:FHKY}, it suffices to show that there exists a weak tangent set $T$ such that $\dimh(T) \le 1$. Let $J\subset\Sigma_*$ be the finite subset defined in Lemma~\ref{lem:subsystem} such that the tuple $(A_{\iii})_{\iii\in J}$ is dominated and strongly irreducible, and denote the set of its Furstenberg directions by $\widetilde{X}_F\subset X_F$. Let $\widetilde{\Pi}\colon J^\N\to\widetilde{X}_F$ be the canonical projection for the tuple $(A_{\iii})_{\iii\in J}$ defined in \eqref{eq:furstenberg-directions-dominated2}. We embed $J^\N$ into $\Sigma$ in the natural way.
  


  Let $\mathcal{S} \subset \RP$ be the at most countable exceptional set of Lemma \ref{lem:almostcone}. Since $\widetilde{X}_F$ is perfect, it is uncountable, and so, there exists $\iii\in J^\N\subset\Sigma$ such that $V:=\widetilde{\Pi}(\iii)\in\widetilde{X}_F\setminus\mathcal{S}$. Let $n_k$ be the sequence such that $\sigma^{n_k}\iii\in J^\N$. By Lemma~\ref{thm:bochi-morris}, there exists a constant $C>0$ such that for every $k\in\N$
  \begin{equation}\label{eq:compare}
    D^{-1}\alpha_2(A_{\overleftarrow{\iii|_{n_k}}})\leq\|A_{\overleftarrow{\iii|_{n_k}}}|V\|\leq D\alpha_2(A_{\overleftarrow{\iii|_{n_k}}}),
  \end{equation}
  moreover, by Lemma~\ref{lem:goto0}, 
  \begin{equation}\label{eq:compare2}
    \lim_{k\to\infty}\sphericalangle(A_{\overleftarrow{\iii|_{n_k}}}W_1,A_{\overleftarrow{\iii|_{n_k}}}W_2)=0
  \end{equation}
  for all $W_1,W_2\in\RP\setminus\{V\}$. Let $y\in X$ be as in Lemma \ref{lem:almostcone}. Let $d>0$ be the constant of \eqref{eq:SSC} and write
  \begin{equation} \label{eq:neigh}
  \begin{split} 
    T_k &= M_{\fii_{\overleftarrow{\iii|_{n_k}}}(y),\alpha_2(A_{\overleftarrow{\iii|_{n_k}}})d}(X) \cap B(0,1) \\
    &= M_{\fii_{\overleftarrow{\iii|_{n_k}}}(y),\alpha_2(A_{\overleftarrow{\iii|_{n_k}}})d}(X \cap B(\fii_{\overleftarrow{\iii|_{n_k}}}(y),\alpha_2(A_{\overleftarrow{\iii|_{n_k}}})d)) 
  \end{split}
  \end{equation}
  for all $n \in \N$. Relying on the strong separation condition, we have
  \begin{equation*}
    X \cap B(\fii_{\overleftarrow{\iii|_{n_k}}}(y),\alpha_2(A_{\overleftarrow{\iii|_{n_k}}})d) = \fii_{\overleftarrow{\iii|_{n_k}}}(X) \cap B(\fii_{\overleftarrow{\iii|_{n_k}}}(y),\alpha_2(A_{\overleftarrow{\iii|_{n_k}}})d).
  \end{equation*}
  By going into a sub-sequence, if necessary, we see that there is $T \in \Tan(X)$ such that $T_k \to T$ in Hausdorff distance.

  Let $\eps>0$. By Lemma \ref{lem:almostcone}, there exists $\delta>0$ such that
  \begin{equation*}
    X \subset \R^2 \setminus H(y,V,\delta,\varepsilon)
  \end{equation*}
  and therefore,
  \begin{equation} \label{eq:neigh3}
    \fii_{\overleftarrow{\iii|_{n_k}}}(X) \subset \fii_{\overleftarrow{\iii|_{n_k}}}(y) + A_{\overleftarrow{\iii|_{n_k}}}(\R^2 \setminus H(0,V,\delta,\eps)).
  \end{equation}
  Hence, by \eqref{eq:neigh} and \eqref{eq:neigh3}, we have
  \begin{equation} \label{eq:neigh2}
  \begin{split}
    T_n & = M_{\fii_{\overleftarrow{\iii|_{n_k}}}(y),\alpha_2(A_{\overleftarrow{\iii|_{n_k}}})d}(\fii_{\overleftarrow{\iii|_{n_k}}}(X) \cap B(\fii_{\overleftarrow{\iii|_{n_k}}}(y),\alpha_2(A_{\overleftarrow{\iii|_{n_k}}})d)) \\
    &\subset \frac{1}{\alpha_2(A_{\overleftarrow{\iii|_{n_k}}})d} A_{\overleftarrow{\iii|_{n_k}}}(\R^2 \setminus H(0,V,\delta,\eps)) \cap B(0,1).
  \end{split}
  \end{equation}
  Let $W_1, W_2\in\RP \setminus \{V\}$ be parallel to the lines appearing in $\partial H(0,V,\delta,\eps)$.  By \eqref{eq:compare2}, there is $W\in\RP$ such that, by possibly going into a sub-sequence once more, $A_{\overleftarrow{\iii|_{n_k}}} W_1 \to W$ and $A_{\overleftarrow{\iii|_{n_k}}} W_2 \to W$ as $n \to \infty$. Therefore, by \eqref{eq:compare} and \eqref{eq:neigh2}, $T$ is contained in a $\frac{2D\eps}{d}$-neighbourhood of $W$. Since $\eps>0$ was arbitrary, we see that there exists a subspace $W'$ such that $T$ is contained in $W'$ which completes the proof.
\end{proof}

By a rank of an affine map, we mean the rank of its linear part. We use the same convention also for the kernel and image. Let us state a useful lemma which is not stated explicitly but is contained in the proof of \cite[Theorem~5.2]{BaranyKaenmakiRossi2021}.

\begin{lemma} \label{lem:embedding}
  If $X$ is a planar self-affine set satisfying the strong separation condition, then for every $T\in\Tan(X)$ there exist affine maps $G_1,G_2\colon \R^2 \to \R^2$ having rank at least one such that $G_1(X)\subset T$ and $G_2(T) \subset X$. In particular, if $X$ is dominated, then $G_1$ and $G_2$ both have rank one and $\im(G_2) \in X_F$.
\end{lemma}

\begin{proof}
  Let $(\iii_k)_{k \in \N}$ be a sequence of infinite words in $\Sigma$ and $(r_k)_{k \in \N}$ be a sequence of positive real numbers converging to zero such that $M_{\pi(\iii_k),r_k}(X) \cap B(0,1) \to T$ in Hausdorff distance. For each $k \in \N$, choose $n_k \in \N$ such that $\alpha_1(A_{\iii_k|_{n_k}}) \leq r_k < \alpha_1(A_{\iii_k|_{n_k-1}})$. The affine map $G_1$ is now obtained as an accumulation point of the sequence $M_{\pi(\iii_k),r_k}\circ\varphi_{\iii_k|_{n_k}}$. Similarly, for each $k \in \N$, choose $m_k \in \N$ such that $\alpha_2(A_{\iii_k|_{m_k}}) \leq r_k < \alpha_2(A_{\iii_k|_{m_k-1}})$. The affine map $G_2$ is now obtained as an accumulation point of the sequence $\varphi_{\iii_k|_{m_k}}^{-1} \circ M_{\pi(\iii_k),r_k}^{-1}$. To see that $G_1(X)\subset T$ and $G_2(T) \subset X$, consult the proof of \cite[Theorem~5.2]{BaranyKaenmakiRossi2021}.

  If $X$ is dominated, then by the definition of domination, $G_1$ and $G_2$ are of rank one. To see that $\im(G_2) \in X_F$, follow the proof of Lemma \ref{thm:nilpotent}.
\end{proof}

Finally, the following proposition gives the lower bound in Theorem \ref{thm:diml}.

\begin{proposition} \label{lem:lblow}
	 If $X$ is a strongly irreducible planar self-affine set satisfying the strong separation condition, then $\diml(X)\geq \min\{ 1, \dimh(X) \}$.
\end{proposition}

\begin{proof}
  By Theorem \ref{thm:BHR} and Lemma \ref{thm:FHKY}, it is enough to show that for every $T \in \Tan(X)$ there exists $V \in \RP$ such that
  \begin{equation*}
    \dimh(T) \ge \dimh(\proj_{V^\bot}(X)).
  \end{equation*}
  To that end, fix $T \in \Tan(X)$. By Lemma \ref{lem:embedding}, there exists an affine map $G\colon \R^2 \to \R^2$ with $\rank(G) \ge 1$ such that $G(X) \subset T$. If $\rank(G)=2$, then $\dimh(T) \ge \dimh(X) \ge \dimh(\proj_{V^\bot}(X))$ for any $V\in\RP$. If $\rank(G)=1$, then the linear part of $G$ is a projection as described in \eqref{eq:rank-one-projection-nilpotent}. In particular, $G(X)$ and $\proj_{\ker(G)^\perp}(X)$ are bi-Lipschitz equivalent. Thus, if $V = \ker(G)$, then $\dimh(T) \geq \dimh(G(X)) = \dimh(\proj_{V^\bot}(X))$ also in this case.
\end{proof}

\section{Assouad dimension of self-affine sets having large Hausdorff dimension} \label{sec:assouad-large}

In this section, we prove Theorem \ref{thm:dima}. By Lemma \ref{thm:KOR}, the task is to estimate the Hausdorff dimension of weak tangent sets. The central ingredients of the proof are Theorem \ref{thm:BHR2} and the Marstrand's slicing theorem. Let us, for the sake of illustration, assume that there exists a weak tangent set $T \in \Tan(X)$ such that $T = L \times W$ with $W$ and $L$ being compact subsets of the reals. If it can be shown that for each $V \in X_F$ and $x \in X$ the set $L$ contains an affine copy of $X \cap (V+x)$ and $W$ contains a projected copy of $X$, then Theorem \ref{thm:BHR2} gives $\dimh(W)=1$ and, consequently, $\dima(X) \ge \dimh(T) \ge 1 + \dimh(X \cap (V+x))$. However, the existence of weak tangent sets having such product structure is not true in general. Instead, we show that they have a quasi-product structure, that is, there exist distinct $W,L \in \RP$ so that $T \cap L$ contains an affine copy of $X \cap (V+x)$ and $T \cap (W+y)$ contains a projected copy of $X$ for all $y \in T$. The proof can then be concluded by the Marstrand's slicing theorem.

We begin with an auxiliary lemma which shows that all the slices of $X$ have dimension strictly smaller than one.

\begin{lemma} \label{lem:smallslice}
  If $X$ is a planar self-affine set satisfying the strong separation condition, then
  \begin{equation*}
    \sup_{\atop{x\in X}{V\in \RP}} \dimh(X\cap(V+x)) < 1.
  \end{equation*}
\end{lemma}

\begin{proof}
  Relying on the strong separation condition, let $\delta>0$ be as in \eqref{eq:SSC}. Then by compactness, let $U$ be the open set formed by the union of finitely many open $\frac{\delta}{3}$-balls centered at $X$ such that $X\subset U$. Note that $U$ has finitely many connected components. Define
  \begin{equation*}
    \Sigma_n(V,x)=\{\iii\in\Sigma_n : \varphi_\iii(X)\cap(V+x) \ne \emptyset\}
  \end{equation*}
  for all $x \in X$, $V \in \RP$, and $n \in \N$. Note that $\fii_\iii(U) \cap (V+x)$ consists of line segments where the number of line segments is bounded above by the number of the connected components of $U$ for all $\iii \in \Sigma_n(V,x)$. Fix $x \in X$ and $V \in \RP$. It is easy to see that, by the strong separation condition,
  \begin{equation} \label{eq:comp1}
    \HH^1(\fii_\iii(U) \cap (V+x)) \ge \frac{2\delta\|A_{\iii}|A_\iii^{-1}V\|}{3} = \frac{2\delta}{3\|A_\iii^{-1}|V\|}
  \end{equation}
  and
  \begin{equation} \label{eq:comp2}
    \HH^1(\fii_\iii(U) \cap (V+x)) \le \diam(\varphi_\iii(U)\cap(V+x)) \le \frac{3\diam(X)+2\delta}{3\|A_\iii^{-1}|V\|}
  \end{equation}
  for all $\iii \in \Sigma_n(V,x)$. Note that \eqref{eq:comp1} and \eqref{eq:comp2} together give
  \begin{equation} \label{eq:comp5}
    \HH^1(\fii_\iii(U) \cap (V+x)) \ge \frac{2\delta}{3\diam(X)+2\delta} \diam(\varphi_\iii(U)\cap(V+x))
  \end{equation}
  for all $\iii \in \Sigma_n(V,x)$.

  Write $M_\iii = \#\{j \in \{1,\ldots,N\} : \iii j \in \Sigma_{n+1}(V,x)\} \ge 1$ for all $\iii \in \Sigma_n(V,x)$ and notice that, by the strong separation condition and \eqref{eq:comp2},
  \begin{align*}
    \sum_{\iii j \in \Sigma_{n+1}(V,x)} \HH^1(\fii_{\iii j}(U) \cap (V+x)) &\le \HH^1(\fii_\iii(U) \cap (V+x)) - \frac{(M_\iii-1)\delta}{3\|A_\iii^{-1}|V\|} \\
    &\le \biggl( 1 - \frac{(M_\iii-1)\delta}{3\diam(X)+2\delta} \biggr) \HH^1(\fii_\iii(U) \cap (V+x))
  \end{align*}
  for all $\iii \in \Sigma_n(V,x)$. Observe that
  \begin{equation*}
    0<(M_\iii-1)\delta\leq \diam(X)+2\delta<3\diam(X)+2\delta.
  \end{equation*}
  Therefore, by H\"older's inequality, we have
  \begin{equation} \label{eq:comp4}
  \begin{split}
    \sum_{\iii j \in \Sigma_{n+1}(V,x)} &\HH^1(\fii_{\iii j}(U) \cap (V+x))^s \\
    &\le M_\iii^{1-s} \biggl( \sum_{\iii j \in \Sigma_{n+1}(V,x)} \HH^1(\fii_{\iii j}(U) \cap (V+x)) \biggr)^s \\
    &\le M_\iii^{1-s} \biggl( 1 - \frac{(M_\iii-1)\delta}{3\diam(X)+2\delta} \biggr)^s \HH^1(\fii_\iii(U) \cap (V+x))^s
  \end{split}
  \end{equation}
  for all $\iii \in \Sigma_n(V,x)$ and $0 \le s \le 1$. For each $M \in \{1,\ldots,N\}$ define a function $f_M \colon [0,1] \to \R$ by setting
  \begin{equation*}
    f_M(s) = M^{1-s} \biggl( 1 - \frac{(M-1)\delta}{3\diam(X)+2\delta} \biggr)^s.
  \end{equation*}
  If $M=1$, then $f_M \equiv 1$. If $M \ge 2$, then it is easy to see that $f_M$ is continuous and strictly decreasing with $f_M(0)=M \ge 2$ and $f_M(1)<1$. Hence for every $M \in \{2,\ldots,N\}$ there exists a unique $0 < s_M < 1$ such that $f_M(s_M)=1$.

  Choosing now $\max_{M \in \{2,\ldots,N\}} s_M \le s < 1$, the estimate \eqref{eq:comp4} gives
  \begin{equation*}
    \sum_{\iii j \in \Sigma_{n+1}(V,x)} \HH^1(\fii_{\iii j}(U) \cap (V+x))^s \le \HH^1(\fii_\iii(U) \cap (V+x))^s.
  \end{equation*}
  Thus, by induction,
  \begin{equation} \label{eq:comp3}
  \begin{split}
    \sum_{\jjj\in\Sigma_{n+1}(V,x)} &\HH^1(\fii_\jjj(U) \cap (V+x))^s = \sum_{\iii\in\Sigma_n(V,x)} \sum_{\iii j \in \Sigma_{n+1}(V,x)} \HH^1(\fii_{\iii j}(U) \cap (V+x))^s \\
    &\hspace{-2em}\le \sum_{\iii\in\Sigma_n(V,x)} \HH^1(\fii_\iii(U) \cap (V+x))^s \le \cdots \le \sum_{i\in\Sigma_1(V,x)} \HH^1(\fii_i(U) \cap (V+x))^s.
  \end{split}
  \end{equation}
  So, by considering the natural cover $\{\varphi_\iii(U)\cap(V+x)\}_{\iii\in\Sigma_{n}(V,x)}$ of $X\cap(V+x)$, the estimates \eqref{eq:comp5} and \eqref{eq:comp3} yield
  \begin{align*}
    \mathcal{H}^{s}(X\cap(V+x)) &\le \lim_{n\to\infty}\sum_{\iii\in\Sigma_n(V,x)}\diam(\varphi_\iii(U)\cap(V+x))^s \\
    &\le \biggl( 1+\frac{3\diam(X)}{2\delta} \biggr)^s \lim_{n\to\infty}\sum_{\iii\in\Sigma_n(V,x)} \HH^1(\fii_\iii(U) \cap (V+x))^s \\
    &\le \biggl( 1+\frac{3\diam(X)}{2\delta} \biggr)^s \sum_{i\in\Sigma_1(V,x)} \HH^1(\fii_i(U) \cap (V+x))^s < \infty.
  \end{align*}
  This implies $\dimh(X\cap(V+x)) \le s < 1$. As the upper bound $s$ for the dimension does not depend on $x \in X$ or $V \in \RP$, we have finished the proof.
\end{proof}

The following lemma examines the slices and establishes the quasi-product structure of weak tangent sets. This allows us to determine the dimension of the self-affine set.

\begin{lemma}\label{lem:teethofcomb}
  If $X$ is a dominated planar self-affine set satisfying the strong separation condition such that $\dimh(X) \ge 1$ and $X_F$ is not a singleton, then for every $x \in X$ and $V \in X_F$ there exist $T \in \Tan(X)$, $W\in\RP$, and $L \in X_F$ with $W \ne L$ such that
  \begin{align*}
    &\dimh(T \cap L) \ge \dimh(X \cap (V+x)), \\
    &\dimh(T\cap(W+y)) = 1
  \end{align*}
  for all $y \in T$.
\end{lemma}

\begin{proof}
  Let us first define the weak tangent set $T \in \Tan(X)$. Let $x \in X$ and $V \in X_F$. By \eqref{eq:furstenberg-directions-dominated2}, there exists $\iii\in\Sigma$ such that $\Pi(\iii)=V$. By Lemma~\ref{thm:bochi-morris}, this will ensure that there exists $D>0$ such that $\alpha_2(A_{\overleftarrow{\iii|_k}})\leq\|A_{\overleftarrow{\iii|_k}}|V\|\leq D\alpha_2(A_{\overleftarrow{\iii|_k}})$. Let $\delta = \delta'/D$, where $\delta'>0$ is as in \eqref{eq:SSC}. Write
  \begin{equation*}
 	  T_k = M_{\fii_{\overleftarrow{\iii|_k}}(x),\|A_{\overleftarrow{\iii|_k}}|V\|\delta} \circ \fii_{\overleftarrow{\iii|_k}}(X) \cap B(0,1)=M_{\fii_{\overleftarrow{\iii|_k}}(x),\|A_{\overleftarrow{\iii|_k}}|V\|\delta}(X) \cap B(0,1)
  \end{equation*}
  for all $k \in \N$, where the second equality follows by the strong separation condition. By going into a sub-sequence $n_k$, if necessary, we see that there is $T \in \Tan(X)$ such that $T_{n_k} \to T$ in Hausdorff distance.

  Let us then show that there exists $L \in X_F$ and an invertible affine map $g \colon V+x \to L$ such that $g(X \cap (V+x)) \subset T \cap L$ and that for every $y \in T$ there exists $W \in \RP$ such that $W \ne L$ and $\dimh(T\cap(W+y))=1$. Fix $y \in T$ and note that, possibly passing through a sub-sequence, there exists a sequence $(\jjj_k)_{k \in \N}$ of words in $\Sigma$ such that $\jjj_k \in [\overleftarrow{\iii|_{n_k}}]$ for all $k\in\N$ and  $M_{\fii_{\overleftarrow{\iii|_{n_k}}}(x),\|A_{\overleftarrow{\iii|_{n_k}}}|V\|\delta}(\pi(\jjj_k)) \to y$. Let $m_k \ge n_k$ be the unique integer such that
  \begin{equation}\label{eq:unnumbered}
 	  \alpha_1(A_{\jjj_k|_{m_k}}) \le \|A_{\overleftarrow{\iii|_{n_k}}}|V\|\delta \le \alpha_1(A_{\jjj_k|_{m_k-1}}).
  \end{equation}
  By the convergence, there exists a constant $C > 0$ such that for every $k$,
  \begin{equation*}
    \biggl| \frac{A_{\jjj_k|_{m_k}}\pi(\sigma^{m_k}\jjj_k) + v_{\jjj_k|_{m_k}} - \fii_{\overleftarrow{\iii|_{n_k}}}(x)}{\|A_{\overleftarrow{\iii|_{n_k}}}|V\|\delta} \biggr| \le C,
  \end{equation*}
  where we have written $\fii_{\jjj_k|_{m_k}}(\pi(\sigma^{m_k}\jjj_k)) = A_{\jjj_k|_{m_k}}\pi(\sigma^{m_k}\jjj_k) + v_{\jjj_k|_{m_k}}$, it follows from \eqref{eq:unnumbered} that
  \begin{equation} \label{eq:unnumbered2}
    \biggl| \frac{v_{\jjj_k|_{m_k}} - \fii_{\overleftarrow{\iii|_{n_k}}}(x)}{\|A_{\overleftarrow{\iii|_{n_k}}}|V\|\delta} \biggr| \le C + \frac{\|A_{\jjj_k|_{m_k}}\| |\pi(\sigma^{m_k}\jjj_k)|}{\|A_{\overleftarrow{\iii|_{n_k}}}|V\|\delta} \le C + \diam(X).
  \end{equation}
  Then, possibly passing again through a sub-sequence, there exist $L \in \RP$ and affine maps $g$ and $P$ such that
  \begin{equation} \label{eq:affine-P}
  \begin{split}
 		M_{\fii_{\overleftarrow{\iii|_{n_k}}}(x),\|A_{\overleftarrow{\iii|_{n_k}}}|V\|\delta} \circ \fii_{\overleftarrow{\iii|_{n_k}}}(V+x) = A_{\overleftarrow{\iii|_{n_k}}}V &\to L, \\
        M_{\fii_{\overleftarrow{\iii|_{n_k}}}(x),\|A_{\overleftarrow{\iii|_{n_k}}}|V\|\delta} \circ \fii_{\overleftarrow{\iii|_{n_k}}}\Big|_{V+x} &\to g, \\
 		M_{\fii_{\overleftarrow{\iii|_{n_k}}}(x),\|A_{\overleftarrow{\iii|_{n_k}}}|V\|\delta} \circ \fii_{\jjj_k|_{m_k}} &\to P.
 	\end{split}
  \end{equation}
  In the last convergence, we used \eqref{eq:unnumbered2} to see that the associated translation vectors remain bounded. Furthermore, since $V \in X_F$, we have $L \in X_F$. Also, by compactness, we see that $g(X \cap (V+x)) \subset T \cap L$ and $y \in P(X)\cap B(0,1)  \subset T$. Let us show that $\rank(P)=1$. Observe that
  \begin{equation} \label{eq:this}
  \begin{split}
	  \frac{\alpha_2(A_{\jjj_k|_{m_k}})}{\alpha_1(A_{\jjj_k|_{m_k}})} &\le \frac{\|A_{\overleftarrow{\iii|_{n_k}}}|V\|\delta\alpha_1(A_{\sigma^{n_k}\jjj_k|_{m_k}})}{\alpha_1(A_{\jjj_k|_{m_k}})} \\
    &\le \Bigl( \min_{i \in \{1,\ldots,N\}} \alpha_2(A_i) \Bigr)^{-1} \alpha_1(A_{\sigma^{n_k}\jjj_k|_{m_k}})\delta,
  \end{split}
  \end{equation}
  where in the second inequality we used \eqref{eq:unnumbered}. If the sequence $(|\sigma^{n_k}\jjj_k|_{m_k}|)_{k \in \N}$ of natural numbers was bounded by some $K\in\N$, then
  \begin{equation*}
  \begin{split}
    \alpha_1(A_{\overleftarrow{\iii|_{n_k}}})\min_i\alpha_2(A_i)^K &\leq \alpha_1(A_{\overleftarrow{\iii|_{n_k}}})\alpha_2(A_{\sigma^{n_k}\jjj_k|_{m_k}}) \leq \alpha_1(A_{\overleftarrow{\iii|_{n_k}}}A_{\sigma^{n_k}\jjj_k|_{m_k}}) \\
    &= \alpha_1(A_{\jjj_k|_{m_k}}) \leq \|A_{\overleftarrow{\iii|_{n_k}}}|V\|\delta \leq D\delta\alpha_2(A_{\overleftarrow{\iii|_{n_k}}}),
  \end{split}
  \end{equation*}
  where the last inequality follows from Lemma~\ref{thm:bochi-morris}. As this contradicts with \eqref{eq:domconst}, the sequence $(|\sigma^{n_k}\jjj_k|_{m_k}|)_{k \in \N}$ must be unbounded. Therefore, it follows from \eqref{eq:this} that
  \begin{equation*}
    \frac{\alpha_2(A_{\jjj_k|_{m_k}})}{\alpha_1(A_{\jjj_k|_{m_k}})} \to 0
  \end{equation*}
  as $k \to \infty$. Hence $\det(\|A_{\overleftarrow{\iii|_{n_k}}}|V\|^{-1} A_{\jjj_k|_{m_k}}) \to 0$ as $k \to \infty$ and $\rank(P)=1$. By Lemma~\ref{thm:nilpotent} and \eqref{eq:rank-one-projection-nilpotent}, we see that the linear part of $P$ is a constant times $\proj_{\im(P)}^{\ker(P)}$. Let us choose $W = \im(P)$ and show that $W \ne L$ and $\dimh(T\cap(W+y))=1$. Relying on domination, let $\CC \subset \RP$ be a strongly invariant multicone. Fix $Q \in \CC$ and notice that, as $A_{\jjj_k|_{m_k}}\CC \subset \CC^o$ for all $k \in \N$, by going into a sub-sequence, if necessary, $A_{\jjj_k|_{m_k}}Q \to \im(P) \in \CC^o$. Since, by \eqref{eq:furstenberg-directions-dominated}, $\CC \cap X_F = \emptyset$, we see that $W \ne L$. Furthermore, since $P(X)$ and $\proj_{\ker(P)^\perp}(X)$ are bi-Lipschitz equivalent by \eqref{eq:rank-one-projection-nilpotent}, the assumption $\dimh(X) \ge 1$, Lemma \ref{thm:irreducible-dominated}, and Theorem \ref{thm:BHR2} give
  \begin{equation*}
  \begin{split}
    \dimh(T \cap (W+y)) &\ge \dimh(P(X)\cap B(0,1)) \\
    &\ge \dimh(P(\varphi_{\kkk}(X)))= \dimh(\proj_{A_{\overleftarrow{\kkk}}^{\top}\ker(P)^\perp}(X)) = 1,
  \end{split}
  \end{equation*}
  where in the second inequality, as $P(X)\cap B(0,1) \ne \emptyset$, we have chosen $\kkk\in\Sigma_*$ such that $P(\varphi_{\kkk}(X)) \subset P(X)\cap B(0,1)$.

  To finish the proof, let us show that $W \in \RP$ does not depend on the choice of $y \in T$. Suppose to the contrary that there exist $y_1,y_2 \in T$ such that the associated lines $W_1, W_2 \in \RP$ satisfy $W_1 \ne W_2$. Let $P_1$ and $P_2$ be the affine maps associated to $y_1$ and $y_2$ defined in \eqref{eq:affine-P}, respectively. By Lemma~\ref{lem:embedding}, there exists an affine map $G \colon \R^2 \to \R^2$ such that $\rank(G) = 1$ and $G(T) \subset X$. Recall that, by \eqref{eq:rank-one-projection-nilpotent}, $G$ is bi-Lipschitz equivalent to $\proj_{\ker(G)^\perp}$. If $i \in \{1,2\}$ is such that $W_i \ne \ker(G)$, then, by Theorem \ref{thm:BHR2},
  \begin{align*}
    \dimh(X \cap (GW_i+Gy_i)) &\ge \dimh(GP_i(X) \cap (GW_i+Gy_i)) \\
    &= \dimh(G\proj_{W_i}(X)) = 1.
  \end{align*}
  Thus, there exists a slice of $X$ with dimension one, which is impossible by Lemma \ref{lem:smallslice}.
\end{proof}

We are now ready to prove Theorem \ref{thm:dima}.

\begin{proof}[Proof of Theorem \ref{thm:dima}]
  Let us first prove that
  \begin{equation} \label{eq:dima-lower-bound}
    \dima(X) \ge 1 + \sup_{\atop{z\in X}{P\in X_F}}\dimh(X\cap(P+z)).
  \end{equation}
  Fix $\eps>0$ and choose $x \in X$ and $V \in X_F$ be such that
  \begin{equation} \label{eq:dima-est1}
    \dimh(X \cap (V+x)) \ge \sup_{\atop{z\in X}{P\in X_F}}\dimh(X\cap(P+z)) - \eps.
  \end{equation}
  By Lemma \ref{lem:teethofcomb}, there exist $T \in \Tan(X)$ and $W,L\in\RP$ with $W \ne L$ such that
  \begin{align*}
    &\dimh(T \cap L) \ge \dimh(X \cap (V+x)),\\
    &\dimh(T\cap(W+y)) = 1
  \end{align*}
  for all $y \in T$. Notice that if $\dimh(T \cap L) = 0$, then trivially $\dimh(T) \ge 1 = 1 + \dimh(X \cap (V+x))$. Let us therefore assume that $0<s<\dimh(T \cap L)$. Relying on Frostman's lemma, see e.g.\ \cite[Theorem 8.8]{Mattila1995}, let $\mu$ be a Borel probability measure on $T \cap L$ such that for some constant $C \ge 1$ it holds that $\mu(B(y,r)) \le Cr^s$ for all $y \in T \cap L$ and $r>0$. Now the Marstrand's slicing theorem \cite[Theorem 3.3.1]{BishopPeres2017} implies that
  \begin{equation*}
    1 = \dimh(T \cap (W+y)) \le \dimh(T) - s
  \end{equation*}
  for $\mu$-almost all $y \in T \cap L$. Therefore, by letting $s \uparrow \dimh(T \cap L)$, we get
  \begin{equation} \label{eq:dima-est2}
    \dimh(T) \ge 1 + \dimh(T \cap L) \ge 1 + \dimh(X \cap (V+x)).
  \end{equation}
  By Lemma \ref{thm:KOR}, \eqref{eq:dima-est2}, and \eqref{eq:dima-est1}, we thus have
  \begin{equation*}
    \dima(X) \ge \dimh(T) \ge 1 + \sup_{\atop{z\in X}{P\in X_F}}\dimh(X\cap(P+z)) - \eps
  \end{equation*}
  and \eqref{eq:dima-lower-bound} follows by letting $\eps \downarrow 0$.

  To show that the claimed quantity is an upper bound for $\dima(X)$ and the associated maximum exists, recall Lemma \ref{thm:KOR} and let $T \in \Tan(X)$ be such that $\dimh(T) = \dima(X)$. By Lemma~\ref{lem:embedding}, there exists an affine map $G \colon \R^2 \to \R^2$ such that $\rank(G) = 1$, $G(T) \subset X$, and $\im(G) \in X_F$. Recall that, by \eqref{eq:rank-one-projection-nilpotent}, $G$ is bi-Lipschitz equivalent to $\proj_{\ker(G)^\perp}$. Let $x \in X$ be such that $G(T) \subset \im(G)+x$ and observe that
  \begin{align*}
    \sup_{\atop{z\in X}{P\in X_F}}\dimh(X\cap(P+z)) &\ge \dimh(X \cap (\im(G)+x)) \ge \dimh(X \cap G(T)) \\ 
    &= \dimh(G(T)) = \dimh(\proj_{\ker(G)^\perp}(T)).
  \end{align*}
  Since $\dimh(T) \le \dimh(\proj_{\ker(G)^\perp}(T) \times \R) \le \dimh(\proj_{\ker(G)^\perp}(T))+1$, we furthermore have
  \begin{equation*}
    \dimh(\proj_{\ker(G)^\perp}(T)) \ge \dimh(T)-1 = \dima(X)-1.
  \end{equation*}
  By \eqref{eq:dima-lower-bound}, we now have an equality throughout. The proof is finished by recalling that the upper bound is strictly smaller than $2$ by Lemma \ref{lem:smallslice}.
\end{proof}

\section{Assouad dimension of self-affine sets having small Hausdorff dimension} \label{sec:assouad-small}

In this section, we show that the projective open set condition implies the projective separation, and prove Theorem \ref{thm:dima2} and part of the claims in Theorem \ref{thm:ahl}. In particular, we show that the conditions of Theorem \ref{thm:ahl} satisfy \eqref{(1)} $\Leftrightarrow$ \eqref{(11)} $\Rightarrow$ \eqref{(4)} $\Rightarrow$ \eqref{(6)} and, if the Hausdorff dimension is strictly smaller than one, also \eqref{(6)} $\Rightarrow$ \eqref{(1)}.

Proofs of these implications are structured below as follows: The implication \eqref{(1)} $\Rightarrow$ \eqref{(11)} is shown in Theorem \ref{thm:posc-sigma} and, by recalling Theorem \ref{thm:BHR2}, Theorem \ref{thm:posc-regular} verifies the implication \eqref{(11)} $\Rightarrow$ \eqref{(4)}. The implication \eqref{(4)} $\Rightarrow$ \eqref{(6)} follows from Lemma \ref{thm:ahlfors-implications}. Proposition \ref{thm:POSC-PSC} below shows that the projective open set condition implies \eqref{(11)}, Proposition \ref{lem:missing} verifies the implication \eqref{(11)} $\Rightarrow$ \eqref{(1)} and finally, Theorem \ref{thm:loweronassuad} proves Theorem \ref{thm:dima2} and also the implication \eqref{(6)} $\Rightarrow$ \eqref{(1)}.

Covering arguments are at the core of the proofs of Proposition \ref{thm:POSC-PSC}, Theorem \ref{thm:posc-sigma}, and Theorem \ref{thm:posc-regular}. Proving Theorem \ref{thm:loweronassuad} is more complicated and it requires two auxiliary lemmas. The purpose of Lemma \ref{lem:key} is to show the existence of distinct maps whose projections are approximately at a given distance on a certain small set. Applying this inductively in Lemma \ref{lem:totangent}, it is possible to find small scales containing as many equally distributed points of $X$ as wished. This shows that a line segment appears as a subset of a weak tangent set which, by Lemma \ref{thm:KOR}, implies $\dima(X) \ge 1$ as claimed. The difficulty in the proof is not only that we have to keep track of different directions but also ensure that the points we define will eventually approximate a line segment. Proposition \ref{lem:missing} is another consequence of Lemma \ref{lem:totangent}.

\begin{proposition} \label{thm:POSC-PSC}
  If $X$ is a dominated planar self-affine set satisfying the projective open set condition, then
  \begin{equation*}
    \sup\{\#\Sigma(V,x,r) : V \in X_F, \; x \in X, \text{ and } r>0\} < \infty,
  \end{equation*}
  where $\Sigma(V,x,r)$ is as in \eqref{eq:POSC-Sigma-def}.
\end{proposition}

\begin{proof}
  Let $U \subset \R^2$ be the open set in the projective open set condition and choose an open ball $B \subset U$. Fix $V \in X_F$ and let us first show that
  \begin{equation} \label{eq:posc-iteration}
    \proj_{V^\bot}(\fii_\iii(U)) \cap \proj_{V^\bot}(\fii_\jjj(U)) = \emptyset
  \end{equation}
  for all $\iii,\jjj \in \Sigma_*$ with $[\iii] \cap [\jjj] = \emptyset$. Notice that if $\hhh = \iii \land \jjj$ and $n=|\hhh|+1$, then $\iii|_n = \hhh i$ and $\jjj|_n = \hhh j$ where $i \ne j$.
  Recalling that
  \begin{equation*}
    \proj_{V^\bot}(\fii_\hhh(x)) = \pm\|A_\hhh^\top|V^\bot\| |\proj_{A_\hhh^\top V^\bot}(x)|v^\bot + \proj_{V^\bot}(\fii_\hhh(0)),
  \end{equation*}
  where $v^\bot$ is a unit vector in $V^\bot$ and $\pm$ indicates that the equality holds with one of the signs, we define a mapping $g \colon V^\bot \to V^\bot$ by setting
  \begin{equation*}
    g(x) = \pm\|A_\hhh^\top|V^\bot\|x + |\proj_{V^\bot}(\fii_\hhh(0))|v^{\bot}
  \end{equation*}
  for all $x \in V^\bot$, where $v^\bot$ is a unit vector in $V^\bot$. If there is $y \in \proj_{V^\bot}(\fii_\iii(U)) \cap \proj_{V^\bot}(\fii_\jjj(U))$, then $g^{-1}(y) \in \proj_{V^\bot}(\fii_i(U)) \cap \proj_{V^\bot}(\fii_j(U))$ and the claim follows.

  As $\A$ is dominated, Lemma \ref{thm:bochi-morris} shows that there exists a constant $D \ge 1$ such that
  \begin{equation} \label{eq:POSC-bochi-this-time}
    \|A_\iii^\top|V^\perp\| \le \alpha_1(A_\iii) \le D\|A_\iii^\top|V^\perp\|
  \end{equation}
  for all $\iii\in\Sigma_*$ and $V \in X_F$. Write $\kappa = \min_{i \in \{1,\ldots,N\}} \alpha_2(A_i)$ and observe that, by \eqref{eq:POSC-bochi-this-time}, $\|A_\iii^\top|V^\perp\| \ge D^{-1}\alpha_1(A_\iii) \ge \kappa D^{-1}\alpha_1(A_{\iii^-}) \ge \kappa D^{-1}\|A_{\iii^-}^\top|V^\perp\|$. Fix $x \in X$ and $r>0$. Since the diameter of a ball is invariant under projections and orthogonal projections are $1$-Lipschitz, we get
  \begin{align*}
    \HH^1(\proj_{V^\bot}(\fii_\iii(U))) &\ge \diam(\proj_{V^\bot}(\fii_\iii(B))) = \|A_\iii^\top|V^\bot\|\diam(\proj_{A_\iii^\top V^\bot}(B)) \\
    &\ge \kappa D^{-1}\|A_{\iii^-}^\top|V^\perp\|\diam(\proj_{A_\iii^\top V^\bot}(B)) \\
    &= \kappa D^{-1}\frac{\diam(\proj_{A_\iii^\top V^\bot}(B))}{\diam(\proj_{A_{\iii^-}^\top V^\bot}(X))} \diam(\proj_{V^\bot}(\fii_{\iii^-}(X))) \\
    &\ge \kappa D^{-1}\frac{\diam(B)}{\diam(X)} r
  \end{align*}
  for all $\iii \in \Sigma(V,x,r)$, we have, by \eqref{eq:posc-iteration},
  \begin{align*}
    4r &= \HH^1(B(x,2r)) \ge \HH^1\biggl( \bigcup_{\iii \in \Sigma(V,x,r)} \proj_{V^\bot}(\fii_\iii(U)) \biggr) \\
    &= \sum_{\iii \in \Sigma(V,x,r)} \HH^1(\proj_{V^\bot}(\fii_\iii(U))) \ge \kappa D^{-1}\frac{\diam(B)}{\diam(X)} r \#\Sigma(V,x,r)
  \end{align*}
  and
  \begin{equation*}
    \#\Sigma(V,x,r) \le \frac{4D\diam(X)}{\kappa\diam(B)}.
  \end{equation*}
  As the upper bound does not depend on $V$, $x$, or $r$, we have finished the proof.
\end{proof}

\begin{theorem} \label{thm:posc-sigma}
  If $X$ is a dominated planar self-affine set satisfying the projective separation, then there exists a constant $C \ge 1$ such that
  \begin{equation} \label{eq:POSC-regular1}
    \sup\{\#\Sigma(V,x,r) : x \in X \text{ and } r>0\} \le C\frac{\exp(C\diam(\proj_{V^\bot}(X))^{-2})}{\diam(\proj_{V^\bot}(X))}
  \end{equation}
  for all $V \in X_F$ with $\diam(\proj_{V^\bot}(X))>0$, where $\Sigma(V,x,r)$ is as in \eqref{eq:POSC-Sigma-def}. In particular, if $X_F$ is not a singleton, then $\sup\{\#\Sigma(V,x,r) : V \in X_F, \; x \in X, \text{ and } r>0\} < \infty$.
\end{theorem}

\begin{proof}
  Let us first assume that $V \in X_F$ is such that $\diam(\proj_{V^\perp}(X)) > 0$. Recall that, by \eqref{eq:transpose-projection}, $\|\proj_{V^\perp}A_\iii\|=\|A_\iii^\top|V^\perp\|$ for all $\iii \in \Sigma_*$. By the projective separation \eqref{eq:POSC}, there is $\eta>0$ such that for every $V \in X_F$, $x \in X$, $r>0$, and $\jjj,\kkk \in \Sigma(V,x,r)$ with $\jjj\ne\kkk$ there is $z \in X$ such that
  \begin{equation} \label{eq:POSC-regular2}
    |\proj_{V^\perp}(\fii_\jjj(z)-\fii_\kkk(z))| \ge \eta\diam(\proj_{V^\perp}(X))\|A_\jjj^\top|V^\perp\|.
  \end{equation}
  As $\A$ is dominated, Lemma \ref{thm:bochi-morris} shows that there exists a constant $D \ge 1$ such that
  \begin{equation} \label{eq:POSC-morris}
    \|A_\iii^\top|V^\perp\| \le \alpha_1(A_\iii) \le D\|A_\iii^\top|V^\perp\|
  \end{equation}
  for all $\iii\in\Sigma_*$ and $V \in X_F$.

  Fix $x \in X$ and $r>0$. Let $\jjj,\kkk \in \Sigma(V,x,r)$ with $\jjj\ne\kkk$ and $z \in X$ be such that \eqref{eq:POSC-regular2} holds, and choose $y \in B(z,\gamma\diam(\proj_{V^\perp}(X)))$, where $\gamma = \kappa\eta/6D$ and $\kappa = \min_{i \in \{1,\ldots,N\}} \alpha_2(A_i)$. Observe that, by \eqref{eq:POSC-morris}, $\|A_\jjj^\top|V^\perp\| \ge D^{-1}\alpha_1(A_\jjj) \ge \kappa D^{-1}\alpha_1(A_{\jjj^-}) \ge \kappa D^{-1}\|A_{\jjj^-}^\top|V^\perp\|$ and hence, by \eqref{eq:POSC-regular2}, the triangle inequality, and \eqref{eq:POSC-Sigma-def},
  \begin{equation} \label{eq:POSC-regular3}
  \begin{split}
    |\proj_{V^\perp}(\varphi_{\jjj}(y)&-\varphi_{\kkk}(y))| \ge |\proj_{V^\perp}(\varphi_{\jjj}(z)-\varphi_{\kkk}(z))| - |\proj_{V^\perp}(A_\jjj-A_\kkk)(z-y)| \\
    &\ge \eta\diam(\proj_{V^\perp}(X))\|A_\jjj^\top|V^\perp\| - (\|A_\jjj^\top|V^\perp\|+\|A_\kkk^\top|V^\perp\|)|z-y| \\
    &\ge (\kappa D^{-1}\eta-2\gamma)r = 4\gamma r.
  \end{split}
  \end{equation}
  As $X$ is compact, there are finitely many points $z_1,\ldots,z_k \in X$ such that
  \begin{equation*}
    X \subset \bigcup_{i=1}^k B(z_i,\gamma\diam(\proj_{V^\perp}(X))).
  \end{equation*}
  By a simple volume argument, the points can be chosen such that
  \begin{equation} \label{POSC-volume-k}
    k \le C\biggl( \frac{\diam(X)}{\gamma\diam(\proj_{V^\bot}(X))} \biggr)^2,
  \end{equation}
  where $C \ge 1$ does not depend on $V$. Therefore, for every $\jjj,\kkk \in \Sigma(V,x,r)$ with $\jjj\ne\kkk$ and the associated $z \in X$ satisfying \eqref{eq:POSC-regular2} there is $i \in \{1,\ldots,k\}$ such that $z_i \in B(x,\gamma\diam(\proj_{V^\perp}(X)))$. Writing $\xi_\lll = (\proj_{V^\perp}(\fii_\lll(z_1)),\ldots,\proj_{V^\perp}(\fii_\lll(z_k))) \in (\R^2)^k$ for all $\lll \in \Sigma_*$, we see that \eqref{eq:POSC-regular3} gives $|\xi_\jjj-\xi_\kkk| \ge 4\gamma r$ and hence,
  \begin{equation} \label{eq:POSC-regular4}
    B(\xi_\jjj,\gamma r) \cap B(\xi_\kkk,\gamma r) = \emptyset
  \end{equation}
  for all $\jjj,\kkk \in \Sigma(V,x,r)$ with $\jjj\ne\kkk$. On the other hand, if $\jjj \in \Sigma(V,x,r)$, then
  \begin{equation*}
    |\proj_{V^\perp}(x)-\proj_{V^\perp}(\varphi_\jjj(z))| \le 2r
  \end{equation*}
  for all $z \in X$, and so $|\proj_{V^\perp}(x)-\xi_\jjj| \le 2\sqrt{k}r$. It follows that
  \begin{equation}  \label{eq:POSC-regular5}
    \bigcup_{\jjj \in \Sigma(V,x,r)} B(\xi_\jjj,\gamma r) \subset B((\proj_{V^\perp}(x),\dots,\proj_{V^\perp}(x)),2\sqrt{k}r).
  \end{equation}
  Therefore, by \eqref{eq:POSC-regular4}, \eqref{eq:POSC-regular5}, and \eqref{POSC-volume-k}, we see that
  \begin{align*}
    \#\Sigma(V,x,r) &\le \biggl( \frac{2\sqrt{k}}{\gamma} \biggr)^{2k} = \biggl( \frac{12D\sqrt{k}}{\kappa\eta} \biggr)^{2k} = \frac{12D\sqrt{k}\exp(2k)}{\kappa\eta} \\
    &\le \frac{12D\sqrt{C}\diam(X)\exp(2C\diam(X)^2\diam(\proj_{V^\bot}(X))^{-2})}{\kappa\eta\diam(\proj_{V^\bot}(X))}.
  \end{align*}
  As the upper bound does not depend on $x \in X$ nor $r>0$, we have shown \eqref{eq:POSC-regular1}.

  To show the second claim, let us assume that $X_F$ is not a singleton. By \eqref{eq:POSC-regular1}, it suffices to show that
  \begin{equation}\label{eq:posdiam}
    \inf_{V \in X_F}\diam(\proj_{V^\perp}(X)) > 0.
  \end{equation}
  If this was not the case, then, by the continuity of $V\mapsto \diam(\proj_{V^\perp}(X))$ and the compactness of $X_F$, there is $V \in X_F$ such that $\diam(\proj_{V^\perp}(X)) = 0$. It follows that there exist $x,y \in X$ such that $X \subset V+x$ and $y \ne x$. Furthermore, since $X_F$ is not a singleton, there exists $i \in \{1,\ldots,N\}$ such that $A_i^{-1}V\neq V$. Note that $\varphi_i(x)\neq\varphi_i(y)$ and since $\varphi_i(X)\subset X\subset V+x$, we have $\varphi_i(x)-\varphi_i(y)\in V$. But then $x-y=A_i^{-1}(\varphi_i(x)-\varphi_i(y))\notin V$, which is a contradiction.
\end{proof}

\begin{theorem} \label{thm:posc-regular}
  If $X$ is a dominated planar self-affine set such that $s = \dimaff(X) \le 1$ and $\sup\{\#\Sigma(V,x,r) : x \in X \text{ and } r>0\} < \infty$ for some $V \in X_F$, then $X$ is Ahlfors $s$-regular.
\end{theorem}

\begin{proof}
  By Lemma \ref{thm:equilibrium-state-existence}, there exist a measure $\mu_K \in \MM_\sigma(\Sigma)$ and a constant $C \ge 1$ such that
  \begin{equation} \label{eq:POSC-regular6}
    C^{-1}\alpha_1(A_\iii)^s \le \mu_K([\iii]) \le C\alpha_1(A_\iii)^s
  \end{equation}
  for all $\iii \in \Sigma_*$. Let $V \in X_F$ be such that $\sup\{\#\Sigma(V,x,r) : x \in X \text{ and } r>0\} < \infty$. By Lemma \ref{thm:bochi-morris}, there exists a constant $D \ge 1$ such that
  \begin{equation} \label{eq:POSC-morris-again}
    \|A_\iii^\top|V^\perp\| \le \alpha_1(A_\iii) \le D\|A_\iii^\top|V^\perp\|
  \end{equation}
  for all $\iii\in\Sigma_*$.

  Fix $x \in X$ and $0<r<\diam(X)$. Let $\iii\in\Sigma$ be such that $\pi(\iii)=x$ and choose $n \in \N$ such that $\fii_{\iii|_n}(X) \subset B(x,r)$ but $\fii_{\iii|_{n-1}}(X) \setminus B(x,r) \ne \emptyset$. Note that
  \begin{equation*}
    \kappa r \le \kappa\diam(\fii_{\iii|_{n-1}}(X)) \le \kappa\alpha_1(A_{\iii|_{n-1}}) \diam(X) \le \alpha_1(A_{\iii|_{n}}) \diam(X),
  \end{equation*}
  where $\kappa = \min_{i \in \{1,\ldots,N\}} \alpha_2(A_i)$. Therefore, by \eqref{eq:POSC-regular6},
  \begin{equation} \label{eq:ahlfors-lower-part}
  \begin{split}
    \pi_*\mu_K(B(x,r)) &\ge \pi_*\mu_K(\fii_{\iii|_n}(X)) \ge \mu_K([\iii|_n]) \\
    &\ge C^{-1}\alpha_1(A_{\iii|_n})^s \ge C^{-1} \kappa^s \diam(X)^{-s}r^s
  \end{split}
  \end{equation}
  and, by \eqref{eq:POSC-morris-again},
  \begin{align*}
    \pi_*\mu_K(B(x,r)) &\le \pi_*\mu_K(\proj_{V^\perp}^{-1}(\proj_{V^\perp}(B(x,r)))) \le \sum_{\jjj \in \Sigma(V,x,r)}\mu_K([\jjj]) \\
    &\le C\sum_{\jjj \in \Sigma(V,x,r)} \alpha_1(A_\jjj)^s \le CD^s \sum_{\jjj \in \Sigma(V,x,r)} \|\proj_{V^\perp}A_\jjj\|^s \\
    &\le CD^s \sup\{\#\Sigma(V,x,r) : x \in X \text{ and } r>0 \} \lambda^{-s} r^s,
  \end{align*}
  where $\lambda = \inf_{V \in X_F}\diam(\proj_{V^\perp}(X)) > 0$. The measure $\pi_*\mu_K$ is thus Ahlfors $s$-regular. As \eqref{eq:POSC-regular6} guarantees that the support of $\pi_*\mu_K$ is $X$, Lemma \ref{thm:ahlfors-implications} finishes the proof.
\end{proof}

Before proving Theorem \ref{thm:loweronassuad}, let us exhibit two auxiliary lemmas. The first one shows the existence of distinct maps whose projections have approximately a given small relative distance on a small open set which we next define. Relying on the strong separation condition, let $\delta>0$ be as in \eqref{eq:SSC}. If the self-affine set $X$ is not contained in a line, let $z_1,z_2,z_3\in X$ be in a general position such that $X \cap \conv(\{z_1,z_2,z_3\})^o \ne \emptyset$ and $|z_i-z_j| \le \frac{\delta}{3}$ whenever $i \ne j$. Write
\begin{equation} \label{eq:Z-set}
  Z = \conv(\{z_1,z_2,z_3\}).
\end{equation}
Bear in mind that if $X$ and $X_F$ are not singletons, then $X$ is not contained in a line.

\begin{lemma}\label{lem:key}
  If $X$ is a dominated planar self-affine set satisfying the strong separation condition, but not the projective separation, such that $X_F$ is not a singleton, then there are $\eta_0>0$ and $C \ge 1$ such that for every $0<\eta<\eta_0$ there exist $V \in X_F$ and $\iii,\jjj\in\Sigma_*$ with $\iii|_1 \ne \jjj|_1$ such that
  \begin{equation*}
    C^{-1}\eta\|A_\iii^\top|V^\perp\| \le |\proj_{V^\perp}(\fii_\iii(x)-\fii_\jjj(x))| \le C\eta\|A_\iii^\top|V^\perp\|
  \end{equation*}
  for all $x\in Z$, where $Z$ is as in \eqref{eq:Z-set}.
\end{lemma}

\begin{proof}
  Define $\eta_0 = \frac16 \diam(X)^{-1} \delta$, where $\delta$ is as in \eqref{eq:SSC}, and choose $0<\eta<\eta_0$. Since the projective separation \eqref{eq:POSC} is not satisfied, there exist $L \in X_F$ and $\hhh,\kkk \in \Sigma_*$ with $\hhh\ne\kkk$ such that
  \begin{align*}
    |\proj_{A_{\hhh\land\kkk}^\top L^\perp}(\fii_{\iii}(x)-\fii_{\jjj}(x))| \|A_{\hhh\land\kkk}^\top|L^\perp\| &= |\proj_{L^\perp}(\fii_{\hhh}(x)-\fii_{\kkk}(x))| \\
    &\le \eta\diam(\proj_{L^\perp}(X))\|A_{\hhh}^\top|L^\perp\|,
  \end{align*}
  where $\iii = \sigma^{|\hhh\land\kkk|}(\hhh)$ and $\jjj = \sigma^{|\hhh\land\kkk|}(\kkk)$, and
  \begin{align*}
    |\proj_{A_{\hhh\land\kkk}^\top L^\perp}(\fii_{\iii}(x)-\fii_{\jjj}(x))| &\le \eta\diam(\proj_{L^\perp}(X)) \|A_{\iii}^\top|A_{\hhh\land\kkk}^\top L^\perp\| \\
    &\le \eta\diam(X)\alpha_1(A_{\iii})
  \end{align*}
  for all $x \in X$. Since $Z$ is the convex hull of $z_1,z_2,z_3\in X$ we see that for any $x \in Z$ there is a probability vector $(p_1,p_2,p_3)$ such that $x=\sum_ip_iz_i$. Hence,
  \begin{equation} \label{eq:key-est2}
  \begin{split}
    |\proj_{A_{\hhh\land\kkk}^\top L^\perp}(\fii_{\iii}(x)-\fii_{\jjj}(x))| &= \biggl|\sum_ip_i\proj_{A_{\hhh\land\kkk}^\top L^\perp}(\fii_{\iii}(z_i)-\fii_{\jjj}(z_i))\biggr| \\
    &\le \eta\diam(X)\alpha_1(A_{\iii})
  \end{split}
  \end{equation}
  for all $x \in Z$. Moreover, since $\diam(Z) \le \frac{\delta}{3}$ we have
  \begin{equation*}
    |\varphi_\iii(x)-\varphi_\jjj(x)| \ge \frac{\delta}{3}
  \end{equation*}
  for every $x\in Z$. Since
  \begin{align*}
    |\proj_{A_{\hhh\land\kkk}^\top L^\perp}(\fii_{\iii}(x)-\fii_{\jjj}(x))| &= |\fii_{\iii}(x)-\fii_{\jjj}(x)||\cos(\sphericalangle(A_{\hhh\land\kkk}^\top L^\perp,\fii_{\iii}(x)-\fii_{\jjj}(x)))| \\
    &\ge \frac{\delta}{3}|\cos(\sphericalangle(A_{\hhh\land\kkk}^\top L^\perp,\fii_{\iii}(x)-\fii_{\jjj}(x)))|,
  \end{align*}
  the estimate \eqref{eq:key-est2} implies
  \begin{equation} \label{eq:key-est4}
    |\cos(\sphericalangle(A_{\hhh\land\kkk}^\top L^\perp,\fii_{\iii}(x)-\fii_{\jjj}(x)))| \le \eta3\delta^{-1}\diam(X)\alpha_1(A_{\iii})
  \end{equation}
  for all $x \in Z$.

  Write $W = A_{\hhh\land\kkk}^{-1} L \in X_F$ and notice that $W^\perp = A_{\hhh\land\kkk}^\top L^\perp$. Recall that since $\A$ is dominated, Lemma \ref{thm:bochi-morris} implies that there exists a constant $D \ge 1$ such that $\alpha_1(A_{\iii}) \le D\|A_{\iii}^\top|V^\perp\|$ for all $V \in X_F$. Recalling that $X_F$ is compact and perfect, let $P,Q \in X_F$ be such that $\sphericalangle(P,Q) = \diam(X_F) > 0$ and write
  \begin{equation*}
    K = \frac{24D^3}{|\sin(\sphericalangle(P,Q))|\min_{i \in \{1,\ldots,N\}}\alpha_2(A_i)}.
  \end{equation*}
  Then, by \eqref{eq:key-est2}, for any $V \in X_F$ with $|\sin(\sphericalangle(W,V))| \le K\eta\delta^{-1}\diam(X)\alpha_1(A_{\iii})$ we have
  \begin{equation} \label{eq:key-est-upper-bound}
  \begin{split}
    |\proj_{V^\perp}(\fii_{\iii}&(x)-\fii_{\jjj}(x))| = |\fii_{\iii}(x)-\fii_{\jjj}(x)||\cos(\sphericalangle(V^\perp,\fii_{\iii}(x)-\fii_{\jjj}(x))| \\
    &= |\fii_{\iii}(x)-\fii_{\jjj}(x)||\cos(\sphericalangle(W^\perp,\fii_{\iii}(x)-\fii_{\jjj}(x))\cos(\sphericalangle(W,V)) \\
    &\quad\qquad\qquad\qquad\quad\pm \sin(\sphericalangle(W^\perp,\fii_{\iii}(x)-\fii_{\jjj}(x))\sin(\sphericalangle(W,V))| \\
    &\le |\proj_{W^\perp}(\fii_{\iii}(x)-\fii_{\jjj}(x))| + |\fii_{\iii}(x)-\fii_{\jjj}(x)||\sin(\sphericalangle(W,V))| \\
    &\le \eta\diam(X)\alpha_1(A_{\iii}) + K\eta\delta^{-1}\diam(X)(\diam(X)+\delta/3)\alpha_1(A_{\iii}) \\
    &\le (1+K\delta^{-1}(\diam(X)+\delta/3))\diam(X)D\eta\|A_{\iii}^\top|V^\perp\|
  \end{split}
  \end{equation}
  for all $x\in Z$ giving the claimed upper bound for any such $V$.

  Let us then show the lower bound. Since
  \begin{equation*}
    |\sin(\sphericalangle(A_\lll^{-1}P,A_\lll^{-1}Q))| = \frac{|A_\lll^{-1}v \land A_\lll^{-1}w|}{|A_\lll^{-1}v||A_\lll^{-1}w|} = \frac{|\det(A_\lll^{-1})|}{\|A_\lll^{-1}|P\|\|A_\lll^{-1}|Q\|}|\sin(\sphericalangle(P,Q))|,
  \end{equation*}
  where $v \in P$ and $w \in Q$ such that $|v|=1=|w|$, we see that
  \begin{equation} \label{eq:key-est-tmp}
  \begin{split}
    \frac{\alpha_2(A_\lll)}{\alpha_1(A_\lll)}|\sin(\sphericalangle(P,Q))| &\le |\sin(\sphericalangle(A_\lll^{-1}P,A_\lll^{-1}Q))| \\
    &\le D^2 \frac{|\det(A_\lll^{-1})|}{\alpha_1(A_\lll^{-1})^2} = D^2 \frac{\alpha_2(A_\lll)}{\alpha_1(A_\lll)}
  \end{split}
  \end{equation}
  for all $\lll \in \Sigma_*$. Let $\lll \in [\hhh\land\kkk] \subset \Sigma$ be such that $W \in A_{\lll|_n}^{-1}X_F$ for all $n \in \N$. Fix $n \in \N$ such that
  \begin{equation} \label{eq:key-est-choice}
    D^2\frac{\alpha_2(A_{\lll|_{n}})}{\alpha_1(A_{\lll|_{n}})} \le K\eta\delta^{-1}\diam(X)\alpha_1(A_{\iii}) < D^2\frac{\alpha_2(A_{\lll|_{n-1}})}{\alpha_1(A_{\lll|_{n-1}})}
  \end{equation}
  and observe that, by the fact that $\sphericalangle(A_{\lll|_n}^{-1}P,A_{\lll|_n}^{-1}Q) = \diam(A_{\lll|_n}^{-1} X_F)$, the estimates \eqref{eq:key-est-tmp} and \eqref{eq:key-est-choice} show that there exists $V \in \{A_{\lll|_n}^{-1}P,A_{\lll|_n}^{-1}Q\} \subset A_{\lll|_n}^{-1} X_F$ such that
  \begin{equation} \label{eq:key-est7}
    \frac12 \frac{\alpha_2(A_{\lll|_n})}{\alpha_1(A_{\lll|_n})}|\sin(\sphericalangle(P,Q))| \le |\sin(\sphericalangle(W,V))| \le K\eta\delta^{-1}\diam(X)\alpha_1(A_{\iii}).
  \end{equation}
  It follows that the upper bound found in \eqref{eq:key-est-upper-bound} is valid for this $V \in X_F$. As $0<\eta<\eta_0$, we have $\sqrt{1-(\eta3\delta^{-1}\diam(X)\alpha_1(A_{\iii}))^2}>\tfrac12$. Therefore, by \eqref{eq:key-est4}, \eqref{eq:key-est7}, \eqref{eq:key-est2}, and \eqref{eq:key-est-choice},
  \begin{align*}
    |\proj_{V^\perp}(\fii_{\iii}&(x)-\fii_{\jjj}(x))| = |\fii_{\iii}(x)-\fii_{\jjj}(x)||\cos(\sphericalangle(W^\perp,\fii_{\iii}(x)-\fii_{\jjj}(x))\cos(\sphericalangle(W,V)) \\
    &\quad\qquad\qquad\qquad\qquad\pm \sin(\sphericalangle(W^\perp,\fii_{\iii}(x)-\fii_{\jjj}(x))\sin(\sphericalangle(W,V))| \\
    &\ge |\fii_{\iii}(x)-\fii_{\jjj}(x)||\sin(\sphericalangle(W^\perp,\fii_{\iii}(x)-\fii_{\jjj}(x))||\sin(\sphericalangle(W,V))| \\
    &\quad\qquad\qquad\qquad\qquad- |\proj_{W^\perp}(\fii_{\iii}(x)-\fii_{\jjj}(x))||\cos(\sphericalangle(W,V))| \\
    &\ge \frac{\delta}{3} \sqrt{1-(\eta3\delta^{-1}\diam(X)\alpha_1(A_{\iii}))^2} \frac12 \frac{\alpha_2(A_{\lll|_n})}{\alpha_1(A_{\lll|_n})}|\sin(\sphericalangle(P,Q))| \\
    &\quad\qquad\qquad\qquad\qquad- \eta\diam(X)\alpha_1(A_{\iii}) \\
    &\ge \biggl(\frac{K\min_{i \in \{1,\ldots,N\}}\alpha_2(A_i)|\sin(\sphericalangle(P,Q))|}{12D^2} - D\biggr)\diam(X)\eta\|A_{\iii}^\top|V^\perp\|.
  \end{align*}
  The choice of $K$ guarantees that the coefficient obtained in the end is positive and therefore, the claimed lower bound holds for this $V \in X_F$.
\end{proof}

If $X$ is dominated and $\CC \subset \RP$ is a strongly invariant multicone, then, by \eqref{eq:furstenberg-directions-dominated}, we have $\CC \cap X_F = \emptyset$. Fix
\begin{equation} \label{L-line}
  L \in \bigcap_{n=1}^\infty\bigcup_{\iii\in\Sigma_n}A_\iii \CC \subset \CC
\end{equation}
and notice that $L$ is uniformly transverse to every $V \in X_F$. In the second auxiliary lemma, we find small scales containing a large number of equally distributed points. The purpose of $L$ is to be the line along which we equidistribute them.

\begin{lemma}\label{lem:totangent}
	If $X$ is a dominated planar self-affine set satisfying the strong separation condition, but not the projective separation, such that $X_F$ is not a singleton, then there exist $x_0 \in X$ and $C \ge 1$ so that for each $k \in \N$ there are $V_k \in X_F$, a natural number $n_k \ge k$, and finite words $\kkk_0,\ldots,\kkk_{n_k} \in \Sigma_*$ such that
	\begin{equation*}
	  C^{-1}\frac{\|A_{\kkk_0}^\top|V_k^\bot\|}{n_k} \le |\proj_{V_k^\bot}(\fii_{\kkk_{i-1}}(x_0) - \fii_{\kkk_{i}}(x_0))| \le C\frac{\|A_{\kkk_0}^\top|V_k^\bot\|}{n_k}
	\end{equation*}
	for all $i \in \{1,\ldots,n_k\}$. Furthermore, the set $\{\proj_{V_k^\bot}(\fii_{\kkk_{i-1}}(x_0) - \fii_{\kkk_{i}}(x_0)) : i \in \{1,\ldots,n_k\}\}$ is completely contained in one of the two halfplanes determined by $L^\bot$, where $L \in \RP$ is as in \eqref{L-line}.
\end{lemma}

\begin{proof}
  Let $Z$ be as in \eqref{eq:Z-set}. Recalling that $X \cap Z^o \ne \emptyset$, we choose $x_0 \in X \cap Z^o$ and let $\fff\in\Sigma$ be such that $\pi(\fff)=x_0$. By Lemma~\ref{lem:key}, there exist $n_0 \in \N$ and $C \ge 1$ such that for every $n \ge n_0$ there exist $V_n\in X_F$ and $\iii_{n},\jjj_{n}\in\Sigma_*$ with $\iii_n|_1 \ne \jjj_n|_1$ such that
  \begin{equation} \label{eq:totangent1}
    C^{-1}\frac{\|A_{\iii_{n}}^\top|V_n^\bot\|}{n} \le |\proj_{V_n^\bot}(\varphi_{\iii_{n}}(x)-\varphi_{\jjj_{n}}(x))|\leq C\frac{\|A_{\iii_{n}}^\top|V_n^\bot\|}{n}
  \end{equation}
  for all $x \in Z$. Notice also that for every $n \in \N$, by the continuity of the projection $V \mapsto \proj_{V^\bot}(\varphi_{\iii_{n}}(x)-\varphi_{\jjj_{n}}(x))$ and the compactness of $Z$, there exists $r_n>0$ such that
  \begin{equation} \label{eq:totangent3}
    \frac12\leq\frac{|\proj_{V_{n}^\bot}(\varphi_{\iii_{n}}(x)-\varphi_{\jjj_{n}}(x))|}{|\proj_{W^\bot}(\varphi_{\iii_{n}}(x)-\varphi_{\jjj_{n}}(x))|}\leq 2
  \end{equation}
  for all $x \in Z$ and $W \in B(V_{n},r_n)$.

  Recalling \eqref{eq:furstenberg-directions-dominated2}, let $\ddd_n \in \Sigma_*$ be such that
  \begin{equation*}
    A_{\ddd_n}^{-1}(X_F) \subset B(V_{n},r_{n}).
  \end{equation*}
  Since $Z$ is connected, $\proj_{V_n^\bot}(\varphi_{\iii_{n}}(x)-\varphi_{\jjj_{n}}(x))$ is contained in the same halfplane determined by $L^\bot$ for all $x \in Z$. Hence, there exists a strictly increasing sequence $(n_k)_{k \in \N}$ of natural numbers such that the points $\proj_{V_{n_k}^\bot}(\varphi_{\iii_{n_k}}(x)-\varphi_{\jjj_{n_k}}(x))$ are contained in the same halfplane determined by $L^\bot$ for all $x \in Z$ and $k\in\N$. Fix $K \in \N$ and write $M = n_K$. Let $k_1,\ldots,k_M \in \N$ and $l_2,\ldots,l_M \in \N$ be arbitrary. We will specify these numbers later. Let us define words $\kkk_0,\ldots,\kkk_M \in \Sigma_*$ by setting
  \begin{equation}\label{eq:kwords}
  \begin{split}
	  \kkk_0 &= \ddd_{n_{k_M}}\iii_{n_{k_{M}}}\fff|_{l_M}\cdots\ddd_{n_{k_2}}\iii_{n_{k_2}}\fff|_{l_2}\ddd_{n_{k_1}}\iii_{n_{k_1}}, \\
	  &\;\;\vdots \\
	  \kkk_i &= \ddd_{n_{k_M}}\iii_{n_{k_{M}}}\fff|_{l_M}\cdots\ddd_{n_{k_{i+1}}}\iii_{n_{k_{i+1}}}\fff|_{l_{i+1}}\ddd_{n_{k_i}}\jjj_{n_{k_i}}\fff|_{l_i}\cdots\ddd_{n_{k_1}}\jjj_{n_{k_1}}, \\
	  &\;\;\vdots \\
	  \kkk_{M} &= \ddd_{n_{k_M}}\jjj_{n_{k_M}}\fff|_{l_M}\cdots\ddd_{n_{k_2}}\jjj_{n_{k_2}}\fff|_{l_2}\ddd_{n_{k_1}}\jjj_{n_{k_1}}.
\end{split}
\end{equation}
  To simplify notation, write
  \begin{align*}
    \lll_i &= \ddd_{n_{k_K}}\iii_{n_{k_{M}}}\fff|_{l_M}\cdots\ddd_{n_{k_{i+1}}}\iii_{n_{k_{i+1}}}\fff|_{l_{i+1}}\ddd_{n_{k_{i}}}, \\
    \hhh_i &= \iii_{n_{k_{i}}}\fff|_{l_i}\cdots\ddd_{n_{k_1}}\jjj_{n_{k_1}}, \\
    \ggg_i &= \jjj_{n_{k_{i}}}\fff|_{l_i}\cdots\ddd_{n_{k_1}}\jjj_{n_{k_1}},
  \end{align*}
  and observe that then
  \begin{equation} \label{eq:totangent2}
    |\proj_{V_{n_{K}}^\bot}(\fii_{\kkk_{i-1}}(x_0) - \fii_{\kkk_{i}}(x_0))| = |\proj_{A_{\lll_i}^{\top} V_{n_{K}}^\bot}(\fii_{\hhh_i}(x_0) - \fii_{\ggg_i}(x_0))|\|A_{\lll_i}^\top|V_{n_{K}}^\bot\|
  \end{equation}
  for all $i \in \{1,\ldots,M\}$. Without loss of generality, we may assume that $A_{\lll_i}^\top|V_{n_K}^\bot$ is orientation preserving. Indeed, adding an extra symbol in the front of $\ddd_{n_{k_i}}$ leaves the properties of $\ddd_{n_{k_i}}$ defined above unchanged but may change the orientation of $A_{\lll_i}^\top|V_{n_K}^\bot$ if necessary. Thus, $\proj_{V_{n_{K}}^\bot}(\fii_{\kkk_{i-1}}(x_0) - \fii_{\kkk_{i}}(x_0))$ is in the same halfplane determined by $L^\bot$ as $\proj_{A_{\lll_i}^\top V_{n_{K}}^\bot}(\fii_{\hhh_i}(x_0) - \fii_{\ggg_i}(x_0))$. Furthermore, since $A_{\lll_i}^{\top}=A_{\ddd_{n_{k_i}}}^{\top}A_{\ddd_{n_{k_M}}\iii_{n_{k_{M}}}\fff\cdots\ddd_{n_{k_{i+1}}}\iii_{n_{k_{i+1}}}\fff}^{\top}$, we have $A_{\lll_i}^{\top}V_{n_K}^\bot\in B(V_{n_{k_i}}^\bot,r_{n_{k_i}})$. Therefore, if $\fii_{\fff|_{l_i}\cdots\ddd_{n_{k_1}}\jjj_{n_{k_1}}}(x_0) \in Z$, then, by \eqref{eq:totangent2}, \eqref{eq:totangent3}, and \eqref{eq:totangent1},
  \begin{equation*}
  \begin{split}
    (2C)^{-1}\frac{\|A_{\iii_{n_{k_i}}}^{\top}|V_{n_{k_i}}^\bot\|\|A_{\lll_i}^{\top}|V_{n_{K}}^\bot\|}{n_{k_i}} &\le |\proj_{V_{n_{K}}^\bot}(\fii_{\kkk_{i-1}}(x_0) - \fii_{\kkk_{i}}(x_0))| \\
    &\le 2C\frac{\|A_{\iii_{n_{k_i}}}^{\top}|V_{n_{k_i}}^\bot\|\|A_{\lll_i}^{\top}|V_{n_{K}}^\bot\|}{n_{k_i}}
  \end{split}
  \end{equation*}
  for all $i \in \{1,\ldots,M\}$. Hence, to finish the proof, it suffices to recall Lemma \ref{thm:bochi-morris} and choose the numbers $k_1,\ldots,k_{M} \in \N$ such that
  \begin{equation}\label{eq:indneed}
  \begin{split}
    c_1\frac{\alpha_1(A_{\fff|_{l_i}\ddd_{n_{k_{i-1}}}\jjj_{n_{k_{i-1}}}\fff|_{l_{i-1}}\cdots\ddd_{n_{k_1}}\jjj_{n_{k_1}}})}{n_K} &\le \frac{1}{n_{k_i}} \\
    &\le c_2\frac{\alpha_1(A_{\fff|_{l_i}\ddd_{n_{k_{i-1}}}\jjj_{n_{k_{i-1}}}\fff|_{l_{i-1}}\cdots\ddd_{n_{k_1}}\jjj_{n_{k_1}}})}{n_K}
  \end{split}
  \end{equation}
  for some constants $c_1,c_2 > 0$ and the numbers $l_2,\ldots,l_M \in \N$ such that the points $\fii_{\fff|_{l_i}\cdots\ddd_{n_{k_1}}\jjj_{n_{k_1}}}(x_0)$ are in $Z$.

  Let us give the precise definition for the numbers $k_1,k_2,\ldots,k_{M}$ and $l_2,\ldots,l_M$. Let $k_1=K$ and choose $l_2$ to be the smallest integer such that $\varphi_{\fff|_{l_2}}(X)\subset Z$ and
  \begin{equation*}
    \Biggl[\frac{\alpha_1(A_{\fff|_{l_2+1}\ddd_{n_{k_1}}\jjj_{n_{k_1}}}^{\top})}{n_{K}}, \frac{\alpha_1(A_{\fff|_{l_2}\ddd_{n_{k_1}}\jjj_{n_{k_1}}}^{\top})}{n_{K}}\Biggr) \cap \{n_{k_1+1}^{-1},n_{k_1+2}^{-1},\ldots\} \neq \emptyset,
  \end{equation*}
  and choose $k_2$ such that $n_{k_2}$ is the largest element of the above set. We continue inductively. If $k_i$ and $l_i$ has been defined for $i \in \{2,\ldots,M-1\}$, then let $l_{i+1}$ be the smallest integer such that $\varphi_{\fff|_{l_{i+1}}}(X)\subset Z$ and the intersection
  \begin{equation*}
    \Biggl[\frac{\alpha_1(A_{\fff|_{l_{i+1}+1}\ddd_{n_{k_i}}\jjj_{n_{k_i}}\fff|_{l_i}\cdots\ddd_{n_{k_1}}\jjj_{n_{k_1}}}^{\top})}{n_{K}},\frac{\alpha_1(A_{\fff|_{l_{i+1}}\ddd_{n_{k_i}}\jjj_{n_{k_i}}\fff|_{l_i}\cdots\ddd_{n_{k_1}}\jjj_{n_{k_1}}}^{\top})}{n_{K}}\Biggr) \cap \{n_{k_i+1}^{-1},n_{k_i+2}^{-1},\ldots\}
  \end{equation*}
  is non-empty, and choose $k_{i+1}$ such that $n_{k_{i+1}}$ is the largest element of the above set. As these choices clearly satisfy \eqref{eq:indneed}, we have finished the proof.
\end{proof}

The following proposition is a straightforward consequence of Lemma~\ref{lem:totangent}.

\begin{proposition} \label{lem:missing}
  If $X$ is a dominated planar self-affine set satisfying the strong separation condition, but not the projective separation, such that $X_F$ is not a singleton, then
  \begin{equation*}
    \sup\{\#\Sigma(V,x,r) : V \in X_F, \; x \in X, \text{ and } r>0\} = \infty,
  \end{equation*}
  where $\Sigma(V,x,r)$ is as in \eqref{eq:POSC-Sigma-def}.
\end{proposition}

\begin{proof}
  Let $x_0 \in X$ be as in Lemma~\ref{lem:totangent} and let $\jjj\in\Sigma$ be such that $\pi(\jjj)=x_0$. Furthermore, for each $k\in\N$ let $V_k \in X_F$, $n_k$, and the finite words $\kkk_0,\ldots,\kkk_{n_k} \in \Sigma_*$ be as in Lemma~\ref{lem:totangent}. Recall that, by \eqref{eq:posdiam}, $\inf_{V\in X_F}\diam(\proj_{V^\bot}(X))>0$. Fix $k \in \N$ and choose $q \in \{1,\ldots,n_k\}$ such that
  \begin{equation*}
    C\frac{q}{n_k}<\inf_{V\in X_F}\diam(\proj_{V^\bot}(X))\leq C\frac{q+1}{n_k}.
  \end{equation*}
  Write $r_k = \inf_{V\in X_F}\diam(\proj_{V^\bot}(X))\|A_{\kkk_0}^\top|V_k^\bot\|>0$ and $\lll_i=\kkk_i\jjj|_{m_i}$ for all $i \in \{0,\ldots,q\}$, where $m_i$ is chosen such that
  \begin{equation*}
    \diam(\proj_{V_k^\perp}(\fii_{\lll_i}(X))) < r_k \le \diam(\proj_{V_k^\perp}(\fii_{(\lll_i)^-}(X))).
  \end{equation*}
  Since
  \begin{equation*}
    \diam(\proj_{V_k^\perp}(\fii_{\kkk_0}(X)))=\|A_{\kkk_0}^\top|V_k^\perp\|\diam(\proj_{A_{\kkk_0}^\top V_k^\perp}(X)),
  \end{equation*}
  we see that $m_0>|\kkk_0|$. We claim that $\lll_i\neq\lll_j$ for all $i,j\in\{0,\ldots,q\}$ with $i<j$. Indeed, if $\lll_i=\lll_j$ for some $i<j$, then, by the construction of the words $\kkk_0,\ldots,\kkk_{n_k}$ in \eqref{eq:kwords}, we have $m_i=m_j\leq|\kkk_i\wedge\kkk_j|<|\kkk_0\wedge\kkk_i|$ implying $m_i=m_0<|\kkk_0|$, which is a contradiction. By telescoping, we have
  \begin{align*}
    |\proj_{V_k^\bot}(\fii_{\kkk_{i}}(x_0) - \fii_{\kkk_{0}}(x_0))| &\le \sum_{j=1}^i |\proj_{V_k^\bot}(\fii_{\kkk_{i}}(x_0) - \fii_{\kkk_{i-1}}(x_0))| \le C\frac{q}{n_k}\|A_{\kkk_0}^\top|V_k^\bot\| \\
    &\le \inf_{V\in X_F}\diam(\proj_{V^\bot}(X))\|A_{\kkk_0}^\top|V_k^\bot\| = r_k.
  \end{align*}
  for all $i \in \{1,\ldots,q\}$. Thus, $\lll_i\in\Sigma(V_k,\varphi_{\kkk_0}(x_0),r_k$ for all $i \in \{1,\ldots,q\}$, and so
  \begin{equation*}
    \#\Sigma(V_k,\varphi_{\kkk_0}(x_0),r_k) \ge q \ge n_kC^{-1}\inf_{V\in X_F}\diam(\proj_{V^\bot}(X))-1.
  \end{equation*}
  Since $n_k\to\infty$ as $k\to\infty$, the claim follows.
\end{proof}

Finally, we apply Lemma \ref{lem:totangent} to show that a line segment appears as a subset of a weak tangent set.

\begin{theorem} \label{thm:loweronassuad}
  If $X$ is a dominated planar self-affine set satisfying the strong separation condition, but not the projective separation, such that $X_F$ is not a singleton, then $\dima(X) \ge 1$.
\end{theorem}

\begin{proof}
	Let $x_0\in X$ and $C \ge 1$ be as in Lemma \ref{lem:totangent}. Fix $k \in \N$ and let $V_k \in X_F$, $n_k \ge k$, and $\kkk_0,\ldots,\kkk_{n_k} \in \Sigma_*$ be as in Lemma~\ref{lem:totangent}. Write $B_k = \{\fii_{\kkk_0}(x_0),\ldots,\fii_{\kkk_{n_k}}(x_0)\}$ and let $L$ be as in \eqref{L-line}. Let $a_i^L\in L$ and $a_i^V\in V_k$ be such that $a_i^L+a_i^V=\varphi_{\kkk_{i}}(x_0)-\varphi_{\kkk_{i-1}}(x_0)$ for all $i \in \{1,\ldots,n_k\}$.	Recalling \eqref{eq:furstenberg-directions-dominated2}, let $\lll_{k} \in \Sigma$ be such that $V_k=\Pi(\lll_k)$. Let $D \ge 1$ be as in Lemma \ref{thm:bochi-morris}. Relying on the definition of domination, choose $m_k$ to be the smallest integer for which
  \begin{equation} \label{eq:loweronassouad1}
    D\alpha_2(A_{\overleftarrow{\lll_k|_{m_k}}})\diam(X)\leq\frac{1}{n_k^2}\alpha_1(A_{\overleftarrow{\lll_k|_{m_k}}})\|A_{\kkk_0}^\top|V_k^\bot\|.
  \end{equation}
  By Lemma \ref{thm:bochi-morris},
  \begin{equation} \label{eq:loweronassouad2}
    |A_{\overleftarrow{\lll_k|_{m_k}}}a_i^V| = \|A_{\overleftarrow{\lll_k|_{m_k}}}|V_k\||a_i^V| \le D\alpha_2(A_{\overleftarrow{\lll_k|_{m_k}}})\diam(X).
  \end{equation}
  Therefore, by \eqref{eq:loweronassouad2} and \eqref{eq:loweronassouad1}, we have
  \begin{equation} \label{eq:loweronassouad3}
    \biggl|\sum_{j=1}^iA_{\overleftarrow{\lll_k|_{m_k}}}a_{j}^V\biggr| \leq \frac{1}{n_k}\alpha_1(A_{\overleftarrow{\lll_k|_{m_k}}})\|A_{\kkk_0}^\top|V_k^\bot\|
  \end{equation}
  for all $i\in \{1,\ldots,n_k\}$. Clearly, each $A_{\overleftarrow{\lll_k|_{m_k}}}a_{i}^L$ is contained in the subspace $A_{\overleftarrow{\lll_k|_{m_k}}}L$ and hence, by \eqref{eq:loweronassouad3},
  \begin{equation} \label{eq:loweronassouad4}
  \begin{split}
    \dist(\varphi_{\overleftarrow{\lll_k|_{m_k}}}(\varphi_{\kkk_i}(x_0))&-\varphi_{\overleftarrow{\lll_k|_{m_k}}}(\varphi_{\kkk_0}(x_0)),A_{\overleftarrow{\lll_k|_{m_k}}}L) \\
    &\le \biggl|\varphi_{\overleftarrow{\lll_k|_{m_k}}}(\varphi_{\kkk_i}(x_0))-\varphi_{\overleftarrow{\lll_k|_{m_k}}}(\varphi_{\kkk_0}(x_0))-\sum_{j=1}^iA_{\overleftarrow{\lll_k|_{m_k}}}a_{j}^L\biggr| \\
    &= \biggl|\sum_{j=1}^iA_{\overleftarrow{\lll_k|_{m_k}}}(\varphi_{\kkk_{j}}(x_0)-\varphi_{\kkk_{j-1}}(x_0))-\sum_{j=1}^iA_{\overleftarrow{\lll_k|_{m_k}}}a_{j}^L\biggr| \\
    &=\biggl|\sum_{j=1}^iA_{\overleftarrow{\lll_k|_{m_k}}}a_{j}^V\biggr| \le \frac{1}{n_k}\alpha_1(A_{\overleftarrow{\lll_k|_{m_k}}})\|A_{\kkk_0}^\top|V_k^\bot\|.
  \end{split}
  \end{equation}
  Thus, $\varphi_{\overleftarrow{\lll_k|_{m}}}(B_k)$ is in the $\frac{1}{n_k}\alpha_1(A_{\overleftarrow{\lll_k|_{m_k}}})\|A_{\kkk_0}^\top|V_k\|$-neighbourhood of the line $A_{\overleftarrow{\lll_k|_{m_k}}}L+\varphi_{\overleftarrow{\lll_k|_{m_k}}}(\varphi_{\kkk_0}(x_0))$.

  By Lemma \ref{lem:totangent},
  \begin{equation*}
    C^{-1}\frac{\|A_{\kkk_0}^\top|V_k^\bot\|}{n_k} \le |\proj_{V_k^\bot}(\fii_{\kkk_{i-1}}(x_0) - \fii_{\kkk_{i}}(x_0))| \le C\frac{\|A_{\kkk_0}^\top|V_k^\bot\|}{n_k}
  \end{equation*}
  for all $i \in \{1,\ldots,n_k\}$. Since $L$ is uniformly transverse to every $V \in X_F$, we may, by possibly adjusting the constant $C$, replace $|\proj_{V_k^\bot}(\fii_{\kkk_{i-1}}(x_0) - \fii_{\kkk_{i}}(x_0))|$ in the above estimate by $|a_i^L|$. Therefore, since $L \in \CC$ and $\CC$ is a strongly invariant multicone, \cite[Lemma~2.2]{BochiMorris2015} implies that $\|A_\iii|L\| \ge D^{-1}\alpha_1(A_\iii)$ for all $\iii \in \Sigma_*$, and so
  \begin{equation} \label{eq:loweronassouad5}
    (CD)^{-1}\frac{\alpha_1(A_{\overleftarrow{\lll_k|_{m_k}}})\|A_{\kkk_0}^\top|V_k^\bot\|}{n_k} \le |A_{\overleftarrow{\lll_k|_{m_k}}}a_i^L| \le C\frac{\alpha_1(A_{\overleftarrow{\lll_k|_{m_k}}})\|A_{\kkk_0}^\top|V_k^\bot\|}{n_k}
  \end{equation}
  for all $i \in \{1,\ldots,n_k\}$. Let
  \begin{align*}
    T_k &= M_{\varphi_{\overleftarrow{\lll_k|_{m_k}}}(\varphi_{\kkk_0}(x_0)),\alpha_1(A_{\overleftarrow{\lll_k|_{m_k}}})\|A_{\kkk_0}^\top|V_k\|}(X)\cap B(0,1) \\
    &\supset M_{\varphi_{\overleftarrow{\lll_k|_{m_k}}}(\varphi_{\kkk_0}(x_0)),\alpha_1(A_{\overleftarrow{\lll_k|_{m_k}}})\|A_{\kkk_0}^\top|V_k\|}(\varphi_{\overleftarrow{\lll_k|_{m_k}}}(B_k))\cap B(0,1).
  \end{align*}
  Since the vectors $a_i^L$ are contained in the same halfplane determined by $L^\bot$, there exists, by possibly going through a subsequence, a weak tangent set $T$ such that $T_{k} \to T$ as $k \to \infty$ in Hausdorff distance such that, by \eqref{eq:loweronassouad4} and \eqref{eq:loweronassouad5}, $T$ contains a line segment of length at least $(CD)^{-1}$. The claim follows now from Lemma \ref{thm:KOR}.
\end{proof}

\section{Hausdorff measure of self-affine sets having small Hausdorff dimension} \label{sec:hausdorff-measure}

In this section, we conclude the proof of Theorem \ref{thm:ahl}. In Section \ref{sec:assouad-small}, we proved that the conditions of Theorem \ref{thm:ahl} satisfy \eqref{(1)} $\Leftrightarrow$ \eqref{(11)} $\Rightarrow$ \eqref{(4)} $\Rightarrow$ \eqref{(6)} and, if the Hausdorff dimension is strictly smaller than one, also \eqref{(6)} $\Rightarrow$ \eqref{(1)}. Therefore, it suffices to show that \eqref{(4)} $\Rightarrow$ \eqref{(2)} $\Rightarrow$ \eqref{(3)} $\Rightarrow$ \eqref{(5)} $\Rightarrow$ \eqref{(11)}.

Recall that the implication \eqref{(4)} $\Rightarrow$ \eqref{(2)} holds by Lemma \ref{thm:ahlfors-implications}. Proofs of the remaining claims are structured below as follows: Recalling Theorem \ref{thm:BHR2}, Theorem \ref{lem:lowb} verifies the implication \eqref{(2)} $\Rightarrow$ \eqref{(3)} and Theorem \ref{thm:kaenmaki-ahlfors} shows the implications \eqref{(3)} $\Leftrightarrow$ \eqref{(5)} $\Rightarrow$ \eqref{(11)}.

Before proving the aforementioned results, let us verify a technical lemma. Let $X$ be a dominated planar self-affine set, $s = \dimaff(X) \le 1$, and, to simplify notation, define a function $H \colon \Sigma \to [0,\infty)$ by setting
\begin{equation*}
  H(\iii) = \HH^s_\infty(\proj_{\Pi(\iii)^\perp}(X))
\end{equation*}
for all $\iii \in \Sigma$. The function $H$ measures the $s$-dimensional Hausdorff content of the projection in different directions. The fact that $H$ takes finite values follows from Lemma \ref{thm:FK}. Recall that the Perron-Frobenius operator $\LL$ for $s$ is the positive linear operator defined by setting
\begin{equation*}
  \mathcal{L}f(\iii) = \sum_{i=1}^N \|A_{i}^\top|\Pi(\iii)^\perp\|^s f(i\iii)
\end{equation*}
for all continuous functions $f \colon \Sigma \to \R$. Let the continuous function $h \colon \Sigma \to (0,\infty)$ and the Borel probability measure $\nu$ on $\Sigma$ be as in Lemma \ref{thm:perron-frobenius}. Although omitted in notation, bear in mind that $H$, $\LL$, $h$, and $\nu$ depend on $s$. The lemma links the Hausdorff content of the projections to the Perron-Frobenius operator by using techniques from the standard measure theory.

\begin{lemma} \label{prop:important}
  If $X$ is a dominated planar self-affine set with $s = \dimaff(X) \le 1$, then $H$ is continuous and
  \begin{equation} \label{eq:H-expression}
    H(\iii) = h(\iii) \int_{\Sigma} H(\jjj) \dd\nu(\jjj)
  \end{equation}
  for all $\iii \in \Sigma$. In particular, either $\max_{\iii \in \Sigma} H(\iii) = 0$ or $\min_{\iii \in \Sigma} H(\iii) > 0$.
\end{lemma}

\begin{proof}
  Let us first show that $H$ is upper semi-continuous. Fix $\jjj \in \Sigma$ and choose $\eps>0$. By the definition of the Hausdorff content, there exists a countable open cover $\{U_i\}_i$ of $\proj_{\Pi(\jjj)^\bot}(X)$ such that
  \begin{equation*}
    \sum_{i} \diam(U_i)^s \le H(\jjj) + \eps.
  \end{equation*}
  Since $\proj_{\Pi(\jjj)^\bot}(X)$ is compact, we may assume that the cover $\{U_i\}_i$ is finite. Write $\roo = \dist(\proj_{\Pi(\jjj)^\bot}(X),\R\setminus\bigcup_{i}U_i) > 0$ and $r = \arcsin(\roo/\diam(X))$. Hence, for every $\iii \in \Sigma$ with $\sphericalangle(\Pi(\iii),\Pi(\jjj)) < r$, we have $|\proj_{\Pi(\iii)^\bot}(x)-\proj_{\Pi(\jjj)^\bot}(x)| < \roo$ for all $x \in X$. This means that $\{U_i\}_i$ covers each such $\proj_{\Pi(\iii)^\bot}(X)$ and consequently,
  \begin{equation*}
    \sup_{\sphericalangle(\Pi(\jjj),\Pi(\iii)) < r} H(\iii) \le \sum_i \diam(U_i)^s \le H(\jjj) + \eps.
  \end{equation*}
  The claim follows now by letting $\roo \downarrow 0$, in which case also $r \downarrow 0$, and then $\eps \downarrow 0$.

  As $H$ is upper semi-continuous, a standard measure-theoretical argument applying Urysohn's lemma shows that for each $n \in \N$ there exists a continuous function $f_n \colon \Sigma \to \R$ such that $f_n \ge H$ and
  \begin{equation} \label{eq:usc-urysohn}
    \int_\Sigma f_n(\iii) \dd\nu(\iii) \le \int_\Sigma H(\iii) \dd\nu(\iii) + \frac{1}{n}.
  \end{equation}
  Since $\proj_V(Ax+y) = \pm\|A^\top|V\||\proj_{A^\top V}(x)|v^\bot+\proj_V(y)$ for all $V \in \RP$, $A \in \GL_2(\R)$, and $x,y \in \R^2$, where $v^\bot$ is a unit vector in $V^\bot$ and $\pm$ indicates that the equality holds with one of the signs, we have
  \begin{equation}\label{eq:mon}
    H(\iii) = \HH^s_\infty\biggl(\bigcup_{i=1}^N \proj_{\Pi(\iii)^\perp}(\varphi_i(X))\biggr) \le \sum_{i=1}^N \|A_i^\top|\Pi(\iii)^\perp\|^s H(i\iii) = \LL H(\iii)
  \end{equation}
  for all $\iii \in \Sigma$. Applying \eqref{eq:mon}, Lemma \ref{thm:perron-frobenius}, and \eqref{eq:usc-urysohn}, we see that
  \begin{equation*}
    H(\iii) \le \LL^k H(\iii) \le \LL^k f_n(\iii) \to h(\iii)\int_\Sigma f_n(\jjj) \dd\nu(\jjj) \le h(\iii)\biggl(\int_\Sigma H(\jjj) \dd\nu(\jjj) + \frac{1}{n}\biggr)
  \end{equation*}
  uniformly as $k \to \infty$ for all $\iii \in \Sigma$ and $n \in \N$. By letting $n \to \infty$, it follows that
  \begin{equation} \label{eq:monmon}
    H(\iii) \le h(\iii) \int_\Sigma H(\jjj) \dd\nu(\jjj)
  \end{equation}
  for all $\iii \in \Sigma$. To show the equality, write
  \begin{equation*}
    \Gamma_n = \biggl\{\iii\in\Sigma : H(\iii) \le h(\iii)\int_\Sigma H(\jjj)\dd\nu(\jjj)-\frac{1}{n}\biggr\}
  \end{equation*}
  for all $n \in \N$ and observe that, by the definition of $\Gamma_n$, \eqref{eq:monmon}, and Lemma \ref{thm:perron-frobenius},
  \begin{align*}
    \int_\Sigma H(\iii)&\dd\nu(\iii) = \int_{\Gamma_n} H(\iii)\dd\nu(\iii) + \int_{\Sigma \setminus \Gamma_n} H(\iii)\dd\nu(\iii) \\
    &\le \int_{\Gamma_n} h(\iii)\dd\nu(\iii) \int_\Sigma H(\jjj)\dd\nu(\jjj) - \frac{\nu(\Gamma_n)}{n} + \int_{\Sigma \setminus \Gamma_n} h(\iii)\dd\nu(\iii) \int_\Sigma H(\jjj)\dd\nu(\jjj) \\
    &= \int_{\Sigma} h(\iii)\dd\nu(\iii) \int_\Sigma H(\jjj)\dd\nu(\jjj) - \frac{\nu(\Gamma_n)}{n} = \int_\Sigma H(\jjj)\dd\nu(\jjj) - \frac{\nu(\Gamma_n)}{n}.
  \end{align*}
  It follows that $\nu(\Gamma_n)=0$ for all $n \in \N$ and consequently,
  \begin{equation*}
    \nu\biggl(\biggl\{\iii\in\Sigma : H(\iii) < h(\iii)\int_\Sigma H(\jjj)\dd\nu(\jjj)\biggr\}\biggr) = \nu\biggl(\bigcup_{n \in \N}\Gamma_n\biggr) \le \sum_{n \in \N} \nu(\Gamma_n) = 0.
  \end{equation*}
  Therefore, by \eqref{eq:monmon},
  \begin{equation} \label{eq:H=h}
    H(\iii) = h(\iii) \int_\Sigma H(\jjj) \dd\nu(\jjj)
  \end{equation}
  for $\nu$-almost all $\iii \in \Sigma$. To see that this holds for all words, fix $\iii\in\Sigma$. Since $\nu$ is fully supported, there exists a sequence $(\jjj_n)_{n \in \N}$ of words in $\Sigma$ converging to $\iii$ such that \eqref{eq:H=h} is holds for each $\jjj_n$. Thus, by the continuity of $h$, \eqref{eq:H=h}, the upper semi-continuity of $H$, and \eqref{eq:monmon},
  \begin{align*}
    h(\iii) \int_\Sigma H(\jjj) \dd\nu(\jjj) &= \lim_{n \to \infty} h(\jjj_n) \int_\Sigma H(\jjj) \dd\nu(\jjj) \\
    &= \lim_{n\to\infty}H(\jjj_n) \le H(\iii) \le h(\iii) \int_\Sigma H(\jjj) \dd\nu(\jjj)
  \end{align*}
  which finishes the proof of \eqref{eq:H-expression}. Since $h$ is continuous, \eqref{eq:H-expression} implies that also $H$ is continous, and therefore, the maximum $\max_{\iii \in \Sigma} H(\iii)$ and the minimum $\min_{\iii \in \Sigma} H(\iii)$ exist. Furthermore, since $h(\iii)>0$ for all $\iii \in \Sigma$, the dichotomy in the last claim follows immediately from the first claim.
\end{proof}

In the following theorem, we bound the Hausdorff measure above by the maximal Hausdorff content of the projections. The proof of the bound is a covering argument relying on domination. The latter claim follows as an immediate application of Lemma \ref{prop:important}.

\begin{theorem} \label{lem:lowb}
  If $X$ is a dominated planar self-affine set with $s = \dimaff(X) \le 1$, then there exists a constant $c>0$ such that
  \begin{equation*}
    \HH^s(X) \le c\max_{\iii \in \Sigma} H(\iii).
  \end{equation*}
  In particular, if $\HH^s(X)>0$, then $\min_{\iii \in \Sigma}H(\iii)>0$.
\end{theorem}

\begin{proof}
  Let us first prove the claimed inequality. Fix $\eps>0$ and notice that for every $V \in X_F$ there exists an open cover $\UU_V$ of $\proj_{V^\bot}(X)$ such that
  \begin{equation} \label{eq:lowb1}
    \sum_{U \in \UU_V} \diam(U)^s \le \HH_\infty^s(\proj_{V^\bot}(X)) + \eps.
  \end{equation}
  Since $\proj_{V^\bot}(X)$ is compact, we may assume that $\UU_V$ is finite and $\max_{U \in \UU_V} \diam(U) \le \diam(X)$. Write $\roo_V = \min_{U \in \UU_V} \diam(U) > 0$ and $\delta_V = \dist(\proj_{V^\bot}(X), V^\bot \setminus \bigcup \UU_V) > 0$. Since $|(\proj_{V^\bot}-\proj_{W^\bot})(x)| = \sin(\sphericalangle(V,W))|x|$ for all $x \in \R^2$ and $V,W \in \RP$, and since $X_F$ is compact, there exist $m \in \N$ and $V_1,\ldots,V_m \in X_F$ such that
  \begin{equation*}
    X_F \subset \bigcup_{i=1}^m B(V_i,\arcsin(\delta_{V_i}/\diam(X))).
  \end{equation*}
  Thus, for every $V \in X_F$ there exists $i \in \{1,\ldots,m\}$ such that $|(\proj_{V^\bot}-\proj_{V_i^\bot})(x)| < \delta_{V_i}$ for all $x \in X$, and in particular, $\UU_{V_i}$ covers $\proj_{V^\perp}(X)$. Let $P(V)$ denote this $V_i$ and write $\roo = \min_{i \in \{1,\ldots,m\}} \roo_{V_i}$.

  Define
  \begin{align*}
    \Sigma(r) = \{\iii\in\Sigma_* :\;&\alpha_2(A_{\iii})\diam(X) < \roo\alpha_1(A_{\iii}) \text{ and } \alpha_1(A_\iii)\diam(X)<r \\
    &\text{but } \alpha_2(A_{\iii^-})\diam(X) \ge \roo\alpha_1(A_{\iii^-}) \text{ or } \alpha_1(A_{\iii^-})\diam(X) \ge r\}
  \end{align*}
  for all $r>0$. Note that $\{[\iii] : \iii \in \Sigma(r)\}$ is a partition of $\Sigma$ for all $r>0$. Fix $r>0$ and for each $\iii \in \Sigma(r)$ choose $\iii' \in \Sigma$ such that $\iii' = \overleftarrow{\iii}\jjj$ for some $\jjj \in \Sigma$; the estimates below are independent of this choice. Hence,
  \begin{equation*}
    \bigcup_{\iii \in \Sigma(r)} \bigcup_{U \in \UU_{P(\Pi(\iii'))}} \fii_\iii(X \cap (\proj_{\Pi(\iii')^\bot})^{-1}(U))
  \end{equation*}
  is clearly a cover of $X$. By Lemma \ref{thm:bochi-morris} and the definition of $\Sigma(r)$, we have 
  \begin{equation*}
    \|A_\iii|\Pi(\iii')\|\diam(X) \le D\alpha_2(A_\iii)\diam(X) \le D\roo\alpha_1(A_\iii),
  \end{equation*}
  where $D \ge 1$ is the constant from the lemma. Since $\|A_\iii|\Pi(\iii')^\perp\| \le \alpha_1(A_\iii)$ and $\roo \le \diam(U)$, by decomposing along $\Pi(\iii')$ and $\Pi(\iii')^\bot$ and applying the triangle inequality, we thus see that
  \begin{equation} \label{eq:lowb2}
  \begin{split}
    \diam(\fii_\iii(X &\cap (\proj_{\Pi(\iii')^\bot})^{-1}(U))) \\
    &\le \|A_{\iii}|\Pi(\iii')\|\diam(X) + \|A_{\iii}|\Pi(\iii')^\perp\|\diam(U) \\
    &\le \alpha_1(A_\iii)(D\roo+\diam(U))\le 2D\alpha_1(A_\iii)\diam(U) \le 2Dr
  \end{split}
  \end{equation}
  for all $\iii \in \Sigma(r)$. Recall also that, by Lemma \ref{thm:equilibrium-state-existence}, there exist a measure $\mu_K \in \MM_\sigma(\Sigma)$ and a constant $C \ge 1$ such that
  \begin{equation} \label{eq:yet-another-gibbs-type-measure}
    C^{-1}\alpha_1(A_\iii)^s \le \mu_K([\iii]) \le C\alpha_1(A_\iii)^s
  \end{equation}
  for all $\iii \in \Sigma_*$. Therefore, by \eqref{eq:lowb2}, \eqref{eq:lowb1}, and \eqref{eq:yet-another-gibbs-type-measure},
  \begin{align*}
    \HH^s_{2Dr}(X) &\le \sum_{\iii \in \Sigma(r)} \sum_{U \in \UU_{P(\Pi(\iii'))}} \diam(\fii_\iii(X \cap (\proj_{\Pi(\iii')^\bot})^{-1}(U)))^s \\
    &\le (2D)^s \sum_{\iii \in \Sigma(r)} \sum_{U \in \UU_{P(\Pi(\iii'))}} \alpha_1(A_\iii)^s \diam(U)^s \\
    &\le (2D)^s \sum_{\iii \in \Sigma(r)} \alpha_1(A_\iii)^s (\HH_\infty^s(\proj_{P(\Pi(\iii'))^\bot}(X)) + \eps) \\
    &\le C(2D)^s (\max_{\kkk \in \Sigma} H(\kkk) + \eps) \sum_{\iii \in \Sigma(r)} \mu_K([\iii]) \\
    &= C(2D)^s (\max_{\kkk \in \Sigma} H(\kkk) + \eps).
  \end{align*}
  The first claim follows by letting $r \downarrow 0$ and then $\eps \downarrow 0$. In particular, if $\HH^s(X)>0$, then, by the first claim, also $\max_{\iii \in \Sigma}H(\iii)>0$. By Lemma \ref{prop:important}, this implies $\min_{\iii \in \Sigma}H(\iii)>0$ and verifies the second claim.
\end{proof}

By Lemma \ref{prop:important}, the Perron-Frobenius operator can be used to study coverings. We use this observation in the following theorem to show that the Hausdorff measure and content of the projections are equal. Together with the existence of the equilibrium state, this introduces a way to prove the Ahlfors regularity of the projected self-affine set.

\begin{theorem} \label{thm:kaenmaki-ahlfors}
  If $X$ is a dominated planar self-affine set with $s = \dimaff(X) \le 1$, then $\min_{\iii \in \Sigma}H(\iii)>0$ if and only if $\proj_{V^\bot}(X)$ is Ahlfors $s$-regular for all $V \in X_F$. Furthermore, if $\min_{\iii \in \Sigma}H(\iii)>0$, then
  \begin{equation*}
    \sup\{\#\Sigma(V,x,r) : V \in X_F, \; x \in X, \text{ and } r>0\} < \infty,
  \end{equation*}
  where $\Sigma(V,x,r)$ is as in \eqref{eq:POSC-Sigma-def}.
\end{theorem}

\begin{proof}
  Let us first assume that $\min_{\iii \in \Sigma}H(\iii)>0$ and show the measure $\HH^s|_{\proj_{V^\bot}(X)}$ Ahlfors $s$-regular for all $V \in X_F$. To show the required upper estimate in \eqref{eq:ahlfors-regularity-def}, fix $\delta>0$ and choose $n \in \N$ such that $\diam(\fii_\jjj(X))<\delta$ for all $\jjj \in \Sigma_n$. By Lemmas \ref{prop:important} and \ref{thm:perron-frobenius}, we have
  \begin{equation} \label{eq:kaenmaki-ahlfors0}
  \begin{split}
    \HH^s_\delta(\proj_{\Pi(\iii)^\perp}(X)) &= \HH^s_\delta\biggl(\bigcup_{\jjj\in\Sigma_n}\proj_{\Pi(\iii)^\perp}(\varphi_\jjj(X))\biggr) \\
    &\le \sum_{\jjj\in\Sigma_n}\HH^s_\infty(\proj_{\Pi(\iii)^\perp}(\varphi_\jjj(X))) \\
    &= \sum_{\jjj\in\Sigma_n}\|A_\jjj^\top|\Pi(\iii)^\perp\|^s H(\overleftarrow{\jjj}\iii) \\
    &= \LL^n H(\iii) = \LL^n h(\iii) \int_{\Sigma} H(\jjj) \dd\nu(\jjj) \\
    &= h(\iii) \int_{\Sigma} H(\jjj) \dd\nu(\jjj) = H(\iii).
  \end{split}
  \end{equation}
  By letting $\delta \downarrow 0$, we see that
  \begin{equation} \label{eq:kaenmaki-ahlfors1}
    \HH^s(\proj_{\Pi(\iii)^\bot}(X)) = H(\iii)
  \end{equation}
  for all $\iii\in\Sigma$. In particular, it follows from \eqref{eq:kaenmaki-ahlfors1} that
  \begin{equation} \label{eq:kaenmaki-ahlfors2}
  \begin{split}
    \HH^s(\proj_{\Pi(\iii)^\bot}&(X) \cap B(\proj_{\Pi(\iii)^\bot}(x),r)) \\
    &= \HH^s(\proj_{\Pi(\iii)^\bot}(X)) - \HH^s(\proj_{\Pi(\iii)^\bot}(X) \setminus B(\proj_{\Pi(\iii)^\bot}(x),r)) \\
    &\le H(\iii) - \HH_\infty^s(\proj_{\Pi(\iii)^\bot}(X) \setminus B(\proj_{\Pi(\iii)^\bot}(x),r)) \\
    &\le \HH_\infty^s(\proj_{\Pi(\iii)^\bot}(X) \cap B(\proj_{\Pi(\iii)^\bot}(x),r)) \le (2r)^s
  \end{split}
  \end{equation}
  for all $x \in X$, $r>0$, and $\iii \in \Sigma$.

  To show the measure $\HH^s|_{\proj_{V^\bot}(X)}$ satisfies the required lower estimate in \eqref{eq:ahlfors-regularity-def}, fix $\hhh, \kkk \in \Sigma_*$ such that $[\hhh] \cap [\kkk] = \emptyset$ and write $n = \max\{|\hhh|,|\kkk|\}$. Let $\Sigma_n' = \{\hhh,\kkk\} \cup \{\jjj \in \Sigma_n : [\jjj] \cap [\hhh] = \emptyset \text{ and } [\jjj] \cap [\kkk] = \emptyset\}$. Relying on \eqref{eq:kaenmaki-ahlfors1} and \eqref{eq:kaenmaki-ahlfors0}, we get
  \begin{align*}
    \HH^s&(\proj_{\Pi(\iii)^\bot}(X)) = \HH^s\biggl(\bigcup_{\jjj \in \Sigma_n'} \proj_{\Pi(\iii)^\bot}(\fii_\jjj(X))\biggr) \\
    &\le \sum_{\jjj \in \Sigma_n'} \HH^s(\proj_{\Pi(\iii)^\bot}(\fii_\jjj(X))) - \HH^s(\proj_{\Pi(\iii)^\bot}(\fii_\hhh(X)) \cap \proj_{\Pi(\iii)^\bot}(\fii_\kkk(X))) \\
    &\le \LL^n H(\iii) - \HH^s(\proj_{\Pi(\iii)^\bot}(\fii_\hhh(X)) \cap \proj_{\Pi(\iii)^\bot}(\fii_\kkk(X))) \\
    &= \HH^s(\proj_{\Pi(\iii)^\bot}(X)) - \HH^s(\proj_{\Pi(\iii)^\bot}(\fii_\hhh(X)) \cap \proj_{\Pi(\iii)^\bot}(\fii_\kkk(X))).
  \end{align*}
  It follows that
  \begin{equation} \label{eq:kaenmaki-ahlfors3}
    \HH^s(\proj_{\Pi(\iii)^\bot}(\fii_\hhh(X)) \cap \proj_{\Pi(\iii)^\bot}(\fii_\kkk(X))) = 0
  \end{equation}
  for all $\hhh, \kkk \in \Sigma_*$ with $[\hhh] \cap [\kkk] = \emptyset$ and $\iii \in \Sigma$. Recall that, by Lemma \ref{thm:equilibrium-state-existence}, there exist a measure $\mu_K \in \MM_\sigma(\Sigma)$ and a constant $C \ge 1$ such that
  \begin{equation} \label{eq:yet-another-gibbs-type-measure2}
    C^{-1}\alpha_1(A_\iii)^s \le \mu_K([\iii]) \le C\alpha_1(A_\iii)^s
  \end{equation}
  for all $\iii \in \Sigma_*$. Fix $V \in X_F$ and notice first that, by \eqref{eq:ahlfors-lower-part} and \eqref{eq:kaenmaki-ahlfors2}, there exists a constant $c > 0$ such that
  \begin{equation} \label{eq:kaenmaki-ahlfors4}
    \pi_*\mu_K(B(x,r)) \ge cr^s \quad \text{and} \quad \HH^s|_{\proj_{V^\bot}(X)}(B(x,r)) \le 2^sr^s
  \end{equation}
  for all $x \in X$ and $0<r<\diam(X)$. Let $\Sigma(V,x,r)$ be as in \eqref{eq:POSC-Sigma-def} for all $x \in X$ and $r>0$. Observe that, by \eqref{eq:kaenmaki-ahlfors4}, \eqref{eq:yet-another-gibbs-type-measure2}, Lemma \ref{thm:bochi-morris}, \eqref{eq:kaenmaki-ahlfors1}, and \eqref{eq:kaenmaki-ahlfors2}, we have
  \begin{equation}\label{eq:extra}
  \begin{split}
    c\min_{\kkk \in \Sigma}H(\kkk) r^s &\le \min_{\kkk \in \Sigma}H(\kkk) \pi_*\mu_K(B(x,r)) \\
    &\le \min_{\kkk \in \Sigma}H(\kkk) \sum_{\jjj \in \Sigma(V,x,r)} \mu_K([\jjj]) \le C\min_{\kkk \in \Sigma}H(\kkk) \sum_{\jjj \in \Sigma(V,x,r)} \|A_\jjj\|^s \\
    &\le CD^s \min_{\kkk \in \Sigma}H(\kkk) \sum_{\jjj \in \Sigma(V,x,r)} \|A_\jjj^\top|V^\bot\|^s \\
    &\le CD^s \sum_{\jjj \in \Sigma(V,x,r)} \|A_\jjj^\top|V^\bot\|^s \HH^s(\proj_{A_\jjj^\top V^\bot}(X)) \\
    &\le CD^s \sum_{\jjj \in \Sigma(V,x,r)} \HH^s(\proj_{V^\bot}(\fii_\jjj(X))) \\
    &= CD^s \HH^s\biggl(\bigcup_{\jjj \in \Sigma(V,x,r)}\proj_{V^\bot}(\fii_\jjj(X))\biggr) \\
    &\le CD^s \HH^s(\proj_{V^\bot}(X) \cap B(\proj_{V^\bot}(x),2r)) \le C(4D)^sr^s
    \end{split}
  \end{equation}
  for all $x \in X$ and $0<r<\diam(X)$. Since $\min_{\kkk \in \Sigma}H(\kkk)>0$, we have shown the measure $\HH^s|_{\proj_{V^\bot}(X)}$ Ahlfors $s$-regular.

  Let us next assume that $\proj_{V^\bot}(X)$ is Ahlfors $s$-regular for all $V \in X_F$ and show that $\min_{\iii \in \Sigma}H(\iii)>0$. Recall that, by Lemma \ref{thm:ahlfors-implications}, $\HH^s(\proj_{\Pi(\iii)^\bot}(X))>0$ for all $\iii \in \Sigma$. If $\min_{\iii \in \Sigma}H(\iii)=0$, then, by Lemma \ref{prop:important}, we have $\max_{\iii \in \Sigma}H(\iii)=0$. Recalling \eqref{eq:kaenmaki-ahlfors1}, we thus have $\mathcal{H}^s(\proj_{\Pi(\iii)^\perp}(X))=H(\iii)=0$ for all $\iii\in\Sigma$ which is a contradiction.

  Finally, let us prove that $\min_{\iii \in \Sigma}H(\iii)>0$ implies
  \begin{equation*}
    \sup\{\#\Sigma(V,x,r) : V \in X_F, \; x \in X, \text{ and } r>0\} < \infty.
  \end{equation*}
  Observe that
  \begin{align*}
    r &\le \diam(\proj_{V^\perp}(\fii_{\jjj^-}(X))) \\ 
    &= \|A_{\jjj^-}^\top|V^\perp\|\diam(\proj_{A_{\jjj^-}^\top V^\perp}(X)) \le \|A_\jjj^\top|V^\perp\|\frac{\diam(X)}{\min_{j \in \{1,\ldots,N\}}\alpha_2(A_j)}
  \end{align*}
  for all $\jjj \in \Sigma(V,x,r)$ and so, by \eqref{eq:extra}, we have
  \begin{equation*}
    \#\Sigma(V,x,r)r^s\biggl(\frac{\min_{j \in \{1,\ldots,N\}}\alpha_2(A_j)}{\diam(X)}\biggr)^s\leq\sum_{\jjj \in \Sigma(V,x,r)} \|A_\jjj^\top|V^\bot\|^s\leq \frac{4^s}{ \min_{\kkk \in \Sigma}H(\kkk)}r^s.
  \end{equation*}
  Therefore,
  \begin{equation*}
    \#\Sigma(V,x,r) \le \frac{4^s\diam(X)^s}{\min_{j \in \{1,\ldots,N\}}\alpha_2(A_j)^s\min_{\kkk \in \Sigma}H(\kkk)}
  \end{equation*}
  for all $V \in X_F$, $x\in X$ and $r>0$, finishing the proof.
\end{proof}

\section{Assouad and affinity dimensions of self-affine sets} \label{sec:assouad-affinity}

In this section, we prove Theorem \ref{thm:baire-typical}. We first establish a transversality lemma that gives local parameter perturbations with controlled projected distances, and then combine this with a Baire-category argument to obtain a residual failure of projective separation. The lemma below treats the transversality step for a fixed affine iterated function system $\Phi_{\mathsf{v}}$ and a fixed direction $w \in S^1$. Recall that $\theta \otimes w = (\theta_1 w,\ldots,\theta_N w) \in (\R^2)^N$ is the Kronecker product for $\theta = (\theta_1,\ldots,\theta_N) \in \R^N$.

\begin{lemma}\label{lem:trans}
  If $\A = (A_1,\ldots,A_N) \in \GL_2(\R)^N$ is strictly affine such that $\max_{\{1,\ldots,N\}} \|A_i\| < \frac12$, $\mathsf{v} = (v_1,\ldots,v_N) \in (\R^2)^N$, and $w \in S^1$, then there exists $\delta > 0$ such that
  \begin{equation*}
    \frac{\mathrm{d}}{\mathrm{d}\tau_{\iii|_1}} |\proj_{\linspan(w)^\bot}(\pi_{\mathsf{v}+\tau \otimes w}(\iii)) - \proj_{\linspan(w)^\bot}(\pi_{\mathsf{v}+\tau \otimes w}(\jjj))| \Big|_{\tau=\theta} > \delta
  \end{equation*}
  for all $\theta = (\theta_1,\ldots,\theta_N) \in \R^N$ and $\iii, \jjj \in \Sigma$ with $\iii|_1 \ne \jjj|_1$.
\end{lemma}

\begin{proof}
	Write $W = \linspan(w) \in \RP$ and $\iii = i_1i_2\cdots$ and $\jjj = j_1j_2\cdots$. It is easy to see that
	\begin{align*}
    \proj_{W^\perp}(\pi_{\mathsf{v}+\theta \otimes w}(\iii)) &= \proj_{W^\perp}\biggl(\sum_{n=1}^{\infty} A_{\iii|_{n-1}}(v_{i_n} + \theta_{i_n}w)\biggr) \\
	  &= \theta_{i_1}+\proj_{W^\perp}(v_{i_1}) + \sum_{n=1}^{\infty}\|A_{\iii|_{n-1}}^\top|W^\perp\|\proj_{A_{\iii|_{n-1}}^\top W^\perp}(v_{i_n}+\theta_{i_n}w)
	\end{align*}
	Hence,
  \begin{equation*}
    \frac{\mathrm{d}}{\mathrm{d}\theta_{j}} \proj_{W^\perp}(\pi_{\mathsf{v}+\theta \otimes w}(\iii)) = \delta_{i_1,j} + \sum_{n=1}^{\infty}\|A_{\iii|_{n-1}}^\top |W^\perp\|\proj_{A_{\iii|_{n-1}}^\top W^\perp}(w)\delta_{i_n,j}
  \end{equation*}
  for all $j \in \{1,\ldots,N\}$, where $\delta_{i,j} = 1$, if $i=j$, and $\delta_{i,j} = 0$, if $i \ne j$. Since $\max_{\{1,\ldots,N\}} \|A_i\| < \frac12$, it follows that
  \begin{align*}
    \frac{\mathrm{d}}{\mathrm{d}\tau_{i_1}} |\proj_{W^\bot}&(\pi_{\mathsf{v}+\tau \otimes w}(\iii)) - \proj_{W^\bot}(\pi_{\mathsf{v}+\tau \otimes w}(\jjj))| \Big|_{\tau=\theta} \\
    &= \biggl|\sum_{n=1}^{\infty}\|A_{\iii|_{n-1}}^\top |W^\perp\|\proj_{A_{\iii|_{n-1}}^\top W^\perp}(w)\delta_{i_n,j} \\
    &\qquad\qquad\qquad\qquad- \sum_{n=1}^{\infty}\|A_{\jjj|_{n-1}}^\top |W^\perp\|\proj_{A_{\jjj|_{n-1}}^\top W^\perp}(w)\delta_{i_n,j}\biggr| \\
    &\ge 1 - \sum_{n=1}^{\infty}\|A_{\iii|_{n-1}}^\top\| - \sum_{n=1}^{\infty}\|A_{\jjj|_{n-1}}^\top\| \ge \frac{1-2\max_i\|A_i\|}{1-\max_i\|A_i\|}>0
  \end{align*}
  as claimed.
\end{proof}

We are now ready to prove Theorem \ref{thm:baire-typical}.

\begin{proof}[Proof of Theorem \ref{thm:baire-typical}]
  To simplify notation, we assume that $X \subset B(0,1)$. Let $\NN(\A)$ be as in \eqref{eq:baire-N-definition}. For each $\mathsf{v} = (v_1,\ldots,v_N) \in (\R^2)^N$ we write $\fii_i^{\mathsf{v}} = A_i + v_i$ and $\fii_\iii^{\mathsf{v}} = \fii_{i_1}^{\mathsf{v}} \circ \cdots \circ \fii_{i_n}^{\mathsf{v}}$ for all $\iii = i_1\cdots i_n \in \Sigma_n$ and $n \in \N$. It is easy to see that, for every $q \in \N$, $W \in X_F$, and $\iii,\jjj\in\Sigma_*$ with $\iii|_1 \ne \jjj|_1$, the set
  \begin{equation*}
    \{\mathsf{v} \in \NN(\A) : \max_{x\in B(0,1)}|\proj_{W^\perp}(\varphi_\iii^{\mathsf{v}}(x)-\varphi_\jjj^{\mathsf{v}}(x))| < 5q^{-1}\|A_\iii^\top |W^\perp\|\}
  \end{equation*}
  is an open subset of $\NN(\A)$. Therefore, the set
  \begin{equation*}
    \RR_q(\A) = \bigcup_{W \in X_F} \bigcup_{\atop{\iii,\jjj \in \Sigma_*}{\iii|_1 \ne \jjj|_1}} \{\mathsf{v} \in \NN(\A) : \max_{x\in B(0,1)}|\proj_{W^\perp}(\varphi_\iii^{\mathsf{v}}(x)-\varphi_\jjj^{\mathsf{v}}(x))| < 5q^{-1}\|A_\iii^\top |W^\perp\|\}
  \end{equation*}
  is open and, recalling the definition of the projective separation \eqref{eq:POSC}, it is enough to show that $\RR_q(\A)$ is dense in $\NN(\A)$ for all $q \in \N$ since then the set $\RR(\A) = \bigcap_{q \in \N} \RR_q(\A)$ is residual.

  Fix $\mathsf{v} \in \NN(\A)$ and $\eps>0$. By the definition of $\NN(\A)$, there are $W \in X_F$ and $\hhh,\lll \in \Sigma$ with $\hhh|_1 \ne \lll|_1$ such that $\proj_{W^\bot}(\pi_{\mathsf{v}}(\hhh)) = \proj_{W^\bot}(\pi_{\mathsf{v}}(\lll))$. We may thus choose $n \in \N$ large enough so that $d_H(\proj_{W^\bot}(\varphi_{\hhh|_n}^{\mathsf{v}}(X)),\proj_{W^\bot}(\varphi_{\lll|_n}^{\mathsf{v}}(X)))<\varepsilon$, where $d_H$ denotes the Hausdorff distance. Note that there is a constant $c>0$ such that $\min\{|x-y|,|x+y|\} \le c\sphericalangle(\linspan(x),\linspan(y))$ for all $x,y \in S^1$. Following \eqref{eq:key-est-tmp} and recalling the definition of domination, we see that there are constants $D,C \ge 1$ and $0<\tau<1$ such that
  \begin{equation} \label{eq:boundonproj}
  \begin{split}
    \max_{x\in B(0,1)}|\proj_{ A^\top_{\overleftarrow{\kkk}} W^\perp}(x) &- \proj_{A^\top_{\overleftarrow{\kkk}}V^\perp}(x)| \le \|\proj_{A^\top_{\overleftarrow{\kkk}} W^\perp}-\proj_{A^\top_{\overleftarrow{\kkk}} V^\perp}\| \\
    &\le c\sphericalangle(A^\top_{\overleftarrow{\kkk}} W^\perp,A^\top_{\overleftarrow{\kkk}} V^\perp) \le cD\frac{\alpha_2(A_{\overleftarrow{\kkk}})}{\alpha_1(A_{\overleftarrow{\kkk}})} \le cDC\tau^{|\kkk|}.
  \end{split}
  \end{equation}
  for all $\kkk \in \Sigma_*$ and $V \in X_F$. Recalling \eqref{eq:furstenberg-directions-dominated2}, let $\fff \in \Sigma$ be such that $W = \Pi(\fff)$. Fix $q \in \N$, choose $m \in \N$ large enough so that $\max\{1+cDC\tau^m,(1+cD^2C\tau^m)/(1-cD^2C\tau^m)\} < 1+q^{-1}$, and write $\kkk = \overleftarrow{\fff|_m}$. Observe also that
  \begin{equation*}
    |\|A_\ggg|W\|-\|A_\ggg|V\|| \le c\|A_\ggg\|\sphericalangle(V,W)
  \end{equation*}
  and hence, by Lemma \ref{thm:bochi-morris},
  \begin{equation}\label{eq:boundonsecond}
    \biggl|\frac{\|A_\ggg|W\|}{\|A_\ggg|V\|}-1\biggr| \le \frac{c\|A_\ggg\|\sphericalangle(V,W)}{\|A_{\ggg}|V\|} \leq cD\sphericalangle(V,W).
  \end{equation}
  for all $\ggg \in \Sigma_*$ and $V \in X_F$. Writing $\iii = \hhh|_n\kkk\lll|_n\kkk$ and $\jjj = \lll|_n\kkk\hhh|_n\kkk$, let $y_\iii$ and $y_\jjj$ be the fixed points of $\varphi_\iii^{\mathsf{v}}$ and $\varphi_\jjj^{\mathsf{v}}$, respectively, and $v_{\iii}=(\mathrm{Id}-A_{\iii})y_{\iii}$ and $v_\jjj=(\mathrm{Id}-A_\jjj)y_\jjj$ be the translation vectors of $\varphi_\iii^{\mathsf{v}}$ and $\varphi_\jjj^{\mathsf{v}}$, respectively. Simple algebraic manipulations show that
  \begin{equation}\label{eq:boundthird}
  \begin{split}
    \proj_{W^\perp}(\varphi_\iii^{\mathsf{w}}(x)-\varphi_\jjj^{\mathsf{w}}(x)) &= \|A_{\iii}^\top |W^\perp\|(\proj_{A_{\iii}^\top W^\perp}(x)-\proj_{A_{\jjj}^\top W^\perp}(x)) \\
    &\qquad+\|A_{\iii}^\top|W^\perp\|\biggl(1-\frac{\|A_{\jjj}^\top|W^\perp\|}{\|A_{\iii}^\top|W^\perp\|}\biggr)\proj_{A_{\jjj}^\top W^\perp}(x)\\
    &\qquad+\proj_{W^\perp}(w_{\iii})-\proj_{W^\perp}(w_\jjj)
  \end{split}
  \end{equation}
  for all $\mathsf{w} = (w_1,\ldots,w_N) \in (\R^2)^N$ and $x \in B(0,1)$, where $w_\iii = \sum_{k=1}^{|\iii|} A_{\iii|_{k-1}}w_{i_k}$ and $w_\jjj = \sum_{k=1}^{|\jjj|} A_{\jjj|_{k-1}}w_{j_k}$ are the translation vectors of $\varphi_\iii^{\mathsf{w}}$ and $\varphi_\jjj^{\mathsf{w}}$, respectively. Furthermore, by using \eqref{eq:boundonsecond}, \eqref{eq:boundonproj}, and recalling the choice of $m$,
  \begin{equation} \label{eq:one-more-bound}
  \begin{split}
    \frac{\|A_{\jjj}^\top|W^\perp\|}{\|A_{\iii}^\top|W^\perp\|} &= \frac{\|A_{\kkk}^\top A_{\hhh|_n}^\top |A_{\kkk}^\top A_{\lll|_n}^\top W^\perp\|\|A_{\kkk}^\top A_{\lll|_n}^\top|W^\perp\|}{\|A_{\kkk}^\top A_{\lll|_n}^\top |A_{\kkk}^\top A_{\hhh|_n}^\top W^\perp\|\|A_{\kkk}^\top A_{\hhh|_n}^\top |W^\perp\|} \\
    &\le \frac{1+cD\sphericalangle(W^\perp,A_{\kkk}^\top A_{\lll|_n}^\top W^\perp)}{1-cD\sphericalangle(W^\perp,A_{\kkk}^\top A_{\hhh|_n}^\top W^\perp)} \le \frac{1+cD^2C\tau^m}{1-cD^2C\tau^m} < 1+q^{-1}.
  \end{split}
  \end{equation}
  Let $\delta>0$ be as in Lemma \ref{lem:trans}. By Lemma \ref{lem:trans}, there exists $\mathsf{w} = (w_1,\ldots,w_N) \in (\R^2)^N$ in the $\eps\delta^{-1}$-neighborhood of $\mathsf{v} \in \NN(\A)$ such that
  \begin{equation*}
    \proj_{W^\bot}(z_\iii) = \proj_{W^\bot}(z_\jjj),
  \end{equation*}
  where $z_\iii$ and $z_\jjj$ are the fixed points of $\varphi_\iii^{\mathsf{w}}$ and $\varphi_\jjj^{\mathsf{w}}$, respectively, and $w_{\iii}=(\mathrm{Id}-A_{\iii})z_{\iii}$ and $w_\jjj=(\mathrm{Id}-A_\jjj)z_\jjj$ are the translation vectors of $\varphi_\iii^{\mathsf{w}}$ and $\varphi_\jjj^{\mathsf{w}}$, respectively. Hence, by \eqref{eq:boundonproj} and \eqref{eq:one-more-bound},
  \begin{equation} \label{eq:yet-another-bound}
  \begin{split}
    |\proj_{W^\perp}(w_\iii)&-\proj_{W^\perp}(w_\jjj)| \\
    &= |\|A_{\iii}^\top |W^\perp\|\proj_{A_{\iii}^\top W^\perp}(z_\iii)-\|A_{\jjj}^\top|W^\perp\|\proj_{A_{\jjj}^\top W^\perp}(z_\jjj)| \\
    &\le (\|A_{\iii}^\top|W^\perp\|+\|A_{\jjj}^\top|W^\perp\|)cDC\tau^m \\
    &\,\qquad\qquad\qquad+ |\|A_{\iii}^\top|W^\perp\|\proj_{W^\perp}(z_\iii) - \|A_{\jjj}^\top|W^\perp\|\proj_{W^\perp}(z_\jjj)| \\
    &\le \|A_{\iii}^\top|W^\perp\|\biggl(\biggl(1+\frac{\|A_{\jjj}^\top|W^\perp\|}{\|A_{\iii}^\top|W^\perp\|}\biggr)cDC\tau^m + \biggl|1-\frac{\|A_{\jjj}^\top|W^\perp\|}{\|A_{\iii}^\top|W^\perp\|}\biggr|\biggr)\\
    &\le 3q^{-1}\|A_{\iii}^\top|W^\perp\|.
  \end{split}
  \end{equation}
  Thus, by (\ref{eq:boundonproj}), (\ref{eq:boundthird}), and \eqref{eq:yet-another-bound}, we see that
  \begin{equation*}
    \max_{x\in B(0,1)}|\proj_{W^\perp}(\varphi_\iii^{\mathsf{w}}(x)-\varphi_\jjj^{\mathsf{w}}(x))| < 5q^{-1}\|A_{\iii}^\top|W^\perp\|
  \end{equation*}
  and consequently, $\mathsf{w} \in \RR_q(\A)$.
\end{proof}

\begin{ack}
  Bal\'azs B\'ar\'any acknowledges support from the grants NKFI K123782, NKFI FK134251, NKFI K142169, and the grant NKFI KKP144059 ``Fractal geometry and applications''. Han Yu was financially supported by Leverhulme Trust (ECF-2023-186), University of Warwick, University of Cambridge, and the Corpus Christi College, Cambridge. Han Yu has also received funding from the European Research Council (ERC) under the European Union's Horizon 2020 research and innovation programme (grant agreement No.\ 803711). 
\end{ack}

\bibliographystyle{abbrv}
\bibliography{Bibliography}

\end{document}